\documentclass[11pt]{amsart}

\usepackage[french,english]{babel}
\usepackage{amsmath,amssymb,amsthm}
\usepackage{enumitem}
\usepackage{geometry}
\usepackage{xspace}
\usepackage[all]{xy}
\usepackage{xcolor}
\usepackage[colorlinks,urlcolor=blue,linkcolor=blue,citecolor=blue]{hyperref}
\usepackage{a4wide}
\usepackage{cleveref}
\usepackage{mathtools}

\usepackage{booktabs}
\usepackage{stmaryrd}
\usepackage{longtable} 
\usepackage{pdflscape} 
\usepackage{array}

\numberwithin{equation}{section}

\newcommand{\N}{{\mathbb N}}
\newcommand{\Z}{{\mathbb Z}}

\newcommand{\Q}{{\mathbb Q}}
\newcommand{\R}{{\mathbb R}}

\newcommand{\U}{{\mathcal U}}

\newcommand{\GL}{\operatorname{GL}}

\newcommand{\Aut}{\operatorname{Aut}}

\newcommand{\rank}{\operatorname{rank}\xspace}
\newcommand{\I}{\mathcal{I}}
\newcommand{\V}{\mathrm{V}}
\newcommand{\ZZ}{\mathcal{Z}}
\newcommand{\Ma}{\operatorname{M}}
\renewcommand{\O}{\mathcal{O}}
\newcommand{\LL}{\mathcal{L}}
\newcommand{\E}{\operatorname{E}\xspace}

\newcommand{\GE}{\operatorname{GE}\xspace}
\newcommand{\D}{\operatorname{D}}
\newcommand{\sm}[1]{\left(\begin{smallmatrix} #1 \end{smallmatrix}\right)}
\newcommand{\SL}{\operatorname{SL}}
\newcommand{\St}{\operatorname{St}}
\newcommand{\B}{\operatorname{B}\xspace}
\newcommand{\BE}{\operatorname{BE}\xspace}
\newcommand{\DE}{\operatorname{DE}\xspace}

\newcommand{\Tr}{\operatorname{Tr}}
\newcommand{\qa}[3]{\left(\frac{#1, #2}{#3}\right)}
\newcommand{\FA}[1][]{\ifx #1\textup{FA}\xspace \else
  \textup{FA$_{#1}$}\xspace
  \fi
}
\newcommand{\HFA}[1][]{\ifx #1\textup{HFA}\xspace \else
  \textup{HFA$_{#1}$}\xspace
  \fi
}
\newcommand{\FR}{\textup{F$\mathbb{R}$}\xspace}

\newcommand{\HFR}{\textup{HF$\mathbb{R}$}\xspace}

\newcommand{\FAb}{\textup{FAb}\xspace}
\newcommand{\T}{\textup{(T)}\xspace}

\newtheorem{lemma}{Lemma}[section]
\newtheorem{proposition}[lemma]{Proposition}
\newtheorem{theorem}[lemma]{Theorem}

\newtheorem{corollary}[lemma]{Corollary}

\newtheorem{maintheorem}{Theorem}

\theoremstyle{definition}

\newtheorem{definition}[lemma]{Definition}

\newtheorem{question}[lemma]{Question}
\newtheorem*{remark*}{Remark}
\newtheorem*{problem*}{Problem}
\newtheorem{remark}[lemma]{Remark}
\newtheorem*{notation*}{Notation}
\newtheorem*{example*}{Example}

\crefname{question}{question}{questions}

\title[Abelianization and fixed point properties]{Abelianization and fixed point properties of units in integral group rings}

\author[A.~B\"achle]{Andreas B\"achle}
\author[G.~Janssens]{Geoffrey Janssens}
\author[E.~Jespers]{Eric Jespers}
\author[A.~Kiefer]{Ann Kiefer}
\author[D.~Temmerman]{Doryan Temmerman}

\address{Vakgroep Wiskunde en Data Science, Vrije Universiteit Brussel, Pleinlaan 2, 1050 Brussels, Belgium}
\email{\href{mailto:Andreas.Bachle@vub.be}{Andreas.Bachle@vub.be}, \href{mailto:Geoffrey.Janssens@vub.be}{Geoffrey.Janssens@vub.be}, \href{mailto:Eric.Jespers@vub.be}{Eric.Jespers@vub.be}, \href{mailto:Ann.Kiefer@vub.be}{Ann.Kiefer@vub.be}, \href{mailto:Doryan.Temmerman@vub.be}{Doryan.Temmerman@vub.be}}

\thanks{The first and second author are grateful to Fonds Wetenschappelijk Onderzoek - Vlaanderen for financial support. The third, fourth and fifth author are grateful to Onderzoeksraad VUB and Fonds Wetenschappelijk Onderzoek - Vlaanderen for financial support.}

\begin{document}

\begin{abstract}
Let $G$ be a finite group and  $\mathcal{U} (\mathbb{Z} G)$ the unit group  of the integral group ring $\Z G$. We prove a unit theorem, namely a characterization of when  $\mathcal{U}(\mathbb{Z}G)$  satisfies Kazhdan's property $(\operatorname{T})$, both in terms of the finite group $G$ and in terms of the simple components of the semisimple algebra $\mathbb{Q}G$.
Furthermore, it is shown that for $\mathcal{U}( \mathbb{Z} G)$ this property is equivalent to the weaker property $\operatorname{FAb}$ (i.e. every subgroup of finite index has finite abelianization), and in particular also to a hereditary version of  Serre's property $\operatorname{FA}$, denoted $\operatorname{HFA}$. More precisely, it is described when all subgroups of finite index in $\mathcal{U} (\mathbb{Z} G)$ have both finite abelianization and are not a non-trivial amalgamated product.

A crucial step for this is a reduction to arithmetic groups $\operatorname{SL}_n(\mathcal{O})$, where $\mathcal{O}$ is an order in a finite dimensional semisimple $\mathbb{Q}$-algebra $D$, and finite groups $G$ which have the so-called cut property.
For such groups $G$ we describe the simple epimorphic images of $\mathbb{Q} G$.
The proof  of the unit theorem fundamentally relies on  fixed point properties and the abelianization of the elementary subgroups $\operatorname{E}_n(D)$ of $\operatorname{SL}_n(D)$.
These groups are well understood  except in the degenerate case of lower rank, i.e.\ for $\operatorname{SL}_2(\mathcal{O})$ with $\mathcal{O}$ an order in a division algebra $D$ with a finite number of units.
In this setting we determine Serre's property \FA for $\operatorname{E}_2(\mathcal{O})$ and its subgroups of finite index. We construct a generic and computable exact sequence describing its abelianization, affording a closed formula for its $\mathbb{Z}$-rank.
\end{abstract}

\maketitle

\section{Introduction}

One of the most natural and important questions in (integral) representation theory is whether a finite group  $G$ is determined by its integral group ring $\Z G$ (the so called Isomorphism problem, in short (ISO)).
Posed for the first time by Higman \cite{Higman} in 1940, popularized by Brauer \cite{BrauerSurvey} in 1963, it was only in the 1980's that firm indications for a positive solution were obtained.
Indeed, these years saw a number of major breakthroughs, starting with Roggenkamp-Scott \cite{RoggScott} who obtained an affirmative solution  to (ISO) for nilpotent groups.
In fact, not only they did prove that $G \cong H$ whenever $\Z G = \Z H$, but also that $G = H^{\alpha}$ for some unit $\alpha \in \U (\Q G)$; hence explaining why the isomorphism occurs. Here $\U (\Q G)$ denotes the unit group of $\Q G$.
In general, this stronger statement is called the second Zassenhaus conjecture (ZC2).
The third and strongest Zassenhaus conjecture (ZC3) asserts that any finite subgroup of the group of units of augmentation one of $\Z G$ should be \emph{rationally conjugated} (that is, conjugated in $\mathcal{U}(\mathbb{Q}G)$) to a subgroup of the basis $G$.
Shortly after Roggenkamp-Scott, Weiss obtained in his landmark papers \cite{WeissPGroup, WeissNilpotent} that nilpotent groups even satisfy (ZC3).
Around the same time, Roggenkamp-Scott \cite{Scott} provided a metabelian counterexample to (ZC2).
It took ten more years until Hertweck constructed, unexpectedly, in \cite{Hertweck} a counterexample to (ISO).
His construction is still the only known type of counterexample and the general philosophy remains that the ring $\Z G$ encodes a lot of information on $G$. 

A remarkable property of integral group rings is the following: if $G$ and $H$ are finite groups, then $\Z G\cong \Z H$ if and only $\U (\Z G)\cong \U (\Z H)$.
Hence (ISO) is equivalent with $G$ being determined by $\U (\Z G)$.
This is one of the reasons why the structure of $\U(\Z G)$ already receives for more than five decades tremendous attention; for an overview on the main advances and open problems we refer to \cite{EricAngel1,EricAngel2,SehgalBook93,SehgalSurvey03}.
Several main research directions emerged:
\begin{enumerate}
\item the search for generic constructions of subgroups of finite index (preferably torsion-free) in $\U(\Z G)$, 
\item the understanding of torsion units in $\Z G$ (for which the Zassenhaus conjectures were a driving force for many years; recently a counter example has been given by Eisele-Margolis \cite{EisMar} for the last of these that remained open, but still many open problems remain on the arithmetic of the torsion structure), and
\item the search for unit theorems, i.e. structure theorems for the unit group $\U (\Z G)$. 
\end{enumerate}
A fairly complete account of the first direction can be found in the recent books \cite{EricAngel1, EricAngel2} and for the second  we refer to the surveys \cite{KimmerleSurvey, MarRioSurvey} and the references therein.

This paper contributes mainly to the third direction listed above.
A very concrete idea of a unit theorem was given by Kleinert \cite{KleinertSurvey} in the context of orders:
\begin{quote}
\textit{a unit theorem for a finite dimensional semisimple rational algebra $A$ consists of the definition, 
in purely group theoretical terms, of a class of groups $C(A)$ such that almost all generic unit groups of $A$ are members of $C(A)$.}
\end{quote} 
Recall that a \emph{generic unit group} of $A$ is a subgroup of finite index in the group of reduced 
norm  $1$ elements of an order in $A$. 
So far, the finite groups $G$ for which a 
unit theorem, in the sense of Kleinert, is known for $\U(\Z G)$ are those for which the class of groups considered 
are either finite groups (Higman), abelian groups (Higman), free groups \cite{JespersFree} or direct products 
of free(-by-free) groups \cite{JesRioCrelle, JPdRRZ}. Remarkably, the latter can also be described in terms of the rational group algebra: every  simple quotient of $\Q G$ is either a field, a totally definite quaternion algebra or a $2$-by-$2$ matrix ring $\Ma_2(K)$, where $K$ is either $\Q$, $\Q(\sqrt{-1})$, $\Q(\sqrt{-2})$ or $\Q(\sqrt{-3})$.
To our knowledge these results cover all the known unit theorems on $\U (\Z G)$. Surprisingly, one obtains a unit theorem when all the non-commutative simple components of $\Q G$ are two-by-two matrices over a field with finitely many units in any order. This is in contrast with the construction of generators of a subgroup of finite index in $\U (\Z G)$, where one has shown that a collection of explicitly constructed units generate a subgroup $B$ of finite index provided
$\Q G$ does not have so called exceptional components, and $G$ does not have non-abelian \emph{fixed point free} (i.e. a group of fixed point free automorphisms of a finite group) images.
These \emph{exceptional components} (cf. \Cref{definitie exceptional component}) are the non-commutative division algebras other than a totally definite quaternion algebra over a number field (of so called \emph{type (I)}) or $\Ma_2(D)$ with $D$ a finite dimensional $\Q$-division algebra having an order $\O$ with $\U(\O)$ finite (of so called \emph{type (II)}).

Very little is known on the structure of $B$. However several authors, including Marciniak,  Sehgal, Salwa, 
Gon\c{c}alves, Passman, del R\'{\i}o   have given explicit constructions of free groups in $B$ (we refer the reader 
to \cite{delRioGonc}). More recently  Gon\c{c}alves and Passman  in \cite{GonPas} and Janssens, Jespers 
and Temmerman in \cite{JanJesTem}   gave explicit constructions (within the group $B$) that generate a free 
product $C_p *  C_p$ of two cyclic groups of order $p$ in $\U (\Z G)$ in all cases where this is possible, i.e.\,provided $G$ has a noncentral element of order $p$. 

In this paper we complement this line of research by determining when the unit group $\U (\Z G)$ or the group $B$ can be decomposed into a non-trivial \emph{amalgamated product} or an \emph{HNN extension} (see \Cref{defAmal,defHNN}) . 
In fact we show when these decompositions do not occur. In case of an HNN extension this is equivalent, by Bass-Serre theory, to $\U(\Z G)^{ab}$ being finite. Here $\U(\Z G)^{ab}$ denotes the \emph{abelianization} $\U(\Z G) / \U(\Z G)'$ of $\U(\Z G)$, where $\U(\Z G)'$ denotes the \emph{commutator subgroup}. In particular, by a classical theorem of Serre \cite{Serre}, the absence of an amalgam and HNN can be rephrased in terms of satisfying the so-called \emph{FA property}. A group is said to have \begin{itemize}
\item \emph{property FA} if every action on a simplicial tree has a global fixed point, 
\item \emph{property F$\R$} if every isometric action on a real tree has a global fixed point.
\end{itemize}
See \Cref{defFA} for more details. Serre proved that a finitely generated group has property \FA exactly when it is neither a HNN extension nor a non-trivial amalgam.

Since unit theorems concern a property on almost all subgroups of finite index, 
we will consider the hereditary properties, denoted $\HFA$ and $\HFR$, and a hereditary finite abelianization property, denoted $\FAb$. One says that a group has \begin{itemize}
\item \emph{HFA} if all its finite index subgroups have property $\FA$,
\item \emph{HF$\R$} if all its finite index subgroups have property $\FR$,
\item \emph{property FAb} if every subgroup of finite index has finite abelianization.
\end{itemize}
It is well-known that properties $\FA$, $\HFA$ and $\FAb$ follow from \emph{Kazhdan's property (T)} \cite[Theorem~2.12.6]{ValetteBook}, .

The main  result of this paper is a   characterization  of when $\U (\Z G)$ satisfies  these hereditary properties.
Surprisingly, all the mentioned hereditary fixed point properties are equivalent and are controlled both in terms of $G$ and in terms of the Wedderburn decomposition of $\Q G$.
Recall that a finite group $G$ is called a \emph{cut group} if and only if $\U(\Z G)$ has only trivial central units, i.e.\, the center of $\U(\Z G)$ is finite.
For example, rational groups are cut. Recently, cut groups gained in interest (see for example \cite{BMJM, BMP17, Trefethen}), but especially the subclass of rational groups has already a long tradition in classical representation theory (for example, see \cite{IN12, Kle84}).

\begin{maintheorem}\label{mainTheoremintro}  (Theorem~\ref{iff HFA}, Corollary~\ref{prop (T) and (HFA) equivalent for ZG} and
Corollary~\ref{odd-order cut})
Let $G$ be a finite group. The following properties are equivalent:
\begin{enumerate}
\item the group $\U(\Z G)$ has property \HFA,
\item the group $\U (\Z G)$ has property \HFR,
\item the group $\U(\Z G)$ has property \T,
\item the group $\U(\Z G)$ has property \FAb,
\item $G$ is cut and $\mathbb{Q}G$ has no exceptional components,
\item $G$ is cut and $G$ does not map onto one of $10$ explicitly described groups.
\end{enumerate}
Moreover, if $\Q G$ does not have an exceptional component of type (II), then any of the above properties also is equivalent with any of the following properties:
\begin{enumerate}[resume]
\item $\mathcal{U}(\mathbb{Z}G)^{ab}$ is finite,
\item $G$ is a cut group.
\end{enumerate}
\end{maintheorem}
It is worth noticing that $\Q G$, for $G$ a group of odd order, does not have exceptional components of type (II).

From this theorem it follows that in particular, if these conditions are satisfied and the group $G$ has no non-abelian homomorphic image which is fixed point free, the group $B$ is not a non-trivial amalgamated product nor an HNN extension (see \Cref{corollarybicyclic}). 

Crucial to prove \Cref{mainTheoremintro} is to reduce the problem to \emph{$\Z$-orders} (also sometimes simply called \emph{orders}, see \Cref{orders and quaternion algebras}) in simple components of the 
semisimple rational algebra $\Q G$. One writes $\Q G =\Ma_{n_{1}}(D_1) \oplus \cdots \oplus \Ma_{{n_k}}(D_{k})$, 
with each 
$D_i$ a division algebra. Let $\O_i$ be an order in $D_i$. Then both $\Z G$ and $\Ma_{n_1}(\O_1 ) 
\oplus \cdots \oplus \Ma_{n_k}(\O_k)$ are orders in $\Q G$. Due to classical results in order theory their respective unit groups are commensurable and 
the hereditary fixed point properties of both orders are related and, as we shall prove, strongly determined by the 
groups $\E_{n_i}(\O_i)$   generated by \emph{elementary matrices} (these are matrices with $1$ on the diagonal and at most one non-zero entry elsewhere). 

The first  part of this paper is therefore dedicated to the 
(hereditary) fixed point properties of the groups $\E_{n_i}(\O_i)$ and some related groups. 

Recall that $\SL_n(D)$, for a division algebra $D$, consists of the matrices of degree $n$ of reduced norm one (see \Cref{orders and quaternion algebras} for more details).
Due to the celebrated works of many (\cite{BakReh, Bass, Tits} amongst others) on the subgroup congruence problem and the seminal work of Margulis \cite{MargulisBook}, $\SL_n(D)$ enjoys a rich theory on subgroup and rigidity results whenever certain geometric-arithmetic invariants, such as the (reductive) rank, are large enough.
More precisely, in these cases (for example when $n\geq 3$), every \emph{arithmetic subgroup} of $\SL_n(D)$ (i.e. every subgroup which is commensurable to $\SL_n(\O)$ for some order $\O$ in $D$), and consequently also $\E_n(\O)$,  has property \T, where $\O$ is any order in $D$.

On the other hand, when $\SL_2(D)$ is of so-called low rank, which amounts to say that $D$ contains an order $\O$ with $\U(\O)$ finite (see \Cref{subsectie higher rank} for more details), the machinery breaks down and the corresponding landscape reshapes. To illustrate this, if $D$ is commutative it was proven \cite{GueHigWei} that $\E_2(\O)$ and all its finitely generated subgroups have the Haagerup property which is a strong form of non-rigidity, hence opposed to property \T. 

The objects $\E_2(\O)$ with $\U(\O)$ finite are the protagonists of the first part of the  paper. 
For a unital ring $R$ we need for our investigations on abelianization to consider the group
$\GE_2(R)$ generated by $\E_2(R)$ and the group of invertible diagonal matrices over $R$. 
We extend Cohn's techniques \cite{Cohn1, Cohn2} to arbitrary finite dimensional division $\Q$-algebras $D$ containing an order $\O$ with finite unit group. In particular, we deal with orders in totally-definite quaternion algebras with center $\mathbb{Q}$.
As a first step we obtain in \Cref{subsect:almost_universalo_quaternion_orders} finite presentations for the groups $\E_2(\O)$ and $\GE_2(\O)$, which allow us to connect $\E_2(\O)^{ab}$ to the arithmetic structure of $\O$. 

\begin{maintheorem}[\Cref{theoremM} and \Cref{abelianization GE2}]\label{mainA}
Let $\O$ be an order in a finite dimensional division $\Q$-algebra with finite unit group. 
There exists an (explicit) additive subgroup $M$ of $(\O, +)$ such that 
$$\E_2(\O)^{ab} \cong (\O/ M, +),$$
and an (explicit) two-sided ideal $N$ of $\O$ such that there is a short exact sequence 
 \[1\ \longrightarrow\ ( \O/N,+)\ \longrightarrow\ \GE_2(\O)^{ab}\ \longrightarrow\ \U(\O)^{ab}\ \longrightarrow\ 1.\]
\end{maintheorem}

More concretely, $N$ is the two-sided ideal generated by the elements $u-1$ with $u \in \U(\O)$. Therefore as a by-product, the exact sequence above yields that $\GE_2(\O)^{ab}$ is finite (\Cref{GE2Oab_finite}) for any order $\O$ as in \Cref{mainA}.
This is in sharp contrast with the elementary group case, as is shown by the following theorem where for a finitely generated abelian group \hspace{-0,01cm}$G,$ 

$$\rank_{\Z} G  := \max \{ n \mid \Z^n \text{ is, up to isomorphism, a subgroup of } G \},$$and $$\operatorname{inv}_{\O} := \max \{ \left| B \cap \U(\O) \right| \mid B \text{ a } \Z\text{-module basis of }\O\}.$$

\begin{maintheorem}[\Cref{finite_abelianization_equiv}] \label{stelling ab van E_2 intro}
Let $\O$ be an order in a finite dimensional division $\Q$-algebra with finite unit group. Then,
$$
\rank_{\Z} \E_2(\O)^{ab} = \rank_{\Z} \O - \operatorname{inv}_{\O}.
$$
\noindent Moreover, the following properties are equivalent:
\begin{enumerate}
\item $\E_2(\O)^{ab}$ is finite,
\item $\O$ has a $\Z$-basis consisting of units of $\O$,
\item $\O$ is isomorphic to a maximal order in $\Q$, $\Q(\sqrt{-1})$, $\Q(\sqrt{-3})$ or in the quaternion algebras $\qa{-1}{-1}{\Q}$,  $\qa{-1}{-3}{\Q}$ or it is the order of \emph{Lipschitz quaternions} $\qa{-1}{-1}{\Z}$.
\end{enumerate}
\end{maintheorem}

A detailed explanation on the notation of quaternion algebras is given in \Cref{orders and quaternion algebras}. When $\U (\O)$ is infinite, the situation is drastically different. Indeed, it is well-known (by Margulis) that then $ \SL_2(\O)$ and $\E_2(\O)$ have property $\FAb$.

As mentioned earlier if $\E_2(\O)$ has property \FA then $\E_2(\O)^{ab}$ is finite. Hence, in this  case, the orders $\O$ that can appear are restricted by \Cref{stelling ab van E_2 intro}.
Investigating further these orders and certain subgroups of finite index in $\E_2(\O)$ we determine precisely when $\E_2(\O)$ has property \FA and \HFA.

\begin{maintheorem}[\Cref{When is E2 FA and some GL2 FA} and \Cref{When is E_2(O) HFA}]\label{TheoremD}
Let $\O$ be an order in a finite dimensional division $\Q$-algebra with finite unit group.
The following properties are equivalent:
\begin{enumerate}
\item $\E_2(\O)$ has property \FA,
\item $\O$ is isomorphic to a maximal order in $\Q(\sqrt{-3}), \qa{-1}{-1}{\Q}$ or in $\qa{-1}{-3}{\Q}$.
\end{enumerate}
Furthermore, $\E_2(\O)$ does not have property \HFA.
\end{maintheorem}
In \Cref{subsectie FR voor Borel}, properties \FA and \HFA are also investigated for the group $\GE_2(\O)$.
For both groups one first needs to understand the respective Borel subgroups. This is done simultaneously in \Cref{subsectie FR voor GRK} by considering a more general type of group.
Interestingly, in case of the Borel subgroup $\B_2(\O)$ of $\GL_2(\O)$ we obtain in \Cref{Borel FA cyclic units} that it has property \FA if and only if $\U (\O)$ is not isomorphic to $C_2$, the cyclic group of order $2$. 

As mentioned earlier, we  do not only consider actions on simplicial trees, but more generally on real trees for property \FR. In all the cases where we obtain property \FA we actually have the stronger property \FR.
In \Cref{sectie hogere graad matrices FA}, we also briefly discuss elementary groups of degree at least $3$.

In comparison with \HFA in \Cref{mainTheoremintro}, studying property \FA for the full unit group $\U(\Z G)$ is even more delicate. In \Cref{FA and cut groups}, we show that if $\U(\Z G)$ has property \FA, then $G$ is a cut group. We then further investigate cut groups in that section and prove in \Cref{no excp 1 by 1 for cut groups} that if $G$ is a cut group, then $\Q G$ has no exceptional components of type (I). Finally in \Cref{possible 2 times 2 components}, we give a complete description of the exceptional components that can appear in the Wedderburn-Artin decomposition of a cut group and we state precisely when such components appear. 

A full characterization of when $\U(\Z G)$ has property \FA has not been obtained but in Section~\ref{sectie FA voor UZG}  necessary conditions will  be proven and open problems will be formulated. 
Earlier, in Section~\ref{sectie HFA}  we propose a trichotomy result (\Cref{Trichotomy conjecture}) about $\U(\Z G)$ having property \HFA, having property \FA but not  \HFA or it having a non-trivial amalgamated decomposition and finite abelianization. So finite abelianization and a decomposition as amalgamated product would hence be inextricable for unit groups in $\Z G$.
We also prove in \Cref {Trichotomy and equivalent problems} that this question is equivalent to two other problems of independent interest.

We point out that in the follow-up paper \cite{amalgamationpaper} we focus, among other things, on obtaining explicit non-trivial amalgamated decompositions for subgroups of finite index in $\U (\Z G)$ provided $\U (\Z G)$ does not satisfy \HFA.
Notably, building on the methods in this paper, a weaker version of the above mentioned trichotomy is proven, namely that $\U(\Z G)$ either has property $\HFA$ or is, up to commensurability, a non-trivial amalgamated product. 

%
%
%
%

\section{Preliminaries}\label[section]{p1:ab_for_linear_groups}

In this first section we review facts that are needed in the sequel of the article. As a rule, in this paper, a ring $R$ is always meant to be unital and associative, but not necessarily commutative. Moreover, we use the notation $\mathcal{Z}(R)$ for its \emph{center}.
For any group $\Gamma$, $\Gamma'$ denotes its \emph{commutator subgroup} and $\Gamma^{ab} = \Gamma / \Gamma'$ its \emph{abelianization}.

\subsection{Trees} 
In this subsection  we recall some background on the geometric concepts used in the paper, see \cite{Chiswell}.
We will be considering two kinds of trees: simplicial trees and real trees.

\begin{definition}
A connected, undirected graph is called a \emph{simplicial tree} (or simply a \emph{tree}) if it contains no cycle graph as a subgraph.

A metric space is called a \emph{real tree} (or \emph{$\R$-tree}) if it is a geodesic space with no subspace isomorphic to $S^1$.
\end{definition}

Recall that a metric space $(X,d)$ is \emph{geodesic} if between every two points $x$ and $y$ there exists a curve of length $d(x,y)$. The length $L(\alpha)$ of a curve $\alpha : [0,1] \rightarrow (X,d)$ is defined as $\sup \sum d(\alpha(t_i), \alpha(t_{i+1}))$, where $0 = t_0 < ... < t_k = 1$ is a partition of $[0,1]$ and the supremum is taken over all possible partitions.

This definition of a real tree is equivalent to saying that it is a connected \emph{$0$-hyperbolic space}, that is to say all triangles are $0$-thin.
For more on trees or the definition of a $0$-hyperbolic space, we refer the reader to \cite{Chiswell}.

An isomorphism $g$ of a simplicial tree is called an \emph{inversion} if there exist two adjacent vertices which are mapped to one-another.
This is equivalent with the fact that $g$ does not have a fixed vertex, but $g^2$ does. An isomorphism with the former property does not exist for $\R$-trees.
Indeed let $T_{\mathbb{R}}$ be a real tree with an isomorphism $g$ and let $x \in T_{\mathbb{R}}$ be a point fixed under $g^2$.
Considering the geodesic between $x$ and $g(x)$, it is easy to see that the midpoint should be a fixed point for $g$.

In this paper we will be interested, for various types of linear groups $\Gamma$, in the existence of global fixed points when we let $\Gamma$ act on trees.

\begin{definition}\label[definition]{defFA}

A group $\Gamma$ is said to have \emph{property \FA} if whenever $\Gamma$ acts on a tree such that no non-identity element acts as inversion, this action has a globally fixed vertex. 

A group $\Gamma$ is said to have \emph{property \FR} if every isometric action on a real tree has a globally fixed point.

A group $\Gamma$ is said to have the \emph{hereditary property \FA (respectively \FR)} if every finite index subgroup of $\Gamma$ has property \FA (respectively \FR). We denote these properties by \HFA and \HFR.
\end{definition}

Property \FA was first introduced by Jean-Pierre Serre and the name \FA comes from the French ``points Fixes sur les Arbres''.

A simplicial tree can be considered as a real tree by its \emph{geometric realization} \cite[Chapter~2, Section~2]{Chiswell}.
In this way, an action on a simplicial tree $T$ induces an action on its geometric realization $T_{\R}$.
If the action was without inversion, then a point $x_{\R}$ of $T_{\R}$ which is fixed under this induced action, has to correspond to a fixed vertex $x$ of $T$.
Thus property \FR implies property \FA. There are however real trees which are not simplicial trees.

In general, \FA is a weaker property than \FR and an example of a group satisfying \FA but not \FR can be found in \cite{Minasyan}.
Our interest in property \FA originates from the structural properties it implies.
In order to be more precise we first recall the definition of an amalgamated product and an HNN extension.
\begin{definition}\label[definition]{defAmal}
Let $G_1$, $G_2$ and $H$ be groups and $f_1: H \rightarrow G_1$ and $f_2: H \rightarrow G_2$ be injective homomorphisms.
Let $N$ be the normal subgroup of the free product $G_1 \ast G_2$ generated by the elements $f_1(h)f_2(h)^{-1}$ with $h \in H$.
Then the \emph{amalgamated product $G_1 \ast_H G_2$} is defined as the quotient
$$(G_1 \ast G_2) / N.$$
\end{definition}

This amalgamated product is said to be trivial if either $f_1$ or $f_2$ is an epimorphism.

\begin{definition}\label[definition]{defHNN}
Let $\Gamma$ be a group with presentation $\langle S \mid R \rangle$, $H_1$ and $H_2$ be two isomorphic subgroups of $\Gamma$ and  $\theta: H_1 \rightarrow H_2$ an isomorphism. Let $t \not \in \Gamma$ be a new element and $\langle t \rangle$ a cyclic group of infinite order. The \textit{HNN extension} of $\Gamma$ relative to $H_1$, $H_2$ and $\theta$ is the group 
$$\langle S, t \mid R, tgt^{-1}=\theta(g), g\in H_1 \rangle.$$
\end{definition}

A group theoretical characterisation of property \FA was obtained by Serre \cite[I.6.1 Theorem 15]{Serre}.

\begin{theorem}[Serre] \label{prop:iff_conditions_FA}
A countable group $\Gamma$ has property \FA if and only if it satisfies the following properties
\begin{itemize}
\item $\Gamma$ has finite abelianization,
\item $\Gamma$ has no non-trivial decomposition as amalgamated product,
\item $\Gamma$ is finitely generated.
\end{itemize}
\end{theorem}

By classical Bass-Serre theory, a finitely generated group is an HNN extension if and only if it has infinite abelianization (see \cite[I.5.1. example 3 and the proof of I.6.1 Theorem 15]{Serre}). Thus we get the following immediate corollary. 

\begin{corollary}\label[corollary]{FAHNN}
A finitely generated group has property \FA if and only if it is neither an HNN extension nor a non-trivial amalgamated product.
\end{corollary}

Unfortunately for property \FR such a group-theoretical description is still an open problem. The following properties are well-known (in case of \FA a proof can be found in \cite[I.6.3 Examples 1-4]{Serre}, and the other cases can be handled in a similar fashion) and will (mostly) be used without further notice.

\begin{proposition}\label[proposition]{FAconserved}
Let $\Gamma$ be a group, $N$ a subgroup of $\Gamma$ and $\mathcal{P}$ a property among \HFR, \HFA, \FR and \FA.
\begin{itemize}
\item If $\Gamma$ is finitely generated and torsion, then $\Gamma$ has property $\mathcal{P}$. In particular, finite groups have property $\mathcal{P}$.
\item If $N$ is normal in $\Gamma$ and both $N$ and $\Gamma/N$ have property $\mathcal{P}$, then so does $\Gamma$.
\item If $N$ is a subgroup of finite index in $\Gamma$ with property $\mathcal{P}$, then $\Gamma$ has property $\mathcal{P}$.
\item If $N$ is normal in $\Gamma$ and $\Gamma$ has property $\mathcal{P}$ then so does $\Gamma/N$.
\end{itemize}
In particular, a finite direct product $\prod\limits_{i=1}^{q} G_i$ has property $\mathcal{P}$ if and only if every $G_i$ has property $\mathcal{P}$.
\end{proposition}

Two subgroups $\Gamma_1, \Gamma_2$ of a group $\Gamma$ are called \emph{commensurable}, if their intersection is of finite index in both $\Gamma_1$ and $\Gamma_2$. From \Cref{FAconserved} it follows that properties \HFA and \HFR are actually properties of the commensurability class of a group, meaning that either all or none of the groups in the class have this property.

\subsection{Orders and quaternion algebras}\label[subsection]{orders and quaternion algebras} 

Let $A$ be a finite-dimensional algebra over $\Q$. Recall that a \emph{$\Z$-order} (or for brevity just \emph{order}) is a subring of $A$ that is finitely generated as a $\Z$-module and contains a $\Q$-basis of $A$. The following property will be primordial and used very regularly in the rest of the paper. For a proof see \cite[Lemma~4.6.9]{EricAngel1}

\begin{proposition}\label[proposition]{order zijn commensurable}
Let $A$ be a finite dimensional semisimple $\Q$-algebra and let $\O$ and $\O'$ be both orders in $A$. Then their unit groups $\U(\O)$ and $\U(\O')$ are commensurable.
\end{proposition}

Let $K$ be a field of characteristic $0$. Recall that for $u, v \in K\setminus \{0\}$ the \emph{quaternion algebra $D =  \qa{u}{v}{K}$} is the central $K$-algebra $D$, i.e. $\mathcal{Z}(D) = K\cdot 1 $, that is a $4$-dimensional $K$-vector space with basis $\{1, i, j, k\}$ and multiplication determined by \[i^2 = u, \quad j^2 = v, \quad ij = k = -ji. \]
Due to following classical theorem of Hasse-Brauer-Noether-Albert, a quaternion algebra is uniquely determined by the places at which it ramifies.
Recall that a field extension $E$ of a field $K$ is said to be a \emph{splitting field} of a central simple $K$-algebra $A$ if $E\otimes_K A \cong \Ma_{n}(E)$, and $E\otimes_K A$ is said to be the \emph{split extension} of $A$.

\begin{theorem}{{\cite[Theorem 32.11]{Reiner}}}
Let $K$ be a number field and $D$ a quaternion algebra over $K$. Define $\mbox{Ram}(D)$ as the set of places $v$ of $K$ such that $D$ is \emph{ramified} at $v$, i.e.\ such that the completion $K_v$ of $K$, with respect to $v$, is not a splitting field of $D$. Then $\mbox{Ram}(D)$ is a finite set with an even number of elements.
Moreover, for any finite set $S$ of places of $K$ such that $|S|$ is even, there is a unique quaternion algebra $D$ with center $K$ such that $\mbox{Ram}(D)= S$.
\end{theorem}

For $K = \Q$, it is well-known that every place corresponds to a prime integer (for a finite place) or $\infty$ (for the unique infinite place).
Thus a quaternion algebra $D$ over $\Q$ is uniquely determined by its \emph{discriminant} $d = \prod\limits_{p \in \mbox{\tiny Ram}(D) \setminus \{\infty \}} p$ which is the product of all finite places at which $D$ is ramified.
For simplicity's sake, we will sometimes denote a quaternion algebra $\qa{u}{v}{\Q}$ with discriminant $d$ and center $\Q$ by $\mathbb{H}_d$, which is well defined by the above.
Later we will frequently encounter the following three quaternion algebras: 
\begin{center}
$\mathbb{H}_2 = \qa{-1}{-1}{\Q}$,\quad $\mathbb{H}_3 = \qa{-1}{-3}{\Q}$ \quad and \quad $\mathbb{H}_5 = \qa{-2}{-5}{\Q}$. 
\end{center}

If $K$ is a totally real number field and $\sigma(u), \sigma(v) < 0$ for every embedding $\sigma \colon K \to \mathbb{R}$, then the quaternion algebra $D = \qa{u}{v}{K}$ is called \emph{totally definite}. The \emph{conjugate} $\bar{x}$ of $x = a_1\cdot 1 + a_2\cdot i + a_3\cdot j + a_4\cdot k \in D$, $a_1, a_2, a_3, a_4 \in K$ is \[\bar{x} = a_1\cdot 1 - a_2\cdot i - a_3\cdot j - a_4\cdot k.\]

We now recall the concept of reduced norm.  Let $A$ be a finite dimensional central simple algebra over a field $K$ of characteristic $0$. Let $E$ be a splitting field of $A$. The \emph{reduced norm} of $a \in A$ is defined as \[\operatorname{RNr}_{A/K}(a) = \det(1_E \otimes_K a). \] 
Note that $\operatorname{RNr}_{A/K}(\cdot)$ is a multiplicative map, $\operatorname{RNr}_{A/K}(A) \subseteq K$ and $\operatorname{RNr}_{A/K}(a)$ does only depend on $K$ and $a \in A$ (and not on the chosen splitting field $E$ and isomorphism $E \otimes_K A \cong \Ma_{n}(E)$), see \cite[page 51]{EricAngel1}. For a subring $R$ of $A$, put
 \[ \SL_1(R) = \{\ a \in \U(R)\ |\ \operatorname{RNr}_{A/K}(a) = 1 \ \}, \]
 which is a (multiplicative) group.
If $A = \Ma_n(A')$ and $R= \Ma_n(R')$ with $A'$ a finite dimensional central simple algebra over $K$ and $R'$ a subring of $A'$, then we also write $\SL_1(A)= \SL_n(A')$ and $\SL_1(R) = \SL_n(R')$.
 
If we write $\O$ for an order in a finite dimensional division $\Q$-algebra $D$ then, as will be explained in further detail and more generally in \Cref{subsectie higher rank}, $\SL_n(D)$ is an algebraic $\Q$-group and $\SL_n(\O)$ an arithmetic subgroup therein.
The properties of $\SL_n(\O)$ strongly depend on whether $\U (\O)$ is finite or not.
If it is infinite, there is a vast literature showing that $\SL_n(\O)$ satisfies strong properties as illustrated in the introduction.
Therefore in this paper we will consider the case where $\U (\O)$ is finite.
Interestingly this is not a condition on $\O$ but rather a condition on $D$ and one can classify the division algebras containing such an order. Due to the importance of the following classical result we recall its proof.

\begin{theorem}[Folklore]\label{When is unit group order finite}
Let $A$ be a finite dimensional simple $\Q$-algebra and $\O$ an order in $A$. Then, $\U(\O)$ is finite if and only if one of the following holds:
\begin{enumerate}
\item $A= \Q(\sqrt{-d})$ with $d \geq 0$ a non-negative integer,
\item $A= \qa{u}{v}{\Q}$ with $u,v < 0$ negative integers.
\end{enumerate}
\end{theorem}
\begin{proof}
We may assume that $A = \Ma_n(D)$ for $D$ a division algebra containing $\mathbb{Q}$ in its center. Due to \Cref{order zijn commensurable} we may assume that the order is of the form $\O = \Ma_n(\O')$, for an order $\O'$ in $D$. Since $\GL_n(\O')$ is infinite for $n \geq 2$, we have $n = 1$, and $A = D$ a division algebra.

If $D$ is commutative, $D$ is a number field and the statement (in both directions) is a direct consequence of Dirichlet's unit theorem  \cite[Corollary~5.2.6]{EricAngel1}.
  
If $D$ is non-commutative and $\U(\O)$ is finite, Kleinert's theorem \cite[Proposition~5.5.6]{EricAngel1} implies that $D$ is a totally definite quaternion algebra. However, $\langle \SL_1 (\O), \mathcal{U}(\mathcal{Z}(\O)) \rangle$ has finite index in $\mathcal{U}(\O)$, see \cite[Proposition~5.5.1]{EricAngel1}. In particular, also the unit group of $\mathcal{Z}(\O)$, which is an order in $\mathcal{Z}(D)$ by \cite[Lemma~4.6.6]{EricAngel1}, must be finite and consequently by the commutative case $\mathcal{Z}(D)= \Q(\sqrt{-d})$, $d \geq 0$. As $D$ is a totally definite quaternion algebra, $\mathcal{Z}(D)$ is a totally real extension of $\Q$. Hence $\mathcal{Z}(D) = \Q$.

Conversely, if $D= \qa{u}{v}{\Q}$ with $u,v < 0$ negative integers, then by the previous $\mathcal{U}(\mathcal{Z}(\O))$ is a finite group and by Kleinert's theorem $\SL_1 (\O)$ is also finite. Hence, the group $\langle \SL_1 (\O), \mathcal{U}(\mathcal{Z}(\O)) \rangle$, which is of finite index in $\mathcal{U}(\O)$, is finite. This in turn implies that $\mathcal{U}(\O)$ is finite.
\end{proof}

\subsection{Linear groups}\label[subsection]{lineargroups}
When studying the groups $\GL_n(R)$ and its subgroups, it sometimes helps to consider the groups $\GE_n(R)$ and $\E_n(R)$.
Here, $\E_n(R)$ is the subgroup of $\GL_n(R)$ generated by the matrices having $1$ on the diagonal and one non-zero entry off the diagonal and $\GE_n(R)$ is the subgroup of $\GL_n(R)$ generated by $\E_n(R)$ and the invertible diagonal matrices.
These groups have been thoroughly studied in the literature, see, for example \cite{Cohn1,Cohn2}. Note that if $R$ is a subring of a division algebra, then $\E_n(R) \leq \SL_n(R)$.

In the case of $n=2$, we will be using a special (but equivalent) set of generators for $\E_2(R)$. 
What follows in this subsection is based on \cite{Cohn1}.
By $I$ we denote the $2 \times 2$ identity matrix.

The group $\GE_2(R)$ is the group generated by all matrices 
\[ [\mu, \nu] = \left( \begin{matrix} \mu & 0 \\ 0 & \nu \end{matrix}\right), \; (\mu, \nu \in \U(R)), \qquad E(x) = \left( \begin{matrix} x & 1 \\ -1 & 0  \end{matrix}\right), \; (x \in R). \]
For $\mu \in \U(R)$, put $D(\mu) = [\mu, \mu^{-1}]$. Define the group $\D_2(R) = \langle [\mu, \nu] \ | \ \mu, \nu \in \U(R) \rangle$. Note that
\[ E(0)^{-1}E(x) = \left(\begin{matrix} 1 & 0 \\ x & 1 \end{matrix} \right), \qquad  E(-x)E(0)^{-1} = \left(\begin{matrix} 1 & x \\ 0 & 1 \end{matrix} \right)  \] 
and 
\[  \left(\begin{matrix} 1 & 1 \\ 0 & 1 \end{matrix} \right) \left(\begin{matrix} 1 & 0 \\ -1 & 1 \end{matrix} \right)  \left(\begin{matrix} 1 & 1 \\ 0 & 1 \end{matrix} \right) = E(0).  \] Consequently, $$\E_2(R) = \langle E(x) \mid x \in R \rangle.$$
A priori, $\GE_2(R) \leq \GL_2(R)$, but it can happen that these groups are equal. In this case, we call the ring $R$ a \emph{$\GE_2$-ring}. The following type of rings will form an important class of examples.

\begin{definition}
Let $R$ be a ring.
A \emph{left Euclidean} map on $R$ is a map $\delta: R \setminus \{0\} \rightarrow \N$ satisfying $$\forall\ a,b \in R \text{ with } b\neq 0, \exists\  q,r \in R: a= qb + r\text{ with } \delta(r) < \delta(b) \text{ or } r = 0.$$ 
A \emph{right Euclidean} map on $R$ is a map $\delta: R \setminus \{0\} \rightarrow \N$ satisfying $$\forall\ a,b \in R \text{ with } b\neq 0, \exists\  q,r \in R: a= bq + r\text{ with } \delta(r) < \delta(b) \text{ or } r = 0.$$
We call the ring $R$ a \emph{left (right) Euclidean} ring if it has a left (right) Euclidean map.
\end{definition}

If $R$ is a subring of $\mathbb{C}$ or a quaternion algebra with totally real center, then $R$ is endowed with an algebraic norm $x \bar{x}$, $x \in R$.
If this map is a left (right) Euclidean map, then we call $R$ \emph{left (right) norm Euclidean}.

We will omit the proof of the following, since it is the same as in the well known commutative case (see for example \cite[Proposition 1.4.1]{Doryanthesis}).
\begin{proposition}\label[proposition]{euclidian are GE2}
Left or right Euclidean rings are $\GE_2$-rings.
\end{proposition}

Let $G$ be a group generated by a set of elements $X$. Then a subset $\mathcal{R}$ of the free group $F_X$ on the elements $X$ is called a \emph{defining set of relations of $G$ with respect to $X$} if the canonical epimorphism $$F_X \rightarrow G, $$ has as kernel the normal closure of the group generated by $\mathcal{R}$, i.e. $\langle \mathcal{R}^{F_X} \rangle$.

If $S \subseteq X$ and $H = \langle S \rangle \leq G$, then any element of $F_S \leq F_X$ is said to be \emph{expressed in abstract letters of $H$}.

In the group $\GE_2(R)$ the following relations hold, see \cite[(2.2)-(2.4)]{Cohn1}.
\begin{align}
E(x)E(0)E(y) & = E(0)^2E(x+y), & & x, y \in R \tag{R1}\label[equation]{R1} \\
E(\mu)E(\mu^{-1})E(\mu) & = E(0)^2D(\mu), & & \mu \in \U(R) \tag{R2}\label[equation]{R2} \\
E(x)[\mu, \nu] & = [\nu, \mu] E(\nu^{-1}x\mu),  & &  x \in R, \; \mu, \nu \in \U(R) \tag{R3}\label[equation]{R3} \\
E(0)^2 & = D(-1). \tag{R4}\label[equation]{R4}
\end{align}
The ring $R$ is called \emph{universal for $\GE_2$} if these relations, together with the relations in the group $\D_2(R)$, form a complete set of defining relations of $\GE_2(R)$ with respect to the elements $E(x)$ and $[\mu, \nu]$ for $x \in R$ and $\mu, \nu \in \U(R)$.
When the group $\GE_2(R)$ is discussed, we will often omit mentioning these explicit generators.
If we talk about a set of defining relations for $\E_2(R)$, the generators are always assumed to be $E(x), x \in R$.
The relations \eqref{R1}-\eqref{R4} together with the relations in the group $\D_2(R)$ are called the \emph{universal relations}.
Clearly $E(0)^2 = D(-1) = -I$.

Equation \eqref{R3} specializes to 
\begin{align}
E(x)D(\mu) & = D(\mu^{-1}) E(\mu x\mu),  \quad  x \in R, \; \mu \in \U(R). \tag{R3'}\label[equation]{R3'} 
\end{align}
The inverse of $E(x)$ is given by the formula \begin{equation}E(x)^{-1} = E(0) E(-x) E(0), \qquad \forall x \in R, \tag{R5}\label[equation]{inv} \end{equation} which follows from \eqref{R1} and $D(-1)^2 = I$.
From the universal relations one can also derive the following useful formulas, see \cite[(2.8), (2.9) and (9.2)]{Cohn1}
\begin{equation}
 E(x)E(y)^{-1}E(z) \; = \; E(x -y+z), \qquad x, y, z\in R, \tag{R6}\label[equation]{2.8}
\end{equation}
\begin{equation}
 E(x)E(\alpha)E(y) \; = \; E(x - \alpha^{-1})D(\alpha)E(y - \alpha^{-1}), \qquad x, y \in R, \alpha \in \U(R), \tag{R7}\label[equation]{2_9}
\end{equation}
\begin{equation}
[u^{-1}v^{-1}uv,1] \; = \; D(u^{-1})D(v^{-1})D(uv), \qquad u,v \in \U(R). \tag{R8}\label[equation]{R5} 
\end{equation}

Rings that are not universal for $\GE_2$ have to have some additional defining relations.
For several results the actual form of these non-universal relations is not of importance, but rather the fact that they can be chosen to have a special form (for example can be expressed in abstract letters of $E_2(R) = \langle E(x) \mid x \in R \rangle$).
Hence we introduce the following class of rings.

\begin{definition}
Let $R$ be a ring for which there exists a set $\Phi$ of words expressed in abstract letters of $\E_2(R)$ such that $\Phi$ together with the universal relations yield a full list of defining relations for $\GE_2(R)$. Then we call $R$ \emph{almost-universal for $\GE_2$}.
\end{definition}

In \Cref{quaternion cohn} we will prove that orders in totally definite quaternion algebras are almost-universal. 
The following is a slight generalization of \cite[Theorem 1]{Cohn2} for $\GE_2(R)$. 

\begin{theorem}\label{generators and relations for E_2(R)}
Let $R$ be a ring, almost-universal for $\GE_2$ with $\Phi$ a set of relations expressed in letters of $\E_2(R)$ such that $\Phi$ together with the universal relations is a complete set of defining relations of $\GE_2(R)$. The group $\E_2(R)$ is generated by the symbols $E(x)$, $x \in R$ and if we define $D(u)$ and $[w,1]$ for $u \in \U(R), w\in \U(R)^{\prime}$ by the relations \eqref{R2} and repeated use of \eqref{R5} then a complete set of defining relations for $\E_2(R)$ is given by \begin{align}
E(x)E(0)E(y) &= E(0)^2E(x+y) \tag{R1} \\
E(x)D(u)&=D(u^{-1})E(uxu) \tag{R3'} \\
E(0)^2 &= D(-1) \tag{R4}\\
[w_1,1]\ldots[w_n,1] &= I \text{ where } w_j \in \U(R)^{\prime} \text{ and } w_1\ldots w_n = 1. \tag{R9}\label[equation]{R6} \\
f &= I \mbox{ for all } f \in \Phi \tag{R10}\label[equation]{R7} 
\end{align}
\end{theorem}
\begin{proof}
From the given relations, it is clear we can still deduce \eqref{2_9}, $D(u^{-1}) = D(u)^{-1}$ and \eqref{inv}.
Thus, using these relations, we can rewrite any relation $w = I$ in $\E_2(R)$ as 
\begin{equation}\label[equation]{relationw'}
w' = D(u_1)\ldots D(u_k)E(a_1)\ldots E(a_r) = I.
\end{equation}
We will show that we may reduce the latter relation to a relation with $r = 0$. Note that $r = 1$ cannot occur as $E(a_1)$ is not a diagonal matrix, $r = 2$ is only possible if $a_1 = a_2 = 0$ and this case can be treated with \eqref{R4}. So assume $r \geq 3$. From the universal relations and \eqref{2_9}, the relation \eqref{relationw'} may always be written in such a form that  $a_i \notin \U(R) \cup \{0\}$ if $1 < i < r$ and $a_1 \neq 0$.

Remark that the universal relations for $\GE_2$ are equivalent to the relations \eqref{R1}, \eqref{R3'}, \eqref{R4}, \begin{equation}\label[equation]{relator}
E(x)[u,1] = [1,u]E(xu),
\end{equation}
and those in $\DE_2(R)$.
Since $R$ is almost-universal for $\GE_2$, and $w'$ is also a word in $\GE_2(R)$, it is a product of conjugates of relators \eqref{R1}, \eqref{R3'}, \eqref{R4}, \eqref{relator} and \eqref{R7}. By \eqref{R3}, conjugates of relators \eqref{R1}, \eqref{R3'}, \eqref{R4}, or \eqref{R7} are words in $\E_2(R)$. In particular they are $1$ in $\E_2(R)$. Hence we can write $w'$ as a product of conjugates of relators of the form $[1,u]E(xu)[u^{-1},1]E(x)^{-1}$  and $D(u)$'s.

Further, the relator $[1,u]E(xu)[u^{-1},1]E(x)^{-1}$ can also be expressed in the generators of $\E_2(R)$ as follows:  \begin{align*}
[1,u]E(xu)[u^{-1},1]E(x)^{-1} &= [1,u]E(xu)[1,u^{-1}][1,u][u^{-1},1]E(x)^{-1} \\&= [1,u][u^{-1},1]E(uxu)D(u^{-1})E(x)^{-1} \\ &=D(u^{-1})E(uxu)D(u^{-1})E(x)^{-1} \\&= D(u^{-1})D(u)E(x)E(x)^{-1}.
\end{align*}
This last word is trivial in $\E_2(R)$, so the word $w'$ reduces to the form $$D(v_1)\ldots D(v_l) = I. $$ Note that in the latter form $r=0$.
Moreover, again as $R$ is almost-universal for $\GE_2$, by the relations in the group $\D_2(R)$, we have that $v_1 \ldots v_l=1$.
By \eqref{R5}, $D(u)D(v)= [uvu^{-1}v^{-1}, 1] D(u^{-1}v^{-1})^{-1}$.
Now by repeated use of the latter, $w$ can be further rewritten as 
$$[v_1v_2v_1^{-1}v_2^{-1},1][v_2v_1v_3v_1^{-1}v_2^{-1}v_3^{-1},1] \ldots [v_{l-1}\ldots v_1v_lv_1^{-1} \ldots v_l^{-1},1]D(v_lv_{l-1} \ldots v_1)^{-1} = I.$$
By the above, this is equivalent with
$[w_1,1]\ldots[w_t,1] = I$ with $w_j \in \U(R)^{\prime}$ and $w_1 \ldots w_t = 1$, which is exactly \eqref{R6}. 
\end{proof}

\section{Abelianization of $\E_2(\O)$ and $\GE_2(\O)$} \label[section]{sectie ab 2 maal 2}

In this section we will study the abelianization of linear groups of degree $2$ over orders $\O$, with a finite unit group, in a rational division algebra.
To do so we will first prove in Subsection~\ref{subsect:almost_universalo_quaternion_orders} that, similar to the case of rings of integers in number fields, those orders allow an almost universal presentation.
This will enable us, in Subsection~\ref{subsection_GE2}, to show that the abelianization of $\GE_2(\O)$ fits into a short exact sequence described in \Cref{mainA} of the introduction, which will be used to calculate this abelianization. Also in \Cref{mainA}, we announced a short exact sequence that describes the abelianization of $\E_2(\O)$. This result, and an explicit formula for the $\Z$-rank of $\E_2(\O)^{ab}$ (see also \Cref{stelling ab van E_2 intro}), forms the main part of Subsection~\ref{sectie ab van E2}. It will also allow us to characterize when this abelianization is finite.

\subsection{An (almost) universal presentation}\label[subsection]{subsect:almost_universalo_quaternion_orders}

In \cite[Lemma page~160]{Cohn2} an explicit description of the non-universal relations for $\GE_2(R)$, with $R$ a subring of the complex numbers satisfying certain conditions (including rings of algebraic integers in imaginary quadratic extensions of the rationals), is obtained. For our purposes we need a quaternion variant thereof. To achieve this, we give a carefully adapted, more detailed version of the arguments in \cite{Cohn2}.

Let $H = \qa{u}{v}{\mathbb{Q}}$ be a totally definite quaternion algebra with center $\mathbb{Q}$, i.e.\,$u,v$ negative integers. Define $|x|:=  \sqrt{x \overline{x}}$, the positive square root of $x \overline{x}$, for $x \in H$ and recall that $x\bar{x} \in \Z$ for $x$ contained in an order in $H$. We record the following well known properties of this norm map on $H$. For all $x, y \in H$, $\lambda \in \mathbb{Q}$:

\begin{tabular}{rlrl}
 (N1) & $|x| \geq 0$ \quad \text{and} \quad $|x| = 0\ \Leftrightarrow x = 0$ \hspace{1cm}  & (N2) & $|\lambda x| = |\lambda||x|$ \\ (N3) & $|x + y| \leq |x| + |y|$ & (N4) & $|xy| = |x||y|$
\end{tabular}

\begin{proposition}\label[proposition]{quaternion cohn}
Let $K= \Q(\sqrt{-d})$, with $d$ a non-negative integer, i.e. $K$ is either $\mathbb{Q}$ or a quadratic imaginary extension of $\Q$. Let $H = \qa{u}{v}{\Q}$ be a totally definite quaternion algebra with center $\mathbb{Q}$. Let $\O$ be an order in $K$ or $H$. Then a complete set of defining relations for $\GE_2(\O)$ is given by the universal relations together with
\begin{equation}
(E(\overline{a}) E(a))^n = E(0)^{2},\qquad \mbox{ for each } a \in \O \mbox{ such that } 1 < |a| = \sqrt{n} < 2. \label[equation]{relation alpha}
\end{equation}
\end{proposition}

We will only give an explicit proof in the case of a quaternion algebra. The case of a quadratic imaginary extension of $\Q$ is an easy adaptation of this proof.
We first need an auxiliary lemma.
Its proof is straightforward which is why we omit it here. 

\begin{lemma}\label[lemma]{norm}Let $K$ and $H$ be as in \Cref{quaternion cohn} and $z, a \in K$, or $z, a \in H$, $z \not= 0$. Let  $1<|a| = \sqrt{n}$. Then 
\begin{equation} |z-a| < 1 \quad \mbox{ if and only if }\quad \left|z^{-1}-\frac{1}{n-1} \overline{a}\right| < \frac{1}{n-1}.\label[equation]{eq:norm_lemma}\end{equation}
\end{lemma}

It is well known that every relation in $\GE_2(\O)$ can be written in the form $$E(t_1) \ldots E(t_l) = D,$$ with $D \in \D_2(\O)$ (see for example \cite[(2.11)]{Cohn1}).
We will call $l$ the \emph{length of the relation}.
Using the universal relations and \eqref{2_9} (which follows from them), one may always rewrite this relation to a relation where $t_1 \neq 0$ and $t_i \notin \U(\O) \cup \{0\}$ for $1<i<l$. We call such a form a \emph{canonical form}.

For the proof of \Cref{quaternion cohn}, we will introduce the following notation. Starting from a list $t = (t_1, t_2, \ldots , t_l)$ of elements of $\O$, one may obtain a list $b(t) = (b_0, b_1, \ldots, b_l)$ and two non-negative integers $m(t)$ and $h(t)$ as follows:
$$ b_0 = 0, \quad b_1 = 1, \quad b_i = b_{i-1} t_{i-1} - b_{i-2} \text{ when }2 \leq i \leq l,$$
$$m(t) = \max \{|b_0|, \ldots, |b_l|\}, \quad h(t) = \max \{ i \mid |b_i| = m(t) \}.$$
We will simply write $m = m(t)$ and $h = h(t)$ when $t$ is clear from the context, and extend this notation to $m' = m(t'), h' = h(t'), b' = b(t')$ when we use a second list $t' = (t'_1, t'_2, \ldots, t'_{l'})$. 
Remark that $b_i$ is the element in the upper-right corner of the product $$E(t_1) \ldots E(t_i), $$ where $b_0 = 0$ is to be interpreted as the upper right corner of the empty product, namely $I$.
On the set $\R^2$, denote the lexicographical order by $\preceq$. In the proof of \Cref{quaternion cohn}, we will show that one can reduce a relation $E(t_1) \ldots E(t_l) = D$ to a different relation $E(t'_1) \ldots E(t'_{l'}) = D' \in \D_2(\O)$ for which $(m',h') \prec (m,h) $.
For this, we first need the following lemma. 

From now on we are working with $\GE_2(\mathcal{O})$ as an abstract group in terms of the generators $E(x)$ and $[\alpha, \beta]$. In particular $-I$ is no longer intrinsic, however we will continue to use it as a notation for $D(-1)$. 

\begin{lemma}\label[lemma]{inductionhelp}
Let $\O$ be an order in a quaternion algebra.  Every relation $E(t_1) \ldots E(t_l) = D$ in $\GE_2(\O)$ (with $D \in \D_2(\O)$ and $t_1 \neq 0$) can be rewritten, using the universal relations of $\GE_2(\O)$, into a relation $E(t'_1) \ldots E(t'_{l'}) = D'$ in canonical form, such that $(m',h') \preceq (m,h)$.
Moreover, $l' \leq l$.
\end{lemma}

\begin{proof}
If the relation is not yet in canonical form, then for some $1 < i < l$, $t_i \in \U(\O) \cup \{0\}$.

If $t_i = 0$, then one can use \eqref{R1} to replace $E(t_{i-1})E(t_i)E(t_{i+1})$ by  $E(t_{i-1} + t_{i+1})$ and $D$ by $-D$.
This reduces the length of the relation by $2$.
Since $b_{i-1}^{\prime}$ and $b_{i-1}$ only depend on the $t$'s coming before, we have $b_{j}^{\prime}=b_{j}$ for $j \leq i-1$.
Moreover, since $$E(t_1) \ldots E(t_{i-1} + t_{i+1})E(t_{i+2}) = - E(t_1) \ldots E(t_{i-1})E(t_i)E(t_{i+1})E(t_{i+2}),$$ we get that $b_{j}^{\prime} = - b_{j+2}$ for $j \geq i$.
So, from $\{|b'_0|, \ldots, |b'_{t-2}|\} \subseteq \{|b_0|, \ldots, |b_t|\}$ then follows that $m' \leq m$ and when $ m' = m$, then $h' \leq h$.
In other words, $(m',h') \preceq (m,h)$.
 
If $t_i \in \U(\O)$ for some $1 <i<l$, then one can use \eqref{2_9} (which follows from the universal relations) to replace $E(t_{i-1})E(t_i)E(t_{i+1})$ by $E(t_{i-1} - t_i^{-1})D(t_i)E(t_{i+1} - t_i^{-1})$ and then use \eqref{R3'} to move $D(t_i^{\pm 1})$ to the right of the equation. 
This reduces the length of the relation by $1$.
Similar to the above, $b_{j}^{\prime}=b_{j}$ for $j \leq i-1$.
For the cases $j \geq i$, let us first consider $b_{i+1}$, the upper right entry of \begin{align*}
 E(t_1) \ldots E(t_{i-1})E(t_{i})E(t_{i+1}) &=  E(t_1) \ldots E(t_{i-1} - t_i^{-1})D(t_i)E(t_{i+1} - t_i^{-1}), \\  &= E(t_1) \ldots E(t_{i-1} - t_i^{-1})E(t_i^{\prime})D(t_i^{-1}),
 \\  &= E(t'_1) \ldots E(t'_{i-1})E(t_i^{\prime})D(t_i^{-1})
\end{align*} 
Since multiplying by $D(t_i^{-1})$ does not change the modulus of the element in the upper right corner, $|b_{i+1}|$ is equal to the modulus of the element in the upper right corner of $$E(t'_1) \ldots E(t'_{i-1})E(t_i^{\prime}),$$ which is $|b'_i|$.
A similar proof shows that for each $j \geq i$ holds $|b'_j| = |b_{j+1}|$.
This shows that $\{|b'_0|, \ldots, |b'_{t-1}|\} \subseteq \{|b_0|, \ldots, |b_t|\}$, and so that $m' \leq m$ and when $ m' = m$, then $h' \leq h$.
In other words, $(m',h') \preceq (m,h)$.

Notice that in both cases, we have reduced the length of the relation.
We repeat this process until the relation is in canonical form.
\end{proof}

We are now ready to prove \Cref{quaternion cohn}.
\begin{proof}[Proof of \Cref{quaternion cohn}.]
Remark that for $z \in \O$, $ |z| \in \{0,1,\sqrt{2},\sqrt{3}\}$, if $\vert z \vert < 2$. Here we use that $\vert z \vert^2 \in \Z$, as $z$ is an algebraic integer. 

In order to prove the proposition, we begin with a relation 
\begin{equation}\label[equation]{relation*}
E(t_1) \ldots E(t_l) = D 
\end{equation} 
in $\GE_2(\O)$ (with $D$ a diagonal matrix) and will reduce it to a relation implied by the universal relations and \eqref{relation alpha}. A relation of length $0$ is the trivial relation and a simple calculation shows that a relation of length $2$ is impossible, except when $t_1=t_2=0$, but this is the relation \eqref{R4}, i.e. $E(0)^2 = -I$.
A relation of length $1$ does not exist since $b_1$ is always equal to $1$.

Assume $l \geq 3$, i.e. assume a relation of length at least $3$.
Without loss of generality, we may assume that $t_1 \neq 0$. Indeed if $t_1 = 0$ we can conjugate the relation with $E(0)^{-1}$ such that $t_1 \neq 0$.
By \Cref{inductionhelp}, we furthermore may assume that the relation is written in a canonical form. Let now $m=m(\mathfrak{t})$ and $h=h(\mathfrak{t})$ where $\mathfrak{t} = (t_1, \ldots, t_l)$ is the list associated to the relation \eqref{relation*}.

\emph{Strategy of the rest of the proof:} the argument we use is adapted from \cite{Cohn2} and will use induction on $(m,h)$.
We will show that such a canonical relation can be reduced to a relation (not necessarily canonical) for which $(m',h') \prec (m,h)$, by using \eqref{relation alpha}.
Afterwards, we show that \Cref{inductionhelp} is applicable, which does not increase $(m',h')$.
For this new relation in canonical form, either $l' < 3$, which finishes the proof, or $l' \geq 3$ and we may continue by induction.
Since $(m',h')$ takes only discrete values in $\R^2_{\geq 0}$, this shows that at some point the length of the new relation will be less than $3$, which finishes the proof.

For notation's sake, write
$$ a = b_{h}, \quad b = b_{h-1} \quad \textrm{ and } \quad t = t_{h}.$$
Note that $a \neq 0$ since $|b_h| \geq |b_1| = 1$.
As $b_l=0$ and $b_1=1$, we get that $h \neq l$.
Hence $h < l$, which in turn implies that $|t|> 1$, for else $t \in \U(\O) \cup \{0\}$, which contradicts the assumption that the relation is in canonical form. Also $h \neq 1$.
Indeed, for suppose $h = 1$, then $b_2=b_3=\ldots=b_l=0$.
Hence $0=b_2 = t_1$, but this is again a contradiction.

Since $b_{h+1} = at - b$ and by definition of $h$, we have
\begin{equation} \label[equation]{alpha1}
|at-b| < |b_h| = |a|,
\end{equation}
and
\begin{equation} \label[equation]{ba}
|b| \leq |b_h| = |a|.
\end{equation}
Furthermore $b \not= 0$, since otherwise $|a||t| = |b_{h+1}| < |a|$, which would imply $|t| < 1$.

Inequality \eqref{alpha1} is equivalent to 
\begin{equation}\label[equation]{alpha2} |a^{-1}b - t | < 1.
\end{equation} Note that $|t| \geq 2$ implies $|at-b| \geq 2 |a| - |b| \geq |a|$, a contradiction with \eqref{alpha1}. Hence $\vert t \vert < 2$ and thus $|t| \in \{\sqrt{2}, \sqrt{3}\}$. We will handle both cases separately.

First suppose that $|t| = \sqrt{2}$. Applying \Cref{norm} to equation \eqref{alpha2} one obtains $|b^{-1}a - \bar{t}| < 1$ or equivalently \begin{equation}\label[equation]{alphasqrt2}
|a - b\bar{t}| < |b|.
\end{equation}
We rewrite \eqref{relation alpha} to \begin{align*} E(t) &= E(0)^2 E(\bar{t})^{-1}E(t)^{-1}E(\bar{t})^{-1}, \tag{\text{basic cancelation rule of groups applied to \eqref{relation alpha}}} \\
&= E(0)^2 E(0)E(-\bar{t})E(0)^2E(-t)E(0)^2E(-\bar{t})E(0), \tag{using \eqref{inv}}\\
&= E(0)^2 E(0)E(-\bar{t})E(-t)E(-\bar{t})E(0), \tag{centrality of $E(0)^2$ and $E(0)^4 = 1$}\end{align*}which we substitute in the relation \eqref{relation*} to obtain (after using \eqref{R1} two times to get rid of $E(0)$) $$E(t_1)\ldots E(t_{h-2})E(t_{h-1}-\bar{t})E(-t)E(t_{h+1} - \bar{t})E(t_{h+2}) \ldots E(t_l) = D^{\prime}, $$ for some diagonal matrix $D^{\prime}$.
This is the new relation for which we claim $(m',h') \prec (m,h)$ after setting $$t'_i = \begin{cases}
t_i, & \text{if } i < h-1;\\
t_{h-1} - \overline{t}, & \text{if } i = h-1; \\
-t, &\text{if } i = h;\\
t_{h+1} - \overline{t}, &\text{if } i = h+1;\\
t_i, & \text{if } i > h+1.
\end{cases}$$
Then it is easy to see that $b_i^{\prime} = b_i$ for $i < h$ and $b_{h}^{\prime} =a-b\bar{t}$. When $i>h$, one sees that $b_i^{\prime} = -b_i$. Indeed,
$$b'_{h+1} = b'_ht'_h - b'_{h-1} = (b_h-b_{h-1}\overline{t})(-t)-b_{h-1} = -b_ht + b_{h-1} = -b_{h+1}. $$
Furthermore, \begin{align*}
\MoveEqLeft b'_{h+2} = b'_{h+1}t'_{h+1} - b'_h = -b_{h+1}(t_{h+1}- \overline{t})-(b_h-b_{h-1} \overline{t}) \\
&=- b_{h+1} t_{h+1}- b_h+ (b_{h+1}+b_{h-1})\overline{t} = -b_{h+2}-2b_h+ (b_{h+1}+b_{h-1}) \overline{t}\\ &= - b_{h+2} + (b_{h+1}+ b_{h-1}-b_h t) \overline{t}
= - b_{h+2}.
\end{align*}
For $i>h+2$, $t'_i = t_i$ and by induction it is easily proven that $b'_i = -b_i$.
By using \eqref{alphasqrt2} it follows that, when $\vert t \vert = \sqrt{2}$, the relation \eqref{relation*} can be reduced to a relation for which $(m',h') \prec(m,h)$

Suppose now that $|t| = \sqrt{3}$. From \eqref{ba} it follows that $|a^{-1}b| \leq 1$. We claim that $|1-a^{-1}b\bar{t}| < |a^{-1}b|$, or equivalently (see \Cref{norm}) $|a^{-1}b-\frac{t}{2}| < \frac{1}{2}$. Indeed, for suppose $|a^{-1}b-\frac{t}{2}|\geq \frac{1}{2}$ and write $a^{-1}b = x + yi + zj + wk$ and $t = x^{\prime} + y^{\prime}i + z^{\prime}j + w^{\prime}k$ with $x,y,z,w, x^{\prime},y^{\prime}, z^{\prime},w^{\prime} \in \Q$.
To keep notation simple, put $\theta = -1 + xx^{\prime} + u yy^{\prime} + v zz^{\prime} + uv ww^{\prime}$. The inequality $|a^{-1}b-\frac{t}{2}|\geq \frac{1}{2}$ translates to 
$$\vert a^{-1}b \vert^2 \geq \frac{1}{4} - \frac{|t|^2}{4} + (\theta+1) = \theta+\frac{1}{2}.$$
On the other hand, from \eqref{alpha2} it follows that 
$$\vert a^{-1}b \vert^2 < 1 - |t|^2 + 2(\theta+1) = 2\theta.$$
These last two inequalities together yield $\theta+\frac{1}{2} < 2\theta$ so $\frac{1}{2} < \theta$. But then the first inequality yields $|a^{-1}b|^2 \geq \theta+\frac{1}{2} > 1$, a contradiction.

So, we have $|1-a^{-1}b\bar{t}| < |a^{-1}b|$, or equivalently by \eqref{alpha1} 
\begin{equation}\label[equation]{Eric1}
|a-b\bar{t}| < |b| \leq |a|,
\end{equation} and, applying \Cref{norm} to $|b^{-1}a - \bar{t}| < 1$, gives 
\begin{equation}\label[equation]{Eric2}
|2b - at| < |a|.
\end{equation}
Applying \Cref{norm} to \eqref{alpha2} also shows that 
\begin{equation}\label[equation]{Eric3}
|2a - b\bar{t}| < |b| \leq |a|.
\end{equation}

We rewrite \eqref{relation alpha} to \begin{align*} E(t) &= E(0)^2 E(\bar{t})^{-1}E(t)^{-1}E(\bar{t})^{-1}E(t)^{-1}E(\bar{t})^{-1}, \tag{\text{basic cancelation rule of groups applied to \eqref{relation alpha}}} \\
&= E(0)^2 E(0)E(-\bar{t})E(0)^2E(-t)E(0)^2E(-\bar{t})E(0)^2E(-t)E(0)^2E(-\bar{t})E(0), \tag{using \eqref{inv}}\\
&= E(0)^2E(0)E(-\bar{t})E(-t)E(-\bar{t})E(-t)E(-\bar{t})E(0), \tag{centrality of $E(0)^2$ and $E(0)^4 = 1$}\end{align*}
which we substitute in the relation \eqref{relation*} again to obtain (after using \eqref{R1} two times to get rid of $E(0)$) $$E(t_1)\ldots E(t_{h-2})E(t_{h-1}-\bar{t})E(-t)E(-\bar{t})E(-t)E(t_{h+1} - \bar{t})E(t_{h+2}) \ldots E(t_l) = D^{\prime},$$ for some diagonal matrix $D^{\prime}$.
This is the new relation for which we claim $(m',h') \prec (m,h)$ after setting $$t'_i = \begin{cases}
t_i, &\text{if } i < h-1;\\
t_{h-1} - \overline{t}, &\text{if } i = h-1;\\
-t, &\text{if } i = h \text{ or } h+2;\\
-\overline{t}, &\text{if } i = h+1;\\
t_{h+1} - \overline{t}, &\text{if } i = h+3;\\
t_{i-2}, &\text{if } i > h+3.
\end{cases} $$
Note that the new relation has length $l+2$.
We have that $b_i^{\prime} = b_i$ for $i < h$ and 
an easy calculation shows that $b_h^{\prime} =  a-b\bar{t}$, $b_{h+1}^{\prime} = 2b - at$, $b_{h+2}^{\prime} = 2a-b\bar{t}$.
Moreover, when $i > h+2$, then $b'_i = - b_{i-2}$.
Indeed, \begin{align*} b'_{h+3} &= b'_{h+2}t'_{h+2} - b'_{h+1} = (2a - b\overline{t})(-t) - (2b - at)\\ &= b-at = b_{h-1} - b_ht_h =  - b_{h+1}, \end{align*}
and \begin{align*}
 b'_{h+4}& = b'_{h+3}t'_{h+3} - b'_{h+2} = (b-at)(t_{h+1} - \overline{t}) - (2a - b\overline{t})\\& = (b-at)t_{h+1} + a = - b_{h+1}t_{h+1} + b_h = -b_{h+2}.
\end{align*}
For $i>h+4$, $t'_i = t_{i-2}$ and by induction it is easily proven that $b'_i = -b_{i-2}$.
Because of \eqref{Eric1}-\eqref{Eric3} it follows that, also when $\vert t \vert = \sqrt{3}$, the relation \eqref{relation*} may be reduced to a relation for which $(m',h') \prec (m,h)$.

Now, by \Cref{inductionhelp}, the relation obtained in the previous two cases can be reduced to a relation in canonical form  for which $(m',h') $ does not further increase.
We only need to show that $t'_1 \neq 0$. In the steps above, the only way to get $t'_1 = 0$ is if $h=2,$ $t_2 = \bar{t_1},$ $|t_1| = |t_2| < 2$ and $b_2 = t_1$ has maximal modulus. Clearly $|t_1\bar{t_1}-1|= ||t_1|^2 - 1| = 1$ or $2$ (remember that $|t_1|^2=|t_2|^2 = 2$ or $3$). Since $|t_1\bar{t_1}-1| = |b_3| < |b_2| = |t_1| \leq \sqrt{3}$ it follows that $||t_1|^2 - 1| = 1$ so $|t_1|^2 = 2$. In this case, $b_3 \neq 0$, so the length of the relation is at least $4$. One calculates that $b_4 = t_3-t_1$. From the assumptions $|b_4| < |b_2| = |t_1| = \sqrt{2}$, so $|b_4| \in \{0,1\}$.

If $|b_4| = 0$, then $t_3 = t_1$.
If the length of the relation is exactly $4$, then one can show that $t_4 = \bar{t_1}$, but this is the relation \eqref{relation alpha} and the induction step would stop here.
The length cannot be exactly $5$. Indeed $|b_4| = 0$ implies that $E(t_1)...E(t_4)$ is a lower-triangular matrix and thus $E(t_1)...E(t_4)E(t_5)$ cannot be a diagonal matrix. So the length of the relation is at least $6$. From easy calculations it follows that $b_5 = -1$ and $b_6 = -t_5$. From the maximality of $h$ we can deduce that $|t_5| = |b_6| < |b_2| = \sqrt{2}$, showing that $t_5$ is either a unit or $0$, a contradiction with the fact that the relation was in canonical form.

Thus suppose that $|b_4| = 1$. We will first show that if $b_{i-1}$ and $b_i$ are units, then $b_{i+1}$ is also a unit. Indeed, $|b_{i+1}| = |b_it_i -b_{i-1}| \geq \left| |b_i||t_i| - |b_{i-1}| \right| = \left| |t_i| - 1 \right| \geq \sqrt{2} - 1 > 0$. The fact that $|t_i| \geq \sqrt{2}$ follows from the fact that the length of the relation is at least $i+1$ and thus $t_i$ is not a unit or $0$ from the canonical form. On the other hand, by the minimality of $h$ we need $|b_{i+1}| < |b_2| = \sqrt{2}$, showing that $|b_{i+1}| = 1$ and $b_{i+1}$ is a unit.

Through $|b_3|=|b_4| = 1$ and the repeated use of the result above we obtain that the word should be infinitely long, a contradiction.

In the end, we proved that $t'_1 \neq 0$ and that \Cref{inductionhelp} can be applied.
This finishes the proof.
\end{proof}

\subsection{On the abelianization of $\GE_2(\O)$ over orders $\O$ with $\U(\O)$ finite}\label[subsection]{subsection_GE2}

Let $\O$ be an order in a finite dimensional division $\Q$-algebra $D$ with $\U(\O)$ finite. The main goal of this section is to describe $\GE_2(\O)^{ab}$ in a computable and uniform way.
More concretely in \Cref{abelianization GE2} we obtain a short exact sequence
\[1\ \longrightarrow\ (\O/N,+) \ \longrightarrow\ \GE_2(\O)^{ab}\ \longrightarrow\ \U(\O)^{ab}\ \longrightarrow\ 1\]
where $N$ is the two-sided ideal generated by the elements $u-1$ with $u \in \U (\O)$.
To start we describe $\GE_2(R)/\E_2(R)$ in the more general context of rings which are almost-universal for $\GE_2$.
Thereafter we restrict to orders in finite-dimensional division $\Q$-algebras with finite unit group and prove that an exact sequence as stated above exists.
This is inspired by the results in \cite{Cohn1}.

\begin{proposition}\label[proposition]{theorem difference GE2 and E2}
Let $R$ be a ring which is almost-universal for $\GE_2$, then 
$$ \GE_2(R)/\E_2(R) \cong \U(R)^{ab}. $$
The isomorphism is induced by the map \begin{equation}\label[equation]{eq:varphi} \varphi : \GE_2(R) \rightarrow \U(R)^{ab} \mbox{ by } \varphi (E(x)) = 1 \mbox{ and } \varphi([\alpha, \beta]) = \widetilde{\alpha \beta},\end{equation} which is a group homomorphism. The map \ $\widetilde{}: \U(R) \rightarrow \U(R)^{ab}$ denotes the canonical morphism.
\end{proposition}
\begin{proof}
Since $\E_2(R)$ is in the kernel of $\varphi$, it is enough to check that the relations not in $\E_2(R)$ are preserved to prove that $\varphi$ is a well-defined group homomorphism. The only such relations are those of the form $E(x)[\alpha, \beta] = [\beta, \alpha] E(\beta^{-1} x \alpha)$ and the relations in $\D_2(R)$. As $\U(R)^{ab}$ is abelian, $\widetilde{\alpha \beta }= \widetilde{ \beta \alpha}$ and hence the first type of relation is preserved. It is easy to check that $\varphi$ preserves the relations in $\D_2(R)$.
 
Now the map $\varphi$ is onto and $\E_2(R) \subseteq \ker (\varphi)$. For the reverse inclusion, let $A \in \ker(\varphi)$. Using the universal relations we may write $A = [\alpha, \beta] E(x_1)\ldots E(x_r)$. Since $A\in \ker (\varphi)$ we have that $\widetilde{\alpha \beta } =\widetilde{\beta \alpha} = \widetilde{1}$, i.e. $\beta \alpha \in \U(R)'$. Hence, by \eqref{R5}, $[\beta \alpha,1] \in \E_2(R)$. As $D(\beta) [\alpha, \beta] = [\beta \alpha,1]$, we have that $[\alpha,\beta] \in \E_2(R)$ and hence $A \in \E_2(R)$.
\end{proof}

\begin{corollary} \label[corollary]{intersection D_2 and E_2}
The following properties hold for a ring $R$ which is almost-universal for $\GE_2$:
\begin{enumerate}
\item $\GE_2(R)' \subseteq \E_2(R)$,
\item $\D_2(R) \cap \E_2(R) = \langle D(\mu) \mid \mu \in \U(R) \rangle$. 
\end{enumerate}
\end{corollary}
\begin{proof}
The first statement is a direct consequence of \Cref{theorem difference GE2 and E2}. For the second statement assume $[\alpha, \beta] \in \D_2(R) \cap \E_2(R)$. Then $\varphi([\alpha, \beta]) = \widetilde{1}$, in particular, $[\alpha, \beta] = D(\beta^{-1})[\beta \alpha, 1]$ with $\beta \alpha \in \U(R)'$. Consequently, writing $\beta \alpha = \prod_{i \in I} \delta_i^{-1} \mu_i^{-1} \delta_i \mu_i$, we see that  $[\alpha, \beta] =D(\beta^{-1}) \prod_{i \in I} D(\delta_i^{-1}) D(\mu_i^{-1}) D(\delta_i \mu_i) \in \langle D(\mu) \mid \mu \in \U(R) \rangle.$ 
\end{proof}

\Cref{theorem difference GE2 and E2} also indicates that in order to understand $\U(R)^{ab}$ for some ring $R$ which is almost-universal for $\GE_2$, one may ``increase its size'' to $\GE_2(R)$ and instead investigate its abelianization (which will be the content of the following subsection).
Recall that the \emph{Borel subgroup} of $\GE_2(R)$, denoted $\B_2(R)$, is the subgroup consisting of the upper-triangular matrices with units on the diagonal, i.e. $\B_2(R) = \{ \left( \begin{matrix}
\alpha & x \\ 0 & \beta
\end{matrix} \right) \mid x\in R, \alpha, \beta \in \U(R) \}$.

\begin{proposition}\label[proposition]{equivalence finite ab for almost universal rings}
Let $R$ be a ring, finitely generated as $\Z$-module, which is almost-universal for $\GE_2$, then the following properties are equivalent:
\begin{enumerate}
\item $\U(R)^{ab}$ is finite,
\item $\B_2(R)^{ab}$ is finite,
\item $\GE_2(R)^{ab}$ is finite.
\end{enumerate}
\end{proposition}
\begin{proof}
It is easy to calculate that for any $a \in R$ we get $$ \begin{pmatrix} -1 & 0\\ 0 & 1 \end{pmatrix}^{-1}  \begin{pmatrix} 1 & a\\ 0 & 1 \end{pmatrix}^{-1} \begin{pmatrix} -1 & 0\\ 0 & 1 \end{pmatrix}  \begin{pmatrix} 1 & a\\ 0 & 1 \end{pmatrix}  =   \begin{pmatrix} 1 & a\\ 0 & 1 \end{pmatrix}^2 \in \B_2(R)',$$
 and clearly also $D_2(R)' \leq B_2(R)'$. This shows that $B_2(R)^{ab}$ is an epimorphic image of the group $H \times D_2(R)^{ab} \cong H \times \U(R)^{ab} \times \U(R)^{ab}$, where $H$ is some finitely generated abelian group of exponent $2$ (and so it is finite).
 Hence, if $\U(R)^{ab}$ is finite then also $B_2(R)$ has finite abelianization.
 
For the next implication notice that $\GE_2(R) = \langle E(0), \B_2(R) \rangle$. Since $E(0)$ has finite order we now easily see that $\GE_2(R)^{ab}$ is finite if $\B_2(R)^{ab}$ is finite.

Finally from \Cref{theorem difference GE2 and E2} and \Cref{intersection D_2 and E_2} it follows that $\U(R)^{ab}$ is an epimorphic image of $\GE_2(R)^{ab}$ and so the remaining implication also follows.
\end{proof}

\begin{corollary}\label[corollary]{GE2Oab_finite}
 Let $\O$ be an order in a finite dimensional division $\Q$-algebra with $\U(\O)$ finite, then $\GE_2(\O)^{ab}$ is finite.
 \end{corollary}
 \begin{proof}
Because $\U (\O)$ is finite, we know from \Cref{When is unit group order finite} and \Cref{quaternion cohn} that $\O$ is almost-universal for $\GE_2$. Hence, we may apply \Cref{equivalence finite ab for almost universal rings} and it suffices to show that $\U(\O)^{ab}$ is finite. However, this follows readily from the fact that $\U(\O)$ is finite.
 \end{proof}
 
If $R$ is almost-universal for $\GE_2$, then by \Cref{theorem difference GE2 and E2} we have that $\GE_2(R)' \subseteq \E_2(R)$ and
\[
\U(R)^{ab} \cong \GE_2(R)^{ab} / \left( \E_2(R)/\GE_2(R)' \right),
\] where the isomorphism is induced by the map 
\begin{equation}\label[equation]{eq:bar_varphi}\varphi : \GE_2(R) \rightarrow \U(R)^{ab}\qquad \mbox{ with }\qquad \varphi (E(x)) = 1 \mbox{ and } \varphi([\alpha, \beta]) = \widetilde{\alpha \beta},\end{equation}
for $\alpha, \beta \in \U(R)$ and $x \in R$. So in order to understand $\GE_2(R)^{ab}$ it remains to describe $\E_2(R)/\GE_2(R)'$. For orders $\O$ in a finite dimensional division $\Q$-algebra with a finite number of units this will be achieved through the following map 
\begin{equation}\label[equation]{eq:psib}
\psi: \E_2(\O) \rightarrow (\O/N, +): E(x) \mapsto x-1 + N ,
\end{equation}
\noindent where $N$ is the two-sided ideal of $\O$ generated by the elements $u -1$ with $u \in \U(\O)$. In the following theorem we will prove that the kernel of $\psi$ is exactly $\GE_2(\O)'$.
 
 \begin{theorem}\label{abelianization GE2}
Let $\O$ be an order in a finite dimensional division $\Q$-algebra with $\U(\O)$ finite and let $N$ be the two-sided ideal of $\O$ generated by the elements $u -1$ with $u \in \U(\O)$. Then
$$ \E_2(\O)/\GE_2(\O)' \cong (\O/N,+).$$
In particular, we have the following short exact sequence of groups:
 \[1\ \longrightarrow\ ( \O/N,+)\ \stackrel{\iota \circ \bar{\psi}^{-1}}{\longrightarrow}\ \GE_2(\O)^{ab}\ \stackrel{\bar{\varphi}}{\longrightarrow}\ \U(\O)^{ab}\ \longrightarrow\ 1,\]
 where $\iota \colon \E_2(\O)/\GE_2(\O)' \hookrightarrow \GE_2(\O)/\GE_2(\O)'$ is induced by the inclusion $\E_2(\O) \hookrightarrow \GE_2(\O)$, $\bar{\psi}$ is the isomorphism induced by $\psi$ defined in \eqref{eq:psib} and $\bar{\varphi}$ is induced by $\varphi$ in \eqref{eq:bar_varphi}.
 \end{theorem}
 \begin{proof}
Because $\U (\O)$ is finite, we know from \Cref{When is unit group order finite} and \Cref{quaternion cohn} that $\O$ is almost-universal for $\GE_2$. Hence, by the discussion before \Cref{abelianization GE2}, we know that \[1\ \longrightarrow\ \E_2(\O)/\GE_2(\O)'\ \stackrel{\iota}{\longrightarrow}\ \GE_2(\O)^{ab}\ \stackrel{\bar{\varphi}}{\longrightarrow}\ \U(\O)^{ab}\ \longrightarrow\ 1,\]
forms an exact sequence of groups. So, in order to prove \Cref{abelianization GE2}, it suffices to show that $\bar{\psi}$ is well-defined and forms an isomorphism between $\E_2(\O)/\GE_2(\O)'$ and $(\O/N,+)$.

First we show that $\psi\colon  \E_2(\O) \rightarrow (\O/N, +): E(x) \mapsto x-1 + N$ is well defined.
For this it is enough to prove that the relations from \Cref{generators and relations for E_2(R)} are preserved.
We will use $\alpha$ to denote an element of $\U(\O)$.

Remark that, by the definition of $N$, $D(\alpha) = E(0)^2 E(\alpha) E(\alpha^{-1}) E(\alpha)$ is mapped to $-2$. As $-2 \in N$, $D(\alpha)$ is mapped to zero for every $\alpha \in \U(\O)$.
In particular, relations \eqref{R4} and \eqref{R6} are preserved.
Further relation \eqref{R3'} reduces to $\alpha x \alpha \equiv x \bmod N$, for $\alpha \in \U(\O)$ and $x \in \O$. Since $\alpha x \alpha- x =(\alpha -1)x \alpha + x (\alpha-1) \in N$, it is indeed preserved under $\psi$.
Relation \eqref{R1} is trivially preserved under $\psi$.
Finally the only relations left to check are \eqref{R7}. By \Cref{quaternion cohn} and \Cref{When is unit group order finite}, these relations are of the form \eqref{relation alpha}.
They  are easily checked using that $2 \in N$ and $a + \overline{a} = 2 \Tr(a)$.

We want to show that $\GE_2(\O)' \subseteq \ker (\psi)$. 
To do this, remark that clearly $\E_2(\O)' \subseteq \ker (\psi)$ and that $D(\alpha) \in \ker (\psi)$, as proven above.
It only remains to prove that, for $x \in \O$ and $\alpha,\beta, \gamma, \delta \in \U(\O)$ we have that the commutator between $[\alpha, \beta]$ and $[\gamma, \delta]$ and the commutator between $[\alpha, \beta]$ and $E(x)$ is in the kernel since these elements (together with $\E_2(\O)'$) generate $\GE_2(\O)'$ as a normal subgroup and the image of $\psi$ is an abelian group.

Clearly the commutator between $[\alpha, \beta]$ and $[\gamma, \delta]$ is a diagonal matrix in $\E_2(\O)$ by \Cref{theorem difference GE2 and E2}, and thus by the above it is in the kernel of $\psi$.

For the other commutator we can write
\begin{align*}
[\alpha, \beta]^{-1}E(x)[\alpha, \beta]E(x)^{-1} &= [\alpha^{-1}, \beta^{-1}][\beta, \alpha]E(\beta^{-1}x\alpha)E(x)^{-1}, \\ &= [\alpha^{-1}\beta, \beta^{-1}\alpha]E(\beta^{-1}x\alpha)E(0)E(-x)E(0).
\end{align*}
Since $[\alpha^{-1}\beta, \beta^{-1}\alpha] = D(\alpha^{-1}\beta)$ this commutator is mapped, under $\psi$, to $\beta^{-1}x\alpha - x - 4$. As $-4 \in N$ and $\beta^{-1}x\alpha - x = \beta^{-1}(x \alpha - \beta x) = \beta^{-1}(x (\alpha - 1) - (\beta - 1)x) \in N$, this commutator is also in $\ker (\psi)$.

Now $\psi$ induces $\bar{\psi}: \E_2(\O)/\GE_2(\O)' \rightarrow (\O/N, +)$. Since $\psi$ is surjective, it remains to prove the injectivity of $\bar{\psi}$.

Note that an arbitrary element in $\E_2(R)$ can be written as
\begin{equation}\label[equation]{will element in E_2}
E(x_1+ 3) \cdots E(x_l+3), \qquad x_1, ..., x_l \in \O.
\end{equation}
Further remark the following crucial identity
\begin{align}\label[equation]{mod E_2 product of two E matrices is a E matrix}
E(x) E(y) &\equiv E(x) E(0) E(y) E(0) E(-1) E(0) E(-1) E(0) E(-1) \notag \\ &\equiv E(x+y -3) \mod \E_2(\O)'
\end{align}
where we used $E(0)^4= E(-1)^3= 1$ and \eqref{R1}.
By induction on $l$ then
\begin{equation}\label[equation]{prod van E's samen nemen}
E(x_1 + 3) \ldots E(x_l + 3) \equiv E\left(\left(\sum_{i=1}^l x_i\right) + 3\right) \mod \E_2(\O)'.
\end{equation}
Suppose now that the expression \eqref{will element in E_2} is in $\ker (\psi)$.
Then $$0 \equiv (x_1 +2) + \ldots+ (x_l +2) \equiv \sum_{i=1}^l x_i \mod N,$$ since $2 \in N$.
In particular by \eqref{prod van E's samen nemen} is it enough to prove that $E(n +3) \equiv 1 \mod \GE_2(\O)'$ for all $n \in N$.
By the definition of $N$, it is enough to do so for $E((\alpha-1)x + 3)$ and $E(x(\alpha - 1) + 3)$, where $x \in \O$ and $\alpha \in \U(\O)$. 
Using equation \eqref{2.8} we obtain
$$\begin{array}{ll}
E((\alpha -1)x + 3) & = E(\alpha(x-3) - (x-3) + 3\alpha) \\
& = E(\alpha(x-3)) E(x-3)^{-1}E(3 \alpha).
\end{array}$$
Moreover \begin{equation}\label[equation]{3.7E} E(3 \alpha) \equiv D(\alpha) \bmod \E_2(\O)', \end{equation} and \begin{equation}\label[equation]{3.7D} D(\alpha) \equiv D(\alpha^{-1}) \bmod \E_2(\O)'.\end{equation} Indeed, using \eqref{R3'}, \eqref{R2} and \eqref{inv}:
\begin{align*}
E(\alpha x \alpha) E(x)^{-1} & \equiv D(\alpha)^2 = - D(\alpha) D(- \alpha) \\
& = - E(\alpha) E(\alpha^{-1}) E(\alpha) E(-\alpha) E(-\alpha^{-1}) E(-\alpha) \\
& = - E(\alpha) E(\alpha^{-1}) E(\alpha)E(0)E(\alpha)^{-1}E(0)^2E(\alpha^{-1})^{-1}E(0)^2E(\alpha)^{-1}E(0) \\
& \equiv 1 \mod \E_2(\O)'
\end{align*}
and we get that both $D(\alpha) \equiv D(\alpha^{-1}) \mod \E_2(\O)'$ and
\begin{equation}\label[equation]{E(axa)E(x) in comm}
 E(\alpha x \alpha) \equiv E(x) \mod \E_2(\O)',
\end{equation}
in particular $E(\alpha) \equiv E(\alpha^{-1}) \bmod \E_2(\O)'$.
By the latter, \eqref{R1} and \eqref{R2}, $$E(3\alpha) = E(\alpha) E(0)E(\alpha)E(0) E(\alpha) \equiv E(0)^2 E(\alpha) E(\alpha^{-1})E(\alpha) =D(\alpha) \mod \E_2(\O)',$$
as claimed.
Taking in \eqref{R3} the diagonal matrix $[\alpha^{-1}, \alpha^{-1}]$ we see that $E(x) \equiv E(\alpha x \alpha^{-1}) \bmod \E_2(\O)'$ and so also $E(\alpha x) \equiv E(x \alpha) \bmod \E_2(\O)',$ (replace $x$ by $x\alpha$). 

Now, 
\begin{align*}
E(\alpha (x-3)) E(x-3)^{-1}E(3 \alpha) & \equiv E(x-3)^{-1} E(\alpha (x-3)) D(\alpha^{-1}), \\& \equiv E(x-3)^{-1} E(\alpha(x-3) )  [1, \alpha][\alpha^{-1}, 1] , \\
&\equiv E(x-3)^{-1} [\alpha, 1] E(x-3) [\alpha^{-1},1] \equiv 1 \mod \E_2(\O)',\\
\end{align*}
where in the second to last equality, \eqref{R3} is used.
So altogether we proved that $E((\alpha -1)x + 3)  \equiv 1 \bmod \E_2(\O)'$. In analogue way one proves that $E(x(\alpha -1) + 3)  \equiv 1 \bmod \E_2(\O)'$, finishing the proof.
\end{proof}

\begin{corollary}\label[corollary]{prop:GE_2_O_ab}  Let $\O$ be an order in a finite dimensional division $\Q$-algebra with $\U(\O)$ finite. If $\U(\O)$ contains an element of odd order, then $\GE_2(\O)^{ab} \cong \U(\O)^{ab}$.
 \end{corollary}
 
 \begin{proof} Assume that an element of $\U(\O)$ has odd order. Then there exists an element $\alpha \in \U(\O)$ of odd prime order, say $p$. For this element holds $1 + \alpha + \ldots + \alpha^{p-1}= 0$ and hence $1 = \sum\limits_{i=1}^{p-1}(-1)^i(1-(-\alpha)^i) \in N$. Thus $N = \O$ and $\GE_2(\O)^{ab} \cong \U(\O)^{ab}$ by \Cref{abelianization GE2}. 
 \end{proof}

At the end of the next section we will exploit \Cref{prop:GE_2_O_ab} to give exact descriptions of $\GE_2(\O)^{ab}$ for certain orders $\O$.\\

\subsection{On the abelianization of $\E_2(\O)$ over orders $\O$ with $\U(\O)$ finite} \label[subsection]{sectie ab van E2}

As in the previous subsection, we will obtain a short exact sequence that will allow us to study the abelianization of $\E_2(\O)$ over orders with $\U(\O)$ finite.

\begin{theorem} \label{theoremM}
Let $\O$ be an order in a finite-dimensional division $\Q$-algebra with $\mathcal{U}(\O)$ finite. Let $M$ be the additive subgroup of $\O$ generated by the following set of elements:
\begin{enumerate}
\item $\alpha x \alpha -x$ with $x \in \O$ and $\alpha \in \U(\O)$,
\item $\sum_{i = 1}^m 3(\alpha_i +1)(\beta_i+1)$ with $\alpha_i, \beta_i \in \mathcal{U}(\O)$ satisfying $\prod_{i = 1}^m \alpha_i^{-1} \beta_i^{-1} \alpha_i \beta_i =1$,
\item the elements $2 (x + \bar{x}) +6$ for each element $x \in \O$ satisfying $|x|^2=2$,
\item the elements $3 (x + \bar{x})$ for each element $x \in \O$ satisfying $|x|^2=3$.
 \end{enumerate}
 Then, $$\tau: \E_2(\O) \rightarrow (\O/M, +): E(x) \mapsto x-3 + M $$
 is an epimorphism with $\ker(\tau) = \E_2(\O)'$.  
 In particular 
 $$  \E_2(\O)/\E_2(\O)' \cong (\O/M, +).$$
\end{theorem} 

\begin{remark}
As $\U(\O)$ is finite, $\O$ is an order in $\Q$, in an quadratic imaginary extension of $\Q$ or a totally definite quaternion algebra over $\Q$, by \Cref{When is unit group order finite} and the norm map appearing in the third and fourth item of the definition of $M$ in \Cref{theoremM} is the same as in the beginning of \Cref{subsect:almost_universalo_quaternion_orders}: $|x| = \sqrt{x \bar{x}}$.
\end{remark} 

\begin{proof}
We first prove that the map $\tau$ is well-defined and a group homomorphism. For this it is enough to check that $\tau$ preserves the defining relations of $\E_2(\O)$ stated in \Cref{generators and relations for E_2(R)}, with $\Phi$ the non-universal set of relations of the form \eqref{relation alpha}.

Relation \eqref{R1} is trivially preserved. Note that  $12= 3 (1+1)(1+1) \in M$. Hence \eqref{R4}, or equivalently $E(0)^2 = E(0)^2E(-1)^3$, is preserved.
Since $\alpha x \alpha \equiv x \mod M$, for any $\alpha \in \mathcal{U}(\mathcal{O})$ and $x \in \mathcal{O}$, we have that $$\alpha \equiv \alpha^{-1} \mod M.$$
Now, the image of \eqref{R3'} under $\tau$ yields the equation $x-3 + 2 \alpha +\alpha^{-1}-3 \equiv 2 \alpha^{-1} + \alpha -3 + \alpha x \alpha -3 \bmod M$ or thus
$$\alpha x \alpha -x \equiv \alpha - \alpha^{-1} \equiv 0 \mod M.$$

We now consider the preservation of (\ref{R6}). Since $\alpha \equiv \alpha^{-1} \mod M$ for any unit $\alpha$, we immediately obtain also that $\tau(D(\alpha^{-1})) \equiv \tau(D(\alpha)) \equiv 3(\alpha - 1) \mod M$.
By definition there is for every $1\leq k \leq n$ a decomposition $[w_k, 1] = \prod_{i \in I_k} D(\alpha_{i,k}^{-1}) D(\beta_{i,k}^{-1}) D(\alpha_{i,k} \beta_{i,k})$.
Furthermore by $(2)$, $$\begin{array}{ll}
\tau\left( \prod\limits_{1 \leq k \leq n } \prod\limits_{i \in I_k} D(\alpha_{i,k}^{-1}) D(\beta_{i,k}^{-1}) D(\alpha_{i,k} \beta_{i,k})\right) & \equiv \sum\limits_{1 \leq k \leq n } \sum\limits_{i \in I_k} 3(\alpha_{i,k} -1) + 3(\beta_{i,k} -1) + 3( \alpha_{i,k} \beta_{i,k} -1) \\
& \equiv \sum\limits_{1 \leq k \leq n } \sum\limits_{i \in I_k} 3(\alpha_{i,k} + \beta_{i,k} + \alpha_{i,k} \beta_{i,k} +1) \\
& \equiv \sum\limits_{1 \leq k \leq n } \sum\limits_{i \in I_k}  3(\alpha_{i,k} + 1 ) (\beta_{i,k} +1) \equiv 0 \mod M,
\end{array}$$
yielding that (\ref{R6}) is preserved.

Finally consider the relation $(E(\overline{x}) E(x))^n = E(0)^{2}$ from \eqref{relation alpha}. If $|x| = \sqrt{2}$ then 
$$
\tau\left( E(\overline{x}) E(x) E(\overline{x}) E(x) E(0)^{2} \right) \equiv 2(\overline{x} + x)  + 6 \equiv  0 \mod M.
$$
Similarly, $\tau\left( (E(\overline{x}) E(x))^3 E(0)^{2} \right)\equiv 3 (\overline{x} + x) \equiv 0 \mod M$, if $|x|= \sqrt{3}$.

Altogether we proved that $\tau$ is well-defined and hence defines an epimorphism. Since $(\O/M,+)$ is abelian, $E_2(\O)' \subseteq \ker (\tau)$.
By \eqref{prod van E's samen nemen} in the proof of \Cref{abelianization GE2} the reverse inclusion follows if $E(m +3) \in \E_2(\O)'$ for all additive generators $m$ of  $M$.

Due to \eqref{mod E_2 product of two E matrices is a E matrix} and \eqref{E(axa)E(x) in comm} in the proof of \Cref{abelianization GE2} one also immediately obtains $E(3) \in \E_2(\O)'$ and $E(\alpha x \alpha) E(x)^{-1} \in \E_2(\O)'$. Consequently, by \eqref{2.8}, $E(\alpha x \alpha - x +3) = E(\alpha x \alpha) E(x)^{-1} E(3) \in \E_2(\O)'$. Next consider an element $\sum_{i \in I} 3 (\alpha_i +1)(\beta_i + 1)$ such that $\prod_{i = 1}^m \alpha_i^{-1} \beta_i^{-1} \alpha_i \beta_i =1$. By first using \eqref{prod van E's samen nemen}, then consecutively \eqref{R1}, \eqref{3.7E}, \eqref{3.7D}, \eqref{R5} and finally \eqref{R6} we obtain
$$\begin{array}{ll}
E\left( \sum_{i = 1}^m 3 (\alpha_i +1)(\beta_i + 1) + 3 \right) & \equiv \prod_{i = 1}^m E( 3 (\alpha_i +1)(\beta_i + 1) + 3) \\
& \equiv \prod_{i = 1}^m E( 3 \alpha_i) E(0) E(3 \beta_i) E(0) E(3 \alpha_i \beta_i) (E(0) E(3) )^2 \\
& \equiv \prod_{i = 1}^m D(\alpha_i)D(\beta_i)D(\alpha_i \beta_i)\\
& \equiv \prod_{i = 1}^m D(\alpha_i^{-1})D(\beta_i^{-1})D(\alpha_i \beta_i) \\
& \equiv \prod_{i = 1}^m [\alpha_i^{-1}\beta_i^{-1}\alpha_i \beta_i,1] \equiv  1 \mod \E_2(\O)'.
\end{array}$$
Now consider an element $2 (x + \overline{x}) +6$ with $|x|^2= 2$. Then, using \eqref{R1}, the fact that modulo $\E_2(\O)'$ all elements commute, that $E(0)^4 = I$, that $E(3) \in \E_2(\O)'$ and finally \eqref{relation alpha},
\begin{align*}
\MoveEqLeft E(2(x+\overline{x}) + 9) \\& \equiv E(0)^2 E(x+\overline{x})E(0) E(0)^2 E(x+\overline{x})E(0) E(0)^2 E(3)E(0)E(0)^2 E(3)E(0)E(3) \\ &\equiv E(x+\overline{x})^2 E(3)^3 \equiv E(x+\overline{x})^2  \equiv E(0)^2E(\overline{x})E(0)E(x)E(0)^2E(\overline{x})E(0)E(x)\\
 &\equiv (E(\overline{x})E(x))^2 E(0)^2  \equiv E(0)^2E(0)^2 \equiv 1 \mod \E_2(\O)'.
\end{align*}
In case of the additive generators $3 (x + \bar{x})$ the proof is analogue, hence finishing the proof.
\end{proof}

\begin{remark}\label[remark]{rem:12_in_M}
Note the subtle, but crucial, point that $M$ is defined to be the additive subgroup generated by those elements listed in \Cref{theoremM}, in contrast to $N$ from \Cref{abelianization GE2} which is defined as the two-sided ideal generated by the elements in that statement.  In fact, since $12 \in M$, one can deduce that $M$ can only be an ideal when $\O / M$ is fnite (which by \Cref{finite_abelianization_equiv} only is the case when $\O$ contains a $\Z$-basis consisting of units).
Also it is interesting to remember that the elements $\alpha x \alpha -x$ exactly encode the image of (\ref{R2}) under $\tau$, the elements $\prod 3(\alpha_i +1)(\beta_i +1)$ the relations (\ref{R6}) and the last two elements encode the relations of the form (\ref{relation alpha}).

Finally note that if $\mathcal{U}(\O)$ is abelian, then the condition $\prod_{i = 1}^m \alpha_i^{-1} \beta_i^{-1} \alpha_i \beta_i =1$ is always satisfied, hence in this case one simply adds all elements $3(\alpha + 1) (\beta + 1)$. 
\end{remark}

\begin{remark}\label[remark]{defofOrders}
We denote by $\mathcal{I}_d$ the ring of algebraic integers in the imaginary quadratic number field $\Q(\sqrt{-d})$ of a positive integer $d$, e.g.\ $\mathcal{I}_1 = \Z[\sqrt{-1}]$ and $\mathcal{I}_3 = \Z[\zeta_3]$, where $\zeta_3$ is a primitive complex third root of unity. It is well known that $\mathcal{I}_d$ is Euclidean if and only if $d \in \{1, 2, 3, 7, 11\}$. In \cite{Fitz}, Fitzgerald showed that the only totally definite quaternion algebras $\mathbb{H}_d$ with center $\mathbb{Q}$ containing a right norm Euclidean order are \[\mathbb{H}_2 = \qa{-1}{-1}{\mathbb{Q}},\quad \mathbb{H}_3 = \qa{-1}{-3}{\mathbb{Q}}\quad \mbox{ and }\quad \mathbb{H}_5 = \qa{-2}{-5}{\mathbb{Q}}.\]

Note that orders that are (right norm) Euclidean are maximal \cite[Proposition 2.8]{CeChLez}. Furthermore a quaternion algebra having a right norm Euclidean order has class number one \cite[Proposition 2.9]{CeChLez} and thus also type number one meaning that there is only one conjugacy class of maximal orders. In \cite{Fitz} also a specific representative of that unique conjugacy class, denoted $\O_2 \subseteq \mathbb{H}_2$, $\O_3 \subseteq \mathbb{H}_3$ and $\O_5 \subseteq \mathbb{H}_5$, is constructed. For later use we explicitly state in the table below specific $\mathbb{Z}$-bases $\{b_1, b_2, b_3, b_4\}$ of these orders (which also can be found in \cite[Proposition~12.3.2]{EricAngel1}). The quaternion algebra $\mathbb{H}_2$ also contains the order of \emph{Lipschitz quaternions} $\LL$ consisting of all integral linear combinations of the basis elements $1, i, j, k$.
\newcolumntype{C}{>{\centering\arraybackslash}p{2.5cm}}
\newcolumntype{B}{>{\centering\arraybackslash}p{2cm}}
\begin{equation}\label[equation]{eq:basis_of_quat_orders} 
\begin{tabular}{|*{2}{B|}*{3}{C|}} 
   & $b_1$ & $b_2$ & $b_3$ & $b_4$  \\ \hline
  $\LL$ & $1$ & $i$ & $j$ & $k$ \\
  $\O_2$ & $1$ & $i$ & $j$ & $\omega_2 = \frac{1 + i + j + k}{2}$ \\
  $\O_3$ & $1$ & $i$ & $\omega_3 = \frac{1+j}{2}$ & $\frac{i+k}{2}$  \\
  $\O_5$ & $1$ & $ \frac{1+i+j}{2}$ & $\omega_5 = \frac{2+i-k}{4}$ & $\frac{2+3i+k}{4}$ \\ \hline
\end{tabular} \end{equation}
\end{remark}

When $R$ is a ring which is also freely generated as a $\Z$-module (e.g. $R$ is an order) we define $$\operatorname{inv}_R = \max \{ \left| B \cap \U(R) \right| \mid B \text{ a } \Z\text{-module basis of }R\}, $$ and for a finitely generated abelian group $G$ one defines $$\rank_{\Z} G  = \max \{ n \mid \Z^n \text{ is, up to isomorphism, a subgroup of } G \}.$$

\begin{theorem} \label{finite_abelianization_equiv}
Let $\O$ be an order in a finite dimensional division $\Q$-algebra $D$ with $\U(\O)$ finite. Then,
 \begin{equation}\label[equation]{rankformula}
\rank_{\Z} \E_2(\O)^{ab} = \rank_{\Z} \O - \operatorname{inv}_{\O}.
\end{equation}
\noindent Moreover, the following properties are equivalent:
\begin{enumerate}[label=(\alph*),ref=\alph*]
\item\label[equation]{it:finiteE2} $\E_2(\O)^{ab}$ is finite,
\item\label[equation]{it:specificO} $\O$ is isomorphic to $\Z, \I_1, \I_3, \LL, \O_2$ or $\O_3$,
\item\label[equation]{it:basis} $\O$ has a $\Z$-basis consisting of units of $\O$,
\item\label[equation]{it:subring} $\O$ is generated as a ring by $\U (\O)$,
\item\label[equation]{it:submodule} $\O$ is generated as a $\Z$-module by $\U (\O)$. 
\end{enumerate}
\end{theorem}

\begin{proof}
Throughout, we will rely on \Cref{When is unit group order finite}. We start off by proving formula \eqref{rankformula}. We will use the description of $\E_2(\O)^{ab}$ given in \Cref{theoremM} and the additive subgroup $M$ defined there.
Since $3(1+1)(1+1) = 12 \in M$, and for any unit $\alpha$ holds that $3(\alpha + 1)(1+1) + 3(\alpha + 1)(1+1) = 12\alpha + 12 \in M$, we readily obtain that $12 \alpha \in M$. Consequently, any unit of $\O$ has finite (additive) order in $(\O/M, +)$.
As such, $\rank_{\Z} \E_2(\O)^{ab} \leq \rank_{\Z} \O - \operatorname{inv}_{\O}$.

If $D = \Q$, then $\O = \Z$ and the statement is correct since $\E_2(\Z)^{ab} \cong C_{12}$ is finite.
If $D$ is a quadratic imaginary extension of $\Q$, $\O$ is a free $\Z$-module of rank $2$.
Assume there exists a basis consisting of units for $\O$. Hence $\rank_{\Z} \O - \operatorname{inv}_{\O} = 2-2 = 0 \leq \rank_{\Z} \E_2(\O)^{ab}$, showing that the inequality holds trivially.
If not, then one may assume the existence of a base of $\O$ of the form $\{ 1, a \}$ with $a \notin \U(\O)$.
It is well known that in this case $\U(\O) = \{ \pm 1\}$.
The generators of type $(1)$ of $M$ (in \Cref{theoremM}) are then all equal to $0$, and the generators of type $(2), (3)$ and $(4)$ are in $\Z$.
As such, $12 \Z \subseteq M \subseteq \Z$ and thus $\rank_{\Z} \E_2(\O)^{ab} = \rank_{\Z}(\O/M, +) = 1 =  \rank_{\Z} \O - \operatorname{inv}_{\O}$.

The last situation to consider is when $D$ is a totally definite quaternion algebra over $\Q$.
If $\O$ contains a basis of units, similar to before, the inequality is trivially satisfied. In particular, we may assume that $\O$ is not isomorphic to $\LL, \O_2$ or $\O_3$.
Hence, by \cite[Theorem~11.5.12]{Voight}, the unit group $\U(\O)$ is cyclic. We will denote the generator by $\beta$.

Clearly, the elements of the forms $(2), (3)$ and $(4)$ in \Cref{theoremM} are in $\Z[\beta]$.
For elements of the form $(1)$ in \Cref{theoremM} we do the following.
Take any element $\gamma \in \O$ of norm $1$.
Then $\gamma$ is a root of a polynomial $(X-\gamma)(X-\bar{\gamma}) = X^2 - t X + 1$ for $t = \gamma + \bar{\gamma} \in \mathcal{Z}(\O) = \mathbb{Z}$, so $\gamma^2 = t\gamma - 1$ and hence for every $x \in \O$
\begin{align*}(\gamma x\gamma  - x)\gamma  =\ & \gamma x\gamma ^2-x\gamma \ =\ \gamma x(t \gamma -1)-x\gamma \ =\ t \gamma x\gamma  - \gamma x - x\gamma  \\ =\ & (t \gamma -1)x\gamma  - \gamma x\ =\ \gamma ^2x\gamma  - \gamma x\ =\ \gamma (\gamma x\gamma -x).\end{align*}
Thus $\gamma x\gamma  - x \in \mathrm{C}_{\O}(\gamma )$, the centralizer of $\gamma $ in $\O$.
A straightforward calculation shows that for $k\geq 2$
\[\gamma ^kx\gamma ^k - x = t \gamma ^{k-2} (\gamma x\gamma  - x) \gamma ^{k-1} + \gamma ^{k-2}x\gamma ^{k-2} - x, \]
and hence, by induction on $k$, $\gamma ^kx\gamma ^k - x \in \mathrm{C}_{\O}(\gamma )$, for every $k$.
If $\beta$ has order $2$, then $\beta = -1$ and $\Z[\beta] = \Z$. The generators in (1) are $0$ and $M \subseteq \mathbb{Z}[\beta]=\Z$. Hence $\O/M$ is of rank at least $3$ and a basis of $\O$ can only contain one unit. So the inequality $$\rank_{\Z} \O - \operatorname{inv}_{\O} = 4-1 = 3 \leq \rank_{\Z}(\O/M, +) = \rank_{\Z} \E_2(\O)^{ab},$$ holds.

If $\beta$ has order larger than $2$, then $\beta$ necessarily is not central. As $\Z \subsetneq \mathbb{Z}[\beta] \subseteq \mathrm{C}_\O(\beta)$ we obtain in this case that $\operatorname{rank}_\mathbb{Z} \mathrm{C}_{\O}(\beta) = 2$ (else by tensoring up with $\Q$ this would mean that $\beta$ is central in $D$).
Furthermore, by the above, $\alpha x\alpha - x \in \mathrm{C}_\O(\beta)$ for every $\alpha \in \mathcal{U}(\O)$ and every $x \in \O$.
So we get that all generators of $M$ are contained in $\mathrm{C}_\O(\beta)$.
Hence $\O/M$ maps surjectively onto $\O/\mathrm{C}_{\O}(\beta)$ and therefore is of rank at least $2$. Altogether, $$\rank_{\Z} \O - \operatorname{inv}_{\O} = 4-2 = 2 \leq \rank_{\Z}(\O/M, +) = \rank_{\Z} \E_2(\O)^{ab},$$ showing the inequality in the last case.

Now we prove that the statements $(a)-(e)$ are equivalent. To start, remark that $(a)$ and $(c)$ are equivalent due to formula \eqref{rankformula}. Hence it remains to prove that \eqref{it:specificO}, \eqref{it:basis}, \eqref{it:subring} and \eqref{it:submodule} are equivalent.
First, for an order $\O$ in an imaginary quadratic number field $\mathcal{U}(\O) = \langle - 1\rangle$, unless $\O \in  \{\mathcal{I}_1, \mathcal{I}_3 \}$ (e.g. see \cite[remark after Th. 240]{HardyWright}). In those cases $\mathcal{U}(\mathcal{I}_1) = \langle i\rangle$ and $\mathcal{U}(\mathcal{I}_3) = \langle - \zeta_3\rangle$, respectively. This implies that the last four conditions are equivalent in the case of orders in number fields with a finite unit group. Second assume that $\O$ is an order in a totally definite quaternion algebra with center $\mathbb{Q}$ and suppose $\O$ is isomorphic to $\LL, \O_2$ or $\O_3$.
In all three cases there exists a basis consisting of units of $\O$ given in \eqref{eq:basis_of_quat_orders}. Hence \eqref{it:specificO} implies \eqref{it:basis}. Clearly \eqref{it:basis} implies \eqref{it:subring} which implies \eqref{it:submodule}. For \eqref{it:submodule} implies \eqref{it:specificO} note that if $\O$ is an order in a totally definite quaternion algebra with center $\mathbb{Q}$ not isomorphic to $\LL$, $\O_2$ or $\O_3$, then, by \cite[Theorem~11.5.12]{Voight}, $\U(\O)$ is cyclic, generated by $\beta$, say. But then $\U(\O)$ is contained in the commutative subring $\mathbb{Z}[\beta] \subseteq \O$, which has $\mathbb{Z}$-rank at most $2$, since $D$ is a quaternion algebra.
\end{proof}

\begin{remark}\label[remark]{basis of units or small cyclic unit group}
One can filter from the proof of the previous theorem that when $\O$ is an order in a finite dimensional division $\Q$-algebra $D$ with $\U(\O)$ finite, then either $\O$ has a $\Z$-basis consisting of units or $|U(\O)| \in \{2,4,6\}$.
\end{remark}

We can also describe now the abelianization of $\GL_2(\O)$ with $\O$ the norm Euclidean maximal orders in quaternion algebras introduced in table (\ref{eq:basis_of_quat_orders}).
Note that an element $x$ in such an order $\O$ is a unit if and only if $\operatorname{N}(x) = x\bar{x} \in \mathbb{Z}_{\geq 0}$ equals $1$. Then it is not hard to find the units from the description of the orders given in \eqref{eq:basis_of_quat_orders}. If we set $\omega_2 = \frac{1+i+j+k}{2}\in \O_2$, $\omega_3 = \frac{1+j}{2} \in \O_3$ and $\omega_5 =  \frac{2+i-k}{4} \in \O_5$, then we have
\begin{align} 
\label[equation]{eq:units_of_quat_orders}
\begin{split}
 \U(\O_2) & = \langle i, \omega_2 \rangle \cong \SL(2,3) \cong Q_8 \rtimes C_3, \\
\U(\O_3) & = \langle i, \omega_3 \rangle \cong C_3 \rtimes C_4, \\
\U(\O_5) & = \langle \omega_5 \rangle \cong C_6.
\end{split}
\end{align}

\begin{corollary}\label[corollary]{prop:GE_2_O_ab_quat} $\GL_2(\O_2)^{ab} \cong C_3$,  $\GL_2(\O_3)^{ab} \cong C_4$ and  $\GL_2(\O_5)^{ab} \cong C_6$.
\end{corollary}

\begin{proof} Since $\O_2, \O_3$ and $\O_5$ are Euclidean, they are  $\GE_2$-rings by \Cref{euclidian are GE2}. Now $\omega_2 \in \U(\O_2)$, $\omega_3 \in \U(\O_3)$ and $\omega_5 \in \U(\O_5)$ are elements of order $6$. Hence, by \Cref{prop:GE_2_O_ab}, for $\O$ one of the three orders, $\GL_2(\O)^{ab} = \GE_2(\O)^{ab} \cong \U(\O)^{ab}$. \end{proof}

\section{Property \HFA[n-2] and \HFR for $\E_n(R)$ if $n\geq 3$} \label{sectie hogere graad matrices FA}

In this section we discuss properties \FA and \FR for the groups $\E_n(R)$, where $R$ is a {\it unital ring} and $n \geq 3$.
We prove fixed point properties on higher-dimensional CAT($0$) cell complexes for the Steinberg groups $\St_n(R)$, where $R$ is a finitely generated unital ring.
This will eventually imply the respective properties for $\E_n(R)$.

For $ i \neq j$, let $e_{ij}(r)$ denote the matrix, called {\it elementary matrix}, in $\GL_n(R)$ having $1$ on the diagonal and $r$ in the $(i,j)$-entry.
Recall that $\E_n(R) = \langle e_{ij}(r) \mid 1\leq i\neq j \leq n, r \in R\rangle$ denotes the elementary subgroup of $\GL_n(R)$.
In case $n \geq 3$ it will turn out that the elementary matrices $\E_n(R)$ over a finitely generated ring do not only have global fixed points on simplicial trees but also on `higher dimensional trees'.
More precisely they will have property \FA[n-2] (in the sense of \cite{Farb}).

\begin{definition}
A group $\Gamma$ is said to have \emph{property \FA[n]} if any isometric action, without inversion, on an $n$-dimensional CAT(0) cell complex has a fixed point.
\end{definition}

For definitions and a more in-depth discussion of CAT(0) spaces and cell complexes, we refer the reader to \cite[Chapter II]{BriHae}.
This definition is indeed a generalization of \FA since a simplicial tree is exactly a $1$-dimensional CAT(0) cell complex.
As such, \FA and \FA[1] are the same property.
Similar to the classical notation, we will say a group has \emph{property \HFA[n]} if every finite index subgroup has \FA[n].
Note that if a group has property \FA[n] for an $n \in \N$, then it has \FA[m] for every $n>m \in \Z_{\geq 1}$. 

In \cite[Theorem 1.2]{Ye} Ye proved that, for a finitely generated ring $R$ and $n\geq 3$, $\E_n(R)$ has property \FA[n-2] and in \cite[Theorem 1.1]{ErsJai} Ershov and Jaikin-Zapirain proved that it also has property \T, which we know to imply \FR (see \cite[Chapter~6., Proposition~11]{dlHV}).
The purpose of this section is to prove the following result.
 
\begin{theorem}\label{theo:E_n(R)_has_FR}
Let $n \geq 3$. Let $R$ be a unital ring which is finitely generated as $\Z$-module, then the group $\E_n(R)^{(m)}$ satisfies property \FR and \FA[n-2] for each $m \geq 1$.
\end{theorem}

The groups $\E_n(R)^{(m)}$ are subgroups of $\E_n(R)$ that will suit our purposes to study hereditary fixed point properties. They are defined below, just before \Cref{properties En(R)}.
Actually we will consider the so-called Steinberg groups $\St_n(R)$ and prove in \Cref{th: elementary matrices FR} (and the remark thereafter) the above statement for these groups.
The construction of $\St_n(R)$ is such that it maps onto $\E_n(R)$ and hence, since property \FR and \FA[n-2] are preserved under quotients, $\E_n(R)$ will inherit these properties from $\St_n(R)$.
From now on, throughout this section we assume $n \geq 3$. 

Straightforward calculations show that over any ring the elementary matrices satisfy the following relations (where $(a,b)= a^{-1} b^{-1} a b$ is the multiplicative commutator). 

\begin{lemma}\label{commutators in En}
Let $R$ be a ring. Then in $\E_n(R)$ we have that 
$$\left( e_{kl}(s), e_{ij}(r) \right) = \left\lbrace 
\begin{array}{ll}
1 & \mbox{ if } j \neq k \mbox{ and } i \neq l ,\\
e_{il}(-rs) & \mbox{ if } j = k \mbox{ and } i \neq l, \\
e_{kj}(sr) & \mbox{ if } j \neq k \mbox{ and } i = l ,\\
\end{array} \right.$$
for $s,r \in R$ and $1 \leq i,j, k,l \leq n$ with $i \neq j$, $k \neq l$ and $|\{ i,j,k,l\}|>2$.
\end{lemma}

In general $\E_n(R)$ may satisfy more relations as those above and this deficiency can be quantified by introducing a kind of 'free model of $\E_n(R)$'. 

\begin{definition}
Let $n \geq 3$ and $J$ an ideal in $R$. The \emph{Steinberg group} $\St_n(J)$ is the abstract group generated by the symbols $\{ x_{ij}(r) \mid  1\leq i\neq j \leq n, r \in J \}$ subject to the following relations:
\begin{align*}
x_{ij}(r)x_{ij}(s) &= x_{ij}(r+s),\\
\left( x_{ij}(r),x_{kl}(s)\right) &= 1 \text{ if } j\neq k \text{ and } i \neq l, \\
\left( x_{ij}(r),x_{jk}(s)\right) &= x_{ik}(rs) \text{ for $i,j,k$ pairwise different},\\
\left( x_{ij}(r),x_{ki}(s)\right) &= x_{kj}(-sr) \text{ for $i,j,k$ pairwise different}.
\end{align*}
The indices will always be taken modulo $n$.
\end{definition}

Clearly there is a natural epimorphism $\pi_n: \St_n(J) \rightarrow \E_n(J)$ defined by $\pi_n(x_{ij}(r)) = e_{ij}(r)$ and $\ker (\pi_n)$ measures 'how many' relations essentially different from those in \Cref{commutators in En} $\E_n(J)$ satisfies.

The proof of the version of \Cref{theo:E_n(R)_has_FR} for the Steinberg groups consists in obtaining a 'nice' generating set in the sense of \cite[Theorem 5.1]{Farb}.
Therefore we start now with providing a first smaller generating set.
We will use the left-normed convention for the iterated commutator, i.e. inductively we define $(a_1, a_2, \ldots , a_n) := ((a_1, a_2, \ldots, a_{n-1}), a_n)$.

\begin{lemma} \label{lemma generating set}
Let $n \geq 3$. Let $J$ be an ideal in $R$ and let $T_J$ and $T$ be a set of ring generators of $J$ and $R$, respectively. Then we have the following.
\begin{enumerate}
\item $\St_n(J) = \langle x_{ij}(t) \mid 1 \leq i \neq j \leq n, t\in T_J \rangle$.
\item If $T$ contains $1$ or generates $R$ as a $\Z$-module, then $$\St_n(R) = \langle x_{i, i+1}(r) \mid r \in T, 1 \leq i \leq n \rangle.$$
\item $\St_n(R)$ is a perfect group.
\end{enumerate}
\end{lemma}
\begin{proof} We first prove statement (1). 
Assume to begin with that $J$ is generated as $\Z$-module by $T_J$. Let $r \in J$ be an arbitrary element and write $r = \displaystyle\sum\limits_{s = 1}^k a_s t_s$ for $a_s \in \Z \setminus \{0\}$ and $t_s \in T_J$. Then clearly $x_{ij}(r) = x_{ij}(t_1)^{a_1} \ldots x_{ij}(t_k)^{a_k}$, proving the first part.

Since $T_J$ generates $J$ as ring, the set $\mathcal{T}_J$ consisting of finite products of elements of $T_J$ generates $J$ as $\Z$-module. By using the defining relations of $\St_n(J)$ and the case considered in the previous paragraph, we get $$ \St_n(J) = \langle x_{ij}(t) \mid 1 \leq i\neq j \leq n, t \in \mathcal{T}_{J} \rangle \leq \langle x_{ij}(t) \mid 1 \leq i\neq j \leq n, t \in T_{J} \rangle \leq \St_n(J).$$

To prove (2), we first assume that $ 1 \in T$. Let $S = \{ x_{i, i+1}(r) \mid r \in T, 1 \leq i \leq n \}$. Recall that the indices are taken modulo $n$.
According to (1), it suffices to show that $S$ generates the $x_{ij}(t)$ for every $t \in T$.
This is similar to what we did earlier:
$$x_{i,j}(t) = (x_{i,i+1}(t), x_{i+1,i+2}(1), \ldots , x_{j-1, j}(1) ), $$ an iterated commutator of elements of $S$.
Now assume that $T$ is an arbitrary generating set for $R$ as $\Z$-module. Similar arguments as above can be used to express $x_{i,i+1}(1)$ as a product of the elements in $S$. Thus the previous argument can be applied. 

Finally (3) follows immediately from the third defining relation of $\St_n(R)$.
\end{proof}

Let $T$ be a generating set for a ring $R$ as a $\Z$-module. We work with the following subgroups of $\St_n(R)$, for $m \in \Z_{\geq 1}$, $$\St_n(R)^{(m)} := \langle x_{i,i+1}(r)^m \mid r \in T, 1 \leq i \leq n \rangle = \langle x_{i,i+1}(r) \mid r \in mT, 1 \leq i \leq n \rangle.$$
We will show (in \Cref{properties En(R)}) that this subgroup is well-defined, i.e.\ independent of the generating set $T$. Unfortunately if $T$ is a set of ring generators of $R$ the definition would in general depend on $T$. Note that $\St_n(R) =  \St_n(R)^{(1)}$. The groups $\E_n(R)^{(m)}$ are analogously defined.

\begin{lemma} \label{properties En(R)}
Let $n \geq 3$ and $m$ a non-zero positive integer. Then
\begin{enumerate}
\item the group $\St_n(R)^{(m)}$ is well-defined, i.e. independent of the generating set of $R$ as additive group,
\item $\St_n(m^{n-1}R) \leq \St_n(R)^{(m)\prime} \leq \St_n(R)^{(m)} \leq \St_n(mR)$.
\end{enumerate}
\end{lemma}
\begin{proof} 
Let $T$ and $\tilde{T}$ be generating sets for $R$ as $\Z$-module. Similar to the proof of \Cref{lemma generating set}, it is clear that every element $x_{i,i+1}(r)^m=x_{i,i+1}(mr)$ for $r$ in $\tilde{T}$ can be made from the elements $x_{i,i+1}(t)^m=x_{i,i+1}(mt)$ where $t \in T$, and vice versa. This proves (1).

For (2), note that the second inclusion is trivial. The rightmost inclusion follows from \Cref{lemma generating set} applied to the ideal $mR$ and generating set $mT$. From the same lemma it also follows that $\St_n(m^{n-1}R) = \langle x_{ij} (m^{n-1}r) \mid 1 \leq i \neq j \leq n, r \in T \rangle$. Using the defining relations as before, we obtain
\begin{align*}
x_{i,i+k}(m^{n-1} r)\ =\ (x_{i,i+1}(mr), x_{i+1, i+2}(m),\ldots , x_{i+k-2, i+k-1}(m), x_{i+k-1, i+k}(m^{n-k})).
\end{align*} 
So, the elements that generate $\St_n(m^{n-1}R)$ can be constructed from the generators $x_{i,i+1}(r)^m = x_{i,i+1}(mr)$ of $\St_n(R)^{(m)}$ by taking commutators, which proves the remaining inclusion.
\end{proof}

\begin{remark}\label{E_n(R) commutators}
The proof of \Cref{properties En(R)} only uses the relations from \Cref{commutators in En} and hence the corresponding statements also hold for $\E_n(R)$.
\end{remark}

We now have the necessary ingredients to prove the following fixed point properties for $\St_n(R)^{(m)}$. We were only recently informed that property \FA[n-2] had already been proven for the group $\E_n(R)$ in \cite[Theorem 2.1]{Ye} in case $R$ is a finitely generated ring. In hindsight both proofs follow the same line and use \cite[Theorem 5.1]{Farb}. Note that the group $\St_n(R)^{(m)}$ is not the group generated by all the $m$th powers of the generators of the Steinberg group, but rather is a suitably chosen subgroup of the latter in order to be able to use \cite[Theorem 5.1]{Farb} and still derive the desired hereditary property.

\begin{theorem}\label{th: elementary matrices FR}
Let $n \geq 3$. Suppose $R$ is finitely generated as $\Z$-module. Then the group $\St_n(R)^{(m)}$ satisfies properties \FR and \FA[n-2] for each $m \geq 1$.
\end{theorem}
\begin{proof}
Let $T$ be a finite generating set of $R$ as a $\Z$-module which we assume to contain $1$.

By definition, $\St_n(R)^{(m)}$ is finitely generated by $S= \{ x_{i, i+1}(r) \mid r \in mT, 1 \leq i \leq n \}$.
By a theorem of Farb \cite[Theorem 5.1]{Farb}, to prove \FA[n-2] it suffices to find a set of finitely generated nilpotent subgroups $C:=\{ \Gamma_1, \ldots , \Gamma_n\}$ of $\St_n(R)^{(m)}$ such that 
\begin{enumerate}
\item the group generated by the subgroups in $C$ is of finite index in $\St_n(R)^{(m)}$,
\item any proper subset of $C$ generates a nilpotent group,
\item there exists a positive integer $z$ such that for all $1\leq i \leq n$ and for all $r \in \Gamma_i$, there exists a nilpotent subgroup $N \leq \St_n(R)^{(m)}$ with $r^z \in N^{\prime}$.
\end{enumerate}

We define these groups to be the finitely generated abelian groups
$$\Gamma_i = \langle x_{i,i+1}(mr) \mid r \in T \rangle, \quad \text{ for } 1 \leq i \leq n.$$
Clearly $\langle \Gamma_1, \ldots , \Gamma_n \rangle = \St_n(R)^{(m)}$, so the first requirement is satisfied.

Let now $\widehat{\Gamma_i}$ be the group generated by the subgroups of $C \setminus \{\Gamma_i\}$. To prove (2), it is sufficient to prove that each $\widehat{\Gamma_i}$ is nilpotent. Clearly $\pi: \St_n(R) \rightarrow \St_n(R), x_{ij}(r) \mapsto x_{i+1, j+1}(r)$ is an isomorphism such that $\pi\left(\Gamma_i \right) = \Gamma_{i+1}$ and  $\pi\left(\widehat{\Gamma_i }\right) = \widehat{\Gamma_{i+1}}$. Hence all $\widehat{\Gamma_i}$ are isomorphic to $\widehat{\Gamma_n}$. It is well known (see for example \cite[Lemma 4.2.3]{Rosenberg}) that $\widehat{\Gamma_n}$ is nilpotent. Hence the second requirement is satisfied.

We will show the last requirement for $r$ a generator of $\Gamma_i$ and $z=m$.
This is sufficient, since the $\Gamma_i$ are finitely generated abelian groups.
Consider $x_{i,i+1}(mt)^m = x_{i,i+1}(m^2t)$ with $t \in T$.
Applying $\pi^{2-i}$ to this element, it suffices to show the last requirement for $x_{2,3}(mt)^m = x_{2,3}(m^2t)$.

Using the defining relations we write $$x_{2,3}(mt)^m =x_{2,3}(m^2t) = (x_{2,1}(m),x_{1,3}(mt)).$$ Now applying the isomorphism of $\St_n(mR)$ which interchanges in $\St_n(mR)$ the indices $1$ and $2$, $x_{2,1}(m)$ and $x_{1,3}(mt)$ are mapped to elements of $\widehat{\Gamma_n}$, a nilpotent group, proving the statement. Here we used that we may assume $1 \in T$ and thus $m$ and $mr \in mT$. Hence conditions (1) to (3) are satisfied and we conclude that $\St_n(R)^{(m)}$ has property \FA[n-2].

To prove that $\St_n(R)^{(m)}$ has property \FR, we will check that for every pair of generators $x_{i,i+1}(s)$ and $x_{j,j+1}(r)$ in $S$, their commutator $\left( x_{i,i+1}(s), x_{j,j+1}(r) \right)$ commutes with $x_{j,j+1}(r)$.
This will indeed suffice, by a result of Culler and Vogtmann \cite[Corollary 2.5]{CulVog} since $\St_n(R)^{(m)}$ already has finite abelianization (recall that it has property \FA[n-2] by the first part of the proof). 

First, if $j \neq i+1 $, $\left( x_{i,i+1}(s), x_{j,j+1}(r) \right) = 1$ which of course commutes with $x_{j,j+1}(r)$. 
So suppose now $j = i+1$, then $\left( x_{i,i+1}(s), x_{i+1,i+2}(r) \right) = x_{i,i+2}(sr)$  (here we used that $i \neq i+2$, or the fact that $n \neq 2$) which commutes with $x_{i+1,i+2}(r)$, proving the theorem.
\end{proof}

\begin{remark}
As a matter of fact, the reasoning in the previous proof also provides an alternative and elementary proof for the fact that $\St_n(R)$ has \FR when $R$ is finitely generated as a unital ring. Indeed, by taking $m=1$ in the proof of \FR, we may provide the same argument when $T$ generates $R$ as a ring. This proof circumvents the use of the much more general result \cite[Theorem 6.2]{ErsJai} which states that $\St_n(R)$ satisfies property \T.
\end{remark}

As explained earlier, since property $\FA[n]$ and $\FR$ are preserved under quotients,  \Cref{theo:E_n(R)_has_FR} now follows from the previous theorem and remark.

\section{Property \FR and \HFR for $\E_2(\O)$}\label{sectie FA voor GL2's}

In this section we discuss properties \FA,  \FR and \HFR for the groups $\E_2(R)$, where $R$ is a suitable ring (which will always at least again be associative and unital).
Since not every $\E_2(R)$ has property \FA,
the situation is significantly different from the case $n \geq 3$ which was the setting of the previous section.
For $R$ an order in a simple $\Q$-algebra having a finite unit group, we classify exactly when $\E_2(R)$ has property \FA and \FR. With a view on the latter we consider first Borel type subgroups.

It is well known when $\SL_2(\I)$, for $\I$ a $\Z$-order in a field with finite unit group, has property \FA.
Indeed, by \cite[Exercise I.6.5, pg 66]{Serre} and \cite[Theorems 2.1 and 2.4]{FroFin} the only such $\I$ for which $\SL_2(\I)$ has \FR (or equivalently \FA) is $\I_3$.

The main goal of this section is to generalize this result to all orders $\O$ (not necessarily commutative) in division $\Q$-algebras with $\U(\O)$ finite.
We will prove the following theorem, which is also \Cref{TheoremD} from the introduction. 

\begin{theorem} \label{When is E2 FA and some GL2 FA}
Let $\O$ be an order in a finite dimensional division $\Q$-algebra with $\U(\O)$ finite. Then the following properties are equivalent:
\begin{enumerate}
\item\label{it_prop_forE2O_FR} $\E_2(\O)$ has property \FR,
\item\label{it_prop_forE2O_FA} $\E_2(\O)$ has property \FA,
\item\label{it_prop_forE2O_O} $\O$ is isomorphic to $\I_3, \O_2$ or $\O_3$.
\end{enumerate}
Furthermore, $\GL_2(\O)$ has  property \FR if $\O$ has a basis of units and $\O \ncong \Z$. 
\end{theorem}

The proof of \Cref{When is E2 FA and some GL2 FA} will be given later in Subsection~\ref{subsectie FR voor GRK} (on page~\pageref{proof_of_When is E2 FA and some GL2 FA}) and will strongly require the results obtained in  Subsection~\ref{sectie ab van E2}.
Moreover we first need to understand the connections between $\E_2(\O)$, the diagonal matrices therein and the Borel subgroup.
The latter will be the content of \Cref{property FA for the GRK groups}.
Next, in Subsection~\ref{subsectie FR voor Borel}, we conjecture when $\GE_2(\O)$ has property \FA and \FR and lay the first stone towards a proof by understanding completely the situation for the Borel subgroup $\B_2(\O)$ of $\GL_2(\O)$.

\begin{remark}
While in \Cref{sectie hogere graad matrices FA}, property \FR for $\E_n(\O)$, $n \geq 3$, is a consequence of the same property for the Steinberg groups $\St_n(\O)$, this is no longer true for the cases in \Cref{When is E2 FA and some GL2 FA}. 
Indeed, if one defines $\St_2(\O)$ in a similar way, then the only non-trivial defining relation is $x_{ij}(r)x_{ij}(s) = x_{ij}(r+s)$, hence $\St_2(\O)$ is the free product of two copies of the additive group of $\O$ and hence cannot have property \FR.
\end{remark}

\subsection{Property \FR for the groups $G_{R,K}$ with applications to \FR for $\E_2(\O)$} \label[subsection]{subsectie FR voor GRK}

We will now investigate $\E_2(R)$ and $\GE_2(R)$ simultaneously by defining a more general type of groups, denoted $G_{R,K}$. Consider a subgroup $K$ of $\D_2(R)$ (the group of invertible diagonal $2 \times 2$-matrices over the ring $R$; recall that we always assume our rings to be unital). 

\begin{definition}
The group generated by $K$ and $N = \sm{1 & R \\ 0 & 1}$ consisting of the unimodular upper triangular matrices will be denoted by $G_{R,K}$. 
\end{definition}

Note that for the choice $K = \D_2(R)$ we have that $G_{R,K}$ is the \emph{Borel subgroup} \emph{$\B_2(R)$} of $\GL_2(R)$, i.e. the group consisting of invertible upper triangular $2\times 2$-matrices over $R$.

\begin{notation*}
 If $K$ consists of the matrices of the form $\sm{
\alpha & 0 \\
0 & \beta 
 }$ with $\alpha \beta \in \U(R)'$ we will instead use the notation $\DE_2(R)$ for $K$ and the notation $\BE_2(R)$ for $G_{R,K}$.
\end{notation*}

If $R$ is almost-universal, using the determinant like map $\varphi$ defined in \eqref{eq:varphi}, one can check that 
$$\BE_2(R) = \B_2(R) \cap  \E_2(R).$$
Indeed, if one restricts $\varphi$ to the subgroup $\B_2(R)$, then its kernel coincides with $\BE_2(R) = \B_2(R) \cap  \E_2(R)$. Also $\DE_2(R)$ equals $\langle\ D(\mu)\ | \ \mu \in \U(R)\ \rangle$ by \eqref{R5} on page~\pageref{R5}, recall that $D(\mu)$ denoted the diagonal matrix $[\mu, \mu^{-1}]$. Note that the group $\DE_2(R)$ already appeared in \Cref{intersection D_2 and E_2}.

Now note that $N = \sm{1 & R \\ 0 & 1}$ is a normal subgroup of $G_{R,K}$.  Thus we have the following split short exact sequence
 \begin{equation} \label[equation]{G -R -K exact sequence}
 1 \longrightarrow N \longrightarrow G_{R,K} \longrightarrow G_{R,K}/N \cong K \longrightarrow 1 .
 \end{equation}
Hence $G_{R,K}$ is isomorphic to the semi-direct product $N \rtimes_\alpha K$, where $\alpha \colon K \to \Aut(N)$ and 
$$\alpha([u_1, u_2]) : N \rightarrow N : n \mapsto [u_1^{-1}, u_2^{-1}]\,  n \, [u_1, u_2]$$ 
is conjugation by $[u_1, u_2]$. Furthermore, $N$ is isomorphic to the additive group of $R$ and hence abelian. 

{\it Remark} that $\alpha(\lambda)$ for $\lambda \in K$, while in general an automorphism of the abstract group $N$, can here be considered as a matrix over $\Z$ since $(R, +) \cong N$ is a finitely generated free $\Z$-module by choosing an arbitrary basis of $R$. Via this identification one may speak of the {\it eigenvalues of $\alpha(\lambda)$ }. Note that its eigenvalues are independent of the chosen basis.

Property \FA of extensions (\ref{G -R -K exact sequence}) has been considered by Serre \cite[I.6.5., Exercise 4]{Serre} and Cornulier-Kar \cite[Proposition 3.2]{CorKar} by means of sufficient group theoretical restrictions on $G_{R,K}$. We will now provide a linear algebra criterion which will turn out to be easy to check in our setting. 

\begin{proposition}\label[proposition]{property FA for the GRK groups}
Let $R$ be a ring, which is finitely generated and free as $\Z$-module. Then the following properties hold:
\begin{enumerate}
\item If $K$ is countable and has property \FR (resp. \FA) and there exists  $\lambda \in K$ with finite order such that $\alpha(\lambda)$ (where $\alpha$ was defined above)  has only non-rational eigenvalues, then $G_{R,K}$ has property \FR (resp. \FA).\label{property FA for the GRK groups(1)}
\item If $G_{R,K}$ has property \FR (resp. \FA), then $K$ has property \FR (resp. \FA).\label{property FA for the GRK groups(2)}
\item Suppose $R$ has a $\Z$-module basis consisting of units and $\DE_2(R) \leq K$. If $G_{R,K}$ has property \FR (resp. \FA), then also $\langle \E_2(R), K \rangle$ has property \FR (resp. \FA).\label{property FA for the GRK groups(3)}
\end{enumerate}
\end{proposition}

We will first need the following lemma which is inspired by \cite[I.6.5., Exercise 4]{Serre}. The exercise is about simplicial trees. We state the lemma for real trees and for the sake of completeness we provide a proof. 

\begin{lemma} \label[lemma]{lemmaNilpotent}
Let $B$ be a finitely generated group and $N \unlhd B$ nilpotent and finitely generated. Suppose there is no subgroup $M$ of $N$ that is normal in $B$ and such that $N/M \cong \Z$. Then $B$ has property \FR if $B/N$ has property \FR.
\end{lemma}

\begin{proof}
We will in fact show that if $B$ acts on a real tree $X$, then $N$ has a fixed point on this tree. This implies of course that $B$ has property \FR.

Clearly, if $B$ acts on a real tree $X$, then $N$ does so as well. Now from \cite[Proposition 3.8]{CulMor} it follows that exactly one of the following happens:
\begin{itemize}
\item the action of $N$ on $X$ has a fixed point,
\item there exists a {\it unique} line $T$ in $X$, stable under the action of $N$, on which $N$ acts by translation.
\end{itemize}
Suppose that the latter happens. Then we have a non-trivial morphism $\varphi: N \rightarrow \Aut(T) \cong $ \emph{Iso}$(\R)$.
Then $T$ is also stable under the action of $B$. Indeed let $t \in T$, the invariant tree for $N$, and $g \in B$. Now $gt \in gT$ and for every $n \in N$ it holds that $n (gt) = (ng) t = (gn') t = g (n't) \in gT$ (for some $n' \in N$), hence $gT$ is invariant under $N$ and thus by the uniqueness $gT = T$, as needed. 

We may thus extend the morphism above to a morphism $\varphi \colon B \rightarrow \Aut(T) \cong$ Iso$(\R)$.
Since $B$ is finitely generated and Iso$(\R)$ consists of reflections and translations, it is easy to see that $\varphi(B)$ is isomorphic to $\left( \Z^n \right)\rtimes C_2$ or $\Z^n$, for some $n\in \mathbb{Z}_{\geq 1}$. Indeed, if the finite number of generators for $\varphi(B)$ are all translations, clearly $\varphi(B) \cong \Z^n$. If some of them are reflections, since a product of two reflections is a translation, one may change the generating set to only contain translations and $1$ reflection. This reflection acts by inversion on the translations, so $\varphi(B) \cong \Z^n \rtimes C_2$ in this case. This also implies that every subgroup of $\Z^n$, the subgroup generated by translations, is normal in $\varphi(B)$.
Moreover, since $N$ is nilpotent and acts via translation on $T$, $\varphi(N) \cong \Z^k$ for some $n\geq k \in \mathbb{Z}_{\geq 1}$.
All this implies that we may compose $\varphi|^{\varphi(B)}$, the corestriction of $\varphi$ to $\varphi(B)$, with another morphism to obtain $\psi \colon B \rightarrow \Z \rtimes C_2$ such that $\psi(N) \cong \Z$ (for example, by modding out all the components of $\Z^n$ except for exactly one which has non-zero intersection with $\varphi(N)$).
As such, there exists a normal subgroup $H$ of $B$ for which $N/(H\cap N) \cong \Z$, a contradiction with the assumptions.
\end{proof}

\begin{proof}[Proof of \Cref{property FA for the GRK groups}]
The second statement immediately follows from the short exact sequence (\ref{G -R -K exact sequence}) and \Cref{FAconserved}.

Assume now that $K$ has property \FR and that there exists $\lambda \in K$ such as in the first statement. To prove that $G_{R,K}$ has property \FR, we verify the conditions of \Cref{lemmaNilpotent}. Assume that $N = \sm{1 & R \\ 0 & 1} \leq G_{R,K}$ has a subgroup $M$, normal in $G_{R,K}$ such that $N/M \cong \Z$. Take $H = \langle \lambda \rangle \leq K$. Then we may restrict $\alpha$ to $H$ and consider $\Q [N] := \Q \otimes_\Z N$ as a $\Q H$-module.
The subgroup $M$, being normal in $G_{R,K}$, is invariant under the action of the restriction of $\alpha$. 
Thus under this identification $\Q [M] := \Q \otimes_\Z M$ is a $\Q H$-submodule of $\Q [N]$. 
Since $H$ is finite, by Maschke's Theorem, $\Q [M]$ has to have a complement, i.e.\ there is a $\Q H$-submodule $V$ of $\Q [N]$ such that $\Q [N] = \Q [M] \oplus V$ and then necessarily  $\dim_\Q V = 1$ (since $N/M \cong \Z$). 
This means in particular that each of the matrices corresponding to an $\alpha(\mu)$, $\mu \in H$, has to have a rational eigenvalue.
However, this is in contradiction with the assumptions. 

Now, since $G_{R,K}/N \cong K$ has property \FR, the first statement follows from \Cref{lemmaNilpotent} if we show that $G_{R,K}$ is finitely generated. For this we need to show that $K$ and $N$ are finitely generated. For the latter let $\{ r_1, \ldots, r_l \}$ be a finite $\Z$-basis of $R$, which exists by the assumptions on $R$, then $\{ \left( \begin{array}{lr} 1 & r_i \\ 0 & 1 \end{array}\right) \mid 1 \leq i \leq l \}$ is a finite generating set of $N$. Also $K$ is finitely generated due to \Cref{prop:iff_conditions_FA}.

Finally, assume that $R$ has a $\mathbb{Z}$-module basis $\mathcal{B}$ consisting of units and that $G_{R,K}$ satisfies property \FR.  Note that $\langle \E_2(R), K \rangle = \langle w , G_{R,K} \rangle$ where $w = \sm{
0 & -1 \\ 1 & 0
}$. Let $T$ be a tree and assume $\langle \E_2(R), K \rangle$ acts on it. Now due to \cite[I.6.5., Proposition 26]{Serre}\footnote{Note that Serre states \cite[I.6.5., Proposition 26]{Serre} for simplicial trees, but the proof stays exactly the same for real trees }, since $w$ is of finite order (hence $\langle w \rangle$ has property \FR) and $G_{R,K}$ has property \FR, $\langle w, G_{R,K} \rangle$ has property \FR if there exists a generating set $\mathcal{G}$ of $G_{R,K}$ such that $w x$ has a fixed point for all $x \in \mathcal{G}$. For this purpose define $x_{\mu} = \sm{ \mu^{-1} & 1 \\ 0 & \mu }$ for $\mu \in \U(R)$.
Then, $$\langle w, G_{R,K} \rangle = \langle w, K, N \rangle = \langle K, \E_2(R) \rangle = \langle w, x_{\mu}, K \mid \mu \in \mathcal{B} \rangle.$$
Indeed, $[\mu, \mu^{-1}]x_{\mu} = \sm{1 & \mu \\ 0 & 1}$ and $\left\{\sm{1 & \mu \\ 0 & 1} \colon \mu \in \mathcal{B} \right\}$ generates the subgroup $N$. It can easily be seen that the generating set $\mathcal{G}= \{ x_{\mu}, K \mid \mu \in \mathcal{B} \}$ satisfies the condition of \cite[I.6.5., Proposition 26]{Serre}. Indeed, \[wx_{\mu} = \left( \begin{matrix} 0 & -\mu \\ \mu^{-1} & 1 \end{matrix} \right), \qquad (wx_{\mu})^3 = \left( \begin{matrix} -1 & 0 \\ 0 & -1 \end{matrix} \right). \]
This implies that $wx_{\mu}$ is of order $6$ and hence has a fixed point.
Next take $d = [\alpha, \beta] \in K$. Note that also $[\beta, \alpha]= [\alpha, \beta].[\alpha^{-1}\beta, \beta^{-1} \alpha] \in K$, where we used that $\DE_2(R) \leq K$.  Consequently $(wd)^4 = \sm{(\beta\alpha)^2 & 0 \\ 0 & (\alpha\beta)^2}$ is an element of $K$ .
Since $G_{R,K}$ has \FR by assumption, $K$ does so as well by the second statement.
As now $(wd)^4 \in K$ has a fixed point on $T$, $wd$ needs to have a fixed point as well since the group $K$ does not act via inversions. More precisely, if $(wd)^4$ has a fixed point, but $(wd)^2$ does not, then this implies there is an action by inversion (see the paragraph before \Cref{defFA}), so $(wd)^2$ needs to have a fixed point and similarly $wd$ needs to have a fixed point. For property $\FR$ the claim follows in a similar way\footnote{If $(wd)^4$ has a fixed point $x$, then the midpoint $y$ between $x$ and $(wd)^2(x)$ is a fixed point for $(wd)^2$, and similarly the midpoint between $y$ and $(wd)(y)$ is a fixed point for $wd$.}.

Thus altogether we have proven that $\langle w, G_{R,K} \rangle = \langle \E_2(R), K \rangle$ has property \FR.  The result for \FA can be obtained similarly by taking $T$ a simplicial tree and using \cite[I.6.5., Exercise 4]{Serre} instead of \Cref{lemmaNilpotent} to prove \eqref{property FA for the GRK groups(1)}.
\end{proof}

\begin{remark}
Unfortunately, the converse of the third statement in \Cref{property FA for the GRK groups} is not true.
It fails already in the (trivial) case where $K = 1$.
For this, consider any ring $R$ for which $\E_2(R)$ has \FR (such as $\I_3$) and notice that $G_{R,1} \cong N$ is a finitely generated torsion-free abelian group.
Also the converse of the first statement in \Cref{property FA for the GRK groups} is not true as we will explain after \Cref{Borel FA cyclic units}.
\end{remark}

We are finally ready to prove the main theorem of this section. 

\begin{proof}[Proof of \Cref{When is E2 FA and some GL2 FA}]\label{proof_of_When is E2 FA and some GL2 FA}
Recall that property \FR implies property \FA which on its turn implies finite abelianization by \Cref{prop:iff_conditions_FA}. Thus thanks to \Cref{finite_abelianization_equiv} we need to understand for which isomorphism type $\{ \mathbb{Z}, \I_1, \I_3, \LL, \O_2, \O_3 \}$ of $\O$ the group $\E_2(\O)$ has property \FR or \FA. 

In the literature it was obtained that $\E_2(\Z)$ is isomorphic to the free product $\SL_2(\Z) \cong C_4 \ast_{C_2} C_6$ (see \cite[I, 4.2. (c)]{Serre}) and also that $\E_2(\I_1)$  and $\E_2(\LL)$ have non-trivial amalgamated products (see respectively \cite[Theorem 4.4.1]{FineBook} and \cite[Theorem~7.8]{amalgamationpaper}). Therefore using again \Cref{prop:iff_conditions_FA} we already know the implications $(1) \Rightarrow (2) \Rightarrow (3)$. 

Therefore it remains to prove that all the groups mentioned in the statement have property \FR. This will be achieved by verifying the conditions from \Cref{property FA for the GRK groups}. We will use the notations and results from \Cref{defofOrders} throughout this proof without always explicitly mentioning it. We first claim the following.

\noindent {\bf Claim: } The groups $\BE_2(\O_2)$, $\BE_2(\O_3)$, $\BE_2(\I_3)$, $\B_2(\LL)$ and $\B_2(\I_1)$ have property \FR as the eigenvalue condition from \Cref{property FA for the GRK groups}~\eqref{property FA for the GRK groups(1)} is fulfilled. 

Once the claim is established \Cref{property FA for the GRK groups}~\eqref{property FA for the GRK groups(3)} implies that also $\E_2(\O_2)$, $\E_2(\O_3)$, $\E_2(\I_3)$, $\GE_2(\LL)$ and $\GE_2(\I_1)$ have property \FR. In particular we would have proven \eqref{it_prop_forE2O_O} implies \eqref{it_prop_forE2O_FR}, as needed.

\noindent {\it Proof of the claim: } to check the condition of the existence of an element $\lambda \in K$ of finite order such that $\alpha(\lambda)$ has no rational eigenvalue (where $K = \D_2(\O)$ or $K = \DE_2(\O)$) it suffices to calculate the impact of it to a basis of $N= \sm{1 & \O \\ 0 & 1}$. By fixing a basis of the $\Z$-module $N$, we identify $\Aut(N) \cong \GL(2, \Z)$ or $\Aut(N) \cong \GL(4, \Z)$ respectively.

We will first carry out the proof for $\LL$, the Lipschitz quaternions and $\O_2$, the Hurwitz quaternions.
The Lipschitz quaternions have a basis $\{1, i, j, k\}$ consisting of units and $\U(\LL) \cong Q_8$.
Take $\lambda_1 = [i, 1] \in \D_2(\LL)$.
Then $\alpha(\lambda_1)$ is just left multiplication by $-i$ on $\LL$ and this has (complex) eigenvalues $i, -i$ both with multiplicity $2$, in particular it does not have any rational eigenvalue.
$\O_2$ has a basis $\{1,i,j,\omega\}$ consisting of units and $\U(\O_2) = \langle i, \omega \rangle \cong Q_8 \rtimes C_3$, so $\U(\O)^{\prime} = \langle i,j \rangle$ (see table (\ref{eq:basis_of_quat_orders}) and \eqref{eq:units_of_quat_orders}). If we set $\lambda_2 = [i, 1]$, then in this case even $\lambda_2 \in \DE_2(\O_2)$ and $\alpha(\lambda_2)$ has the same eigenvalues as $\alpha(\lambda_1)$, hence none of them is rational.

Also the rings of integers $\I_1$ in $\Q(\sqrt{-1})$ and $\I_3$ in $\Q(\sqrt{-3})$, considered as $\Z$-module, have a basis consisting of units. Indeed one can take  $\{ 1, i \}$ and $\{ 1, \frac{1 + \sqrt{-3}}{2} \}$ respectively.
Also here the non-rational eigenvalue condition is satisfied, but the matrices we use are $[i,1]$ and $[\frac{1 + \sqrt{-3}}{2},  \left(\frac{1 + \sqrt{-3}}{2}\right)^{-1} ]$.
They are both in $\D_2$ and in the last case even in $\DE_2$ of their respective orders.

Now consider the maximal order $\O_3$ in $\qa{-1}{-3}{\mathbb{Q}}$.
Take $\omega_3 = \frac{1+j}{2} \in \O_3$ and note that $\omega_3^6 = 1$.
Then $\O_3$ has a basis $\{1, i, \omega_3, i\omega_3\}$ consisting of units. Set $\tau = \omega_3^2$, a unit of order $3$, then $\U(\O_3) = \langle \tau, i \rangle \cong C_3 \rtimes C_4$ and $\U(\O_3)^{\prime} = \langle \tau \rangle$.
Then $\lambda_3 = [\tau, 1] \in \DE_2(\O_3)$ and $\alpha(\lambda_3)$ 
has eigenvalues $\zeta_3$ and $\zeta_3^2$, both with multiplicity $2$, where $\zeta_3$ denotes a complex primitive third root of unity.

So from \Cref{property FA for the GRK groups}~\eqref{property FA for the GRK groups(1)} the claim follows.

Finally, since $\E_2(\O_2)$, $\E_2(\O_3)$ and $\E_2(\I_3)$ are of finite index in the $\GE_2$ of the respective rings (see \Cref{theorem difference GE2 and E2}), by \Cref{FAconserved}, also the $\GE_2$'s of these orders have property \FR.
Since $\O_2$, $\O_3$, $\I_1$ and $\I_3$ are left Euclidean rings,  $\GL_2 = \GE_2$ by \Cref{euclidian are GE2}.
On the other hand, $\LL$ is neither right nor left Euclidean, but one can still directly prove it to be a $\GE_2$-ring (see \cite[Proposition 7.10]{amalgamationpaper}). Thus the last line of the statement follows by \Cref{finite_abelianization_equiv}.
\end{proof}

Next, we join all the pieces in order to proof that $\E_2(\O)$ always contains a subgroup of finite index not enjoying property \FA.

\begin{theorem}\label{When is E_2(O) HFA}
Let $\O$ be an order in a finite dimensional division $\Q$-algebra with $\mathcal{U}(\O)$ finite. Suppose $\O \ncong \O_3$. Then $\E_2(\O)$ does not satisfy property \HFA. In particular also $\GE_2(\O)$ does not satisfy \HFA. 
\end{theorem}
\begin{proof}
If $\E_2(\O)$ has property \FA, then by \Cref{When is E2 FA and some GL2 FA}, $\O$ is isomorphic to $\I_3$, $\O_2$ or $\O_3$. It remains to prove that $\E_2(\I_3)$ and $\E_2(\O_2)$ do not satisfy property \HFA. We will do this by exhibiting concrete subgroups of finite index not having property \FA. 

To start we claim that $\E_2(\Z[\sqrt{-3}])$ is a subgroup of finite index in $\E_2(\I_3)$ with infinite abelianization. Indeed $\Z[\sqrt{-3}]$ is a $\GE_2$-ring \cite{Dennis} and hence $\E_2(\Z[\sqrt{-3}]) = \SL_2(\Z[\sqrt{-3}])$ which is of finite index in $\SL_2(\I_3)= \E_2(\I_3)$ because $\GL_2(\Z[\sqrt{-3}])$ is of finite index in $\GL_2(\I_3)$ using that $\I_3$ is an Euclidean ring. By \Cref{finite_abelianization_equiv}, $\E_2(\Z[\sqrt{-3}])$ has infinite abelianization.

Finally in \cite[Theorem~7.8]{amalgamationpaper} it is proven that $\E_2(\LL)$ is a subgroup of finite index in $\E_2(\O_2)$ with a non-trivial decomposition as amalgamated product and thus does not have property \FA. 
\end{proof}

\begin{remark}\label[remark]{remarkFAb}
It is possible to prove the same statement as in \Cref{When is E_2(O) HFA} for the group $\SL_2(\O)$ for $\O$ an order in a finite dimensional $\Q$-algebra with $\U(\O)$ finite, via geometric methods. Indeed $\SL_2(\O)$ has a discontinuous action on the hyperbolic space $\mathbb{H}^3 $ or $\mathbb{H}^5$ of dimension $3$ or $5$. One can construct a reflection acting on this hyperbolic space, and a congruence subgroup $\Gamma$ of $\SL_2(\O)$ which is normalized by the latter reflection. Then by \cite[Corollary 3.6]{Lubot96}, $\Gamma$ has a virtually free quotient. The latter implies that $\Gamma$ has a finite index subgroup with infinite abelianization. As $\Gamma$ has finite index in $\SL_2(\O)$, this proves the result. Note that for this method, the order $\O_3$ does not have to be excluded. Moreover, as $\E_2(\O_3)$ has finite index in $\SL_2(\O_3)$, this also shows that the condition $\O \ncong \O_3$ is not necessary in \Cref{When is E_2(O) HFA}. However, note that $\E_2(\O)$ is not of finite index in $\SL_2(\O)$ for any order $\O$ such that $\U(\O)$ is finite, \cite{Nica}. Hence this remark does not yield an alternative proof of \Cref{When is E_2(O) HFA}.
\end{remark}

\subsection{Property \FR for the Borel subgroup with a view on $\GE_2(\O)$}\label[subsection]{subsectie FR voor Borel}

Now it is logical to ask, in the same setting as \Cref{When is E2 FA and some GL2 FA}, when $\GE_2(\O)$ has property \FA. We expect a similar theorem to be true.

\begin{question}\label[question]{question on property FR for GE2}
Let $\O$ be an order in a finite dimensional division $\Q$-algebra $D$ with $\U(\O)$ finite. Are the following properties equivalent?
\begin{enumerate}
\item $\GE_2(\O)$ has property \FR ,
\item $\GE_2(\O)$ has no non-trivial decomposition as an amalgamated product,
\item $\O$ isomorphic to $\I_1, \I_3, \LL, \O_2$ or $\O_3$.
\end{enumerate}
\end{question}
In case that $D$ is a field and $\O$ is not isomorphic to $\I_1$ and $\I_3$ it is proven in \cite{amalgamationpaper} that $\GE_2(\O)$ indeed has a non-trivial decomposition as an amalgamated product.
Hence combined with \Cref{When is E2 FA and some GL2 FA}, using that $\I_1$ and $\I_3$ are $\GE_2$-rings, we see that the above question is indeed true for $D$ a field.
For the general case, the missing fact is that $\GE_2(\O)$ for $\O$ an order in a totally definite quaternion algebra of the form $\qa{a}{b}{\Q}$ only has property \FA for the orders $\LL, \O_2$ and $\O_3$.
To achieve this, in view of the proof of \Cref{When is E2 FA and some GL2 FA}, it is natural to first fully understand the situation for $\B_2(\O)$.

\begin{proposition}\label[proposition]{Borel FA cyclic units}
Let $\O$ be an order in a finite dimensional division $\Q$-algebra $D$ with $\U (\O)$ finite. Then the following properties are equivalent:
\begin{enumerate}
\item $\B_2(\O)$ has property \FR,
\item $\B_2(\O)$ has property \FA,
\item $\B_2(\O)$ has no non-trivial decomposition as an amalgamated product,
\item $\U(\O) \ncong C_2$.
\end{enumerate} 
\end{proposition}
\begin{proof}
If $\B_2(\O)$ has property \FR, then it has also property \FA and consequently, by Serre's algebraic characterisation, it cannot be an amalgamated product.

Next, by contraposition, suppose $\U(\O) \cong C_2$ and write $\U(\O) = \langle u : u^2= 1 \rangle$ and $\B_2(\O) = \begin{pmatrix} \langle u \rangle & \O \\ 0 & \langle u \rangle \end{pmatrix}$ is isomorphic to $(\O,+) \rtimes \left( C_2 \times C_2 \right)$, where the  action of $(u,1)$ and $(1,u)$ on $(\O, +)$ is via taking the opposite (i.e. sends $x \in \O$ on $-x$). Furthermore $(\O,+)$ is a free $\Z$-module of rank $1,2$ or $4$. Thus the group $\B_2(\O)$ clearly has an epimorphism to $\Z \rtimes C_2 \cong D_{\infty} \cong C_2 \ast C_2$. Since this last group is a free product, also $\B_2(\O)$ has a non-trivial amalgamated decomposition.

There only remains one implication to be checked. So suppose that $\U(\O) \ncong C_2$. We will prove that $\B_2(\O)$ has property \FR. For this we will use \Cref{lemmaNilpotent}, applied to the group $\B_2(\O)$ with $N \cong (\O, +)$ the free abelian subgroup of unimodular upper triangular matrices. Since $\B_2(\O)/N \cong \U(\O) \times \U(\O)$ has property \FR (indeed, it is finite), it will suffice to prove there is no subgroup $M$ of $N$ which is normal in $\B_2(\O)$ and such that $N/M \cong \Z$.

Suppose such an $M$ does exist. Let $M'$ be the subgroup of the additive group of $\O$ such that $M = \{\sm{1 & y\\ 0 & 1} \mid y \in M'\}$. 
Now by assumption and \Cref{When is unit group order finite}, $D$ is equal to $\Q$, $\Q(\sqrt{-d})$ with $d > 0$ or a totally definite quaternion algebra $\qa{a}{b}{\Q}$. Therefore, by \Cref{basis of units or small cyclic unit group}, either $\O$ has a $\Z$-module basis consisting of units or $\U(\O)$ is isomorphic to $C_2$, $C_4$ or $C_6$. In the former case $\B_2(\O)$ has property \FR as proven in \Cref{When is E2 FA and some GL2 FA} (see the claim in its proof).
So we may now suppose that $\U(\O)$ is isomorphic to $C_4$ or $C_6$. First, for an order $\O$ in an imaginary quadratic number field $\mathcal{U}(\O) = \langle - 1\rangle$, unless $\O \in  \{\mathcal{I}_1, \mathcal{I}_3 \}$ and in these cases $\O$ has a basis of units. Therefore it remains to consider the case where $\O$ is an order in $\qa{a}{b}{\Q}$ with $a,b < 0$. Assume $\alpha$ is a generator of $\U(\O)$. Notice that $\Z[\alpha] \cong \Z[i]$ or $\Z[\zeta_3]$ where $\zeta_3$ is a primitive third root of unity. In both cases, $\Z[\alpha]$ is a principal ideal domain and $\O$ is a finitely generated torsion-free $\Z[\alpha]$-module. Using the fundamental theorem of finitely generated modules over PID's, we obtain that $\O = \Z[\alpha] \oplus b \Z[\alpha]$, for some $b \in \O$. Hence we obtain a $\Z$-basis $\{1,\alpha,b,b\alpha\}$ for $\O$. 

We will now go through the proof in the case $\alpha$ is of order $6$, but the order $4$ case is similar.

Since $M$ is normal in $\B_2(\O)$, taking the conjugate with $\begin{pmatrix} \beta^{-1} & 0 \\ 0 & 1 \end{pmatrix}$ and $\begin{pmatrix}1 & 0 \\ 0 & \beta\end{pmatrix}$ for $\beta \in \U(\O)$, yields $\beta M'\subseteq M'$ and $M'\beta \subseteq M'$ respectively.
Additionally, since $\alpha$ is of order $6$, $\alpha^2 = \alpha-1$. We will use these facts throughout.

Our first claim is that $\Z + \Z \alpha \subseteq M'$, or equivalently $M' \cap \langle \alpha \rangle \neq \{ 0 \}$.
Indeed, if we suppose the opposite, namely for all non-zero $\beta \in \langle \alpha \rangle$ that $\beta \notin M'$, then also for all $r \in \Z \setminus \{0\}$ and $\beta \in \langle \alpha \rangle$ the element $r\beta \notin M'$ (else $N/M$ is no longer torsion free).
However, since $N/M \cong \Z$, then we may find some integers $n,m \in \Z\setminus \{0\}$ such that $m \alpha \equiv n 1 \mod M'$.
This would imply that $m\alpha - n \in M'$, but then also $(m\alpha-n)\alpha = m(\alpha-1) - n\alpha = (m-n)\alpha - m \in M'$.
This shows that 
$$m(m\alpha-n) - n((m-n)\alpha-m) = (m^2-nm+n^2)\alpha \in M'.$$
Since $m^2-nm+n^2 \neq 0$, this yields a contradiction.

Suppose now that $b \in M'$. Then also $b\alpha \in M'$ and thus a whole basis of $\O$ is in $M'$. This contradicts $N/M \cong \Z$. Similarly, suppose $b\alpha \in M'$. Then $b\alpha^2 = b\alpha - b \in M'$, which implies $b \in M'$. This gives again a contradiction. Hence we have that $b \not\in M'$ and $b\alpha \not\in M'$ and thus also $rb \notin M'$ and $rb\alpha \notin M'$ for every $r \in \Z\setminus \{0\}$, for else $N/M$ would not be torsion free. In the same way as above, we again find two integers $n,m \in \Z\setminus \{0\}$ such that $m b \equiv n b\alpha \mod M'$. By a similar calculation, this gives again a contradiction. This shows that the set $M'$ does not exist and hence also $M$ does not exist. So \Cref{lemmaNilpotent} finishes the proof.
\end{proof}

From \Cref{Borel FA cyclic units} we see that also for $\B_2(\O)$ property \FR and \FA are equivalent.
Furthermore we see that $\B_2(\O_5)$ has property \FR. However it can be directly checked that there exists no $\lambda \in \D_2(\O_5)$ of finite order such that $\alpha(\lambda)$ has only non-rational eigenvalues.
So $\O_5$ yields a counterexample to the converse of the first statement in \Cref{property FA for the GRK groups}.

\section{Fixed point properties, exceptional components and cut groups}\label{sectie exc comp and cut groups}\label{FA and cut groups}

For the remainder of the paper {\it $G$ will denote a finite group.} In the sequel we aim at describing property \FA and \HFA for $\U(\Z G)$ both in terms of $G$ and the Wedderburn-Artin components of $\Q G$. In \Cref{(FA) implies cut} we will see that if $\U(\Z G)$ has \FA, then $G$ must be a so-called cut group. Therefore in \Cref{sectie exceptional components cut groups} we investigate the possible simple algebras $\Ma_n(D)$ that arise as a component of $\Q G$ for $G$ a cut group.

Earlier we recalled the concepts of reduced norm and $\SL_1$ for a subring of a central simple algebra in \Cref{orders and quaternion algebras}. In this part we will frequently need the notion $\SL_1(R)$ for $R$ a subring in a semisimple $\Q$-algebra $A$. Let $A = \prod \Ma_{n_i}(D_i)$ be the Wedderburn-Artin decomposition of $A$ and $h_i$ the projections onto the $i$-th component. Then $$\SL_1(R):=\{\ a \in R\ \mid \ \forall\ i \colon \operatorname{RNr}_{\Ma_{n_i}(D_i)/\ZZ(D_i)}(h_i( a)) = 1\ \}.$$ 

\subsection{FA and cut groups}

We start by proving that the size of $\mathcal{U}(\mathbb{Z}G)^{ab}$ restricts the size of the center. More generally the following is true. As before the rank of a finitely generated abelian group $A$ means the rank of its free part and will be denoted by $\rank_{\Z}$.

\begin{proposition}\label[proposition]{(FA) implies cut}
Let $\mathcal{O}$ be an order in a finite dimensional semisimple $\mathbb{Q}$-algebra $A$. Then
$$\rank_{\Z} \left( \U(\mathcal{O})^{ab} \right) \geq \rank_{\Z} \left( \U(\mathcal{Z}(\O))\right).$$
\end{proposition}
\begin{proof}
For any $n$ one can embed $\GL_n(\O)$ into $\GL_{n+1}(\O)$ by sending $B \in \GL_n(\O)$ to $\left( \begin{smallmatrix}
B &0 \\
0 & 1
\end{smallmatrix} \right)$. So (the direct limit) $\GL(\O)= \bigcup \GL_n(\O)$ is equipped with the obvious group structure whose identity element we denote $I_{\infty}$.  Recall that $K_1(\O) := \GL(\O)^{ab}$ and let 
\[ i: \U(\O) \rightarrow K_1( \O) : u \mapsto \left( e_{11}(u-1) + I_{\infty} \right) \GL(\O)' \]
be the canonical map, where $e_{11}(u-1)$ is the matrix with the value $u-1$ in the entry $(1,1)$ and zero elsewhere. So $i(u)$ is the image inside $K_1(\O)$ of the $\mathbb{N}\times \mathbb{N}$-identity matrix but with value $u$ instead of $1$ at place $(1,1)$.

By \cite[Corollary 9.5.10]{EricAngel1}, $i(\U(\mathcal{Z}(\O)))$ is of finite index in $K_1(\mathcal{O})$. In particular also $i(\U(\mathcal{O}))$ has finite index. Since $K_1(\mathcal{O})$ is abelian, $\U(\O)' \subseteq \ker (i)$ and we have an induced map $\overline{i}: \U(\O)^{ab} \rightarrow K_1(\O)$ whose image is still of finite index in $K_1(\O)$. Therefore $\rank_{\Z} \left( \U(\O)^{ab} \right) \geq \rank_{\Z}  \left( K_1(\O) \right)$.  The statement now follows by \cite[Corollary 9.5.10]{EricAngel1}  which says that $
\rank_{\Z} \left(K_1(\mathcal{O})\right) = \rank_{\Z} \left(\mathcal{U}(\mathcal{Z}(\mathcal{O}))\right)$.
\end{proof}

Due to \Cref{(FA) implies cut} and \Cref{prop:iff_conditions_FA}, if $\U(\Z G)$ has \FA, then $\U(\mathcal{Z}(\Z G))$ is finite. The latter is the content of so-called cut groups, a class of groups that was studied in its own right (the term ``cut'' was introduced in \cite{BMP17}).

\begin{definition}\label[definition]{def:cut_group}
A finite group $G$ is called a \emph{cut group}, if $\U(\ZZ(\Z G))$ is finite.
\end{definition}

The word cut derives from ``central units of the integral group ring trivial''. In fact by a classical theorem of Berman and Higman \cite[Proposition~7.1.4]{EricAngel1} each central unit in $\Z G$ not in $\pm \ZZ(G)$ has infinite order. Hence if $\U(\ZZ (\Z G))$ is finite, all central units must be trivial (i.e. in $\pm \ZZ (G)$).

\begin{corollary}\label[corollary]{FA implies cut}
Let $G$ be a finite group such that $\U (\Z G)$ has finite abelianization. Then $G$ is a cut group.
\end{corollary}

This also implies that when $\U (\Z G)$ has \FA, then $G$ is a cut group. The converse is however not true as the following example shows. 
\begin{example*}
Denote the subgroup of units with augmentation one of $\Z G$ by $\V(\Z G)$. Then $\U(\Z S_3)= \pm \V (\Z S_3)$ and $\V(\Z S_3)=\langle s,t,b \mid s^2, t^3, t^s = t^{-1}, b^s = b^{-1} \rangle$ (see for example \cite{MarSeh}). Clearly the latter group is an amalgamated product of the groups $\langle b,s \rangle \cong C_{\infty} \rtimes C_2 = C_2 \ast C_2$ and $\langle s,t \rangle =S_3$ over the subgroup $\langle s \rangle \cong C_2 $. So, also $\U(\Z S_3)$ has a non-trivial decomposition as amalgamated product and thus does not have property \FA. On the other hand, the amalgamated subgroup is finite, this shows that $S_3$ is a cut group.
\end{example*}

For a finite group $G$ the rational group algebra $\Q G$ is semisimple and thus has a Wedderburn-Artin decomposition $\Q G \cong \prod_{i=1}^m \Ma_{n_i}(D_i)$, where all the $D_i$ are rational division algebras. If $\O_i$ is an order in $D_i$, then $\Z G$ and $\prod_{i=1}^m \Ma_{n_i}(\O_i)$ are both orders in $\Q G$ and $\ZZ(\Z G)$ and $\prod_{i=1}^m \ZZ(\Ma_{n_i}(\O_i)) \cong \prod_{i=1}^m \ZZ(\O_i)$ are orders in $\ZZ(\Q G)$. Hence \Cref{order zijn commensurable} implies the following fact that we will use in the sequel without further reference: $G$ is a cut group if and only if all the centres $\ZZ(\O_i)$ have a finite unit group, that is, $\ZZ(D_i)$ is the field of rational numbers or an imaginary quadratic extension of $\Q$ (cf.\ \Cref{When is unit group order finite}).

It would be particularly interesting whether equality holds in \Cref{(FA) implies cut}.  By \cite[Lemma 9.5.6]{EricAngel1} for two orders $\mathcal{O}_1, \mathcal{O}_2$ in a finite dimensional semisimple $\mathbb{Q}$-algebra $A$, $K_1(\O_1)$ and $K_1(\O_2)$ are commensurable.
Furthermore, as stated above, $\rank_{\Z} (\mathcal{U}(\mathcal{Z}(\O_1)))= \rank_{\Z} (K_1(\O_1))$, and similarly for $\O_2$, and hence for orders ``having a center with finitely many units'' is a property defined on commensurability classes. In particular if equality in \Cref{(FA) implies cut} holds, one would also have a positive answer to the following question.

\begin{question} \label[question]{conjecture ab independent order}
Let $\mathcal{O}_1$ and $\mathcal{O}_2$ be two orders in a finite dimensional semisimple $\mathbb{Q}$-algebra. Is $\mathcal{U}(\mathcal{O}_1)^{ab}$ finite if and only if $\mathcal{U}(\mathcal{O}_2)^{ab}$ is finite?
\end{question}

As in general, finite abelianization and property \FA do not descend to subgroups of finite index, a positive answer to the above question cannot be given right away. In contrast, property \HFA does descend to subgroups of finite index and therefore the following holds.

\begin{proposition}\label[proposition]{prop: reductie tot sln en center}
Let $G$ be a finite group, $\Q G \cong \prod\limits_{i=1}^m \Ma_{n_i}(D_i)$ the Wedderburn-Artin decomposition of its rational group algebra $\Q G$ and $\O_i$ an order in $D_i$. Then the following properties are equivalent:
\begin{enumerate}
\item  $\U(\Z G)$ has property \HFR (resp. \HFA), 
\item $\GL_{n_i}(\O_i)$ has property \HFR (resp. \HFA)  for all $1 \leq i \leq m$,
\item $\SL_{n_i}(\O_i)$ has property \HFR (resp. \HFA)  for all $1 \leq i \leq m$ and $G$ is a cut group.
\end{enumerate}
\end{proposition}
\begin{proof}
We prove the equivalences for property \HFR. The proofs for property \HFA are the same.

First note that $\Z G$ and $\prod_{i=1}^m \Ma_{n_i}(\O_i)$ are both orders in $\Q G$. Hence by \Cref{order zijn commensurable}, $\U(\Z G)$ and  $\prod_{i=1}^m \GL_{n_i}(\O_i)$ are commensurable. This shows the equivalence between (1) and (2), see \Cref{FAconserved} and the remark thereafter. 

 For any order $\O$ in a finite dimensional semisimple $\mathbb{Q}$-algebra, $\langle \SL_1(\O), \U (\ZZ ( \O)) \rangle$ has finite index in $\U (\O)$ and $\SL_1(\O) \cap \U (\ZZ ( \O))$ is finite by \cite[Proposition~5.5.1]{EricAngel1}. Hence $\U(\O)$ has property \HFR if and only if $\SL_1(\O)$ and $\U (\ZZ (\O))$ both have property \HFR. 

 Suppose that (1) and hence also (2) hold. By \Cref{FA implies cut}, $G$ is cut. Now consider $\O = \prod_{i=1}^m \Ma_{n_i}(\O_i)$. By the previous paragraph, and the definition of $\SL_1$ for semisimple algebras, all $\SL_{n_i}(\O_i)$ have property \HFR. This gives (3). 
 
 Now assume (3). By the discussion following \Cref{FA implies cut}, all the $\U(\ZZ(\O_i))$ are finite and thus have property \HFR. By the paragraph above, (2) follows. \end{proof}

\subsection{Higher rank and exceptional components}\label[subsection]{subsectie higher rank}

Due to \Cref{prop: reductie tot sln en center}, property \HFA for an order in a finite dimensional semisimple $\Q$-algebra depends on its Wedderburn-Artin components. It will turn out that the main obstruction for \HFA lies in the following type of components.

\begin{definition}\label[definition]{definitie exceptional component}
Let $D$ be a finite dimensional division algebra over $\mathbb{Q}$. The algebra $\Ma_n(D)$ is called \emph{exceptional} if it is of one of the following types:
\begin{enumerate}
\item[(I)] a non-commutative division algebra other than a totally definite quaternion algebra over a number field,
\item[(II)] a $2 \times 2$-matrix ring $\Ma_2(D)$ such that $D$ has an order $\O$ with $\mathcal{U}(\mathcal{O})$ finite. 
\end{enumerate}
\end{definition}

Recall that by a theorem of Kleinert \cite[Proposition~5.5.6]{EricAngel1} the non-commutative division algebras excluded in type (I) are exactly those having an order $\mathcal{O}$ with $\SL_1(\mathcal{O})$ finite.
Also recall that, by \Cref{When is unit group order finite}, the condition in type (II) is a condition which can be formulated in terms of $D$. The name ``exceptional component'' was coined in \cite{JeOlvGdR} because under the presence of such a component the known generic constructions of units do not necessarily generate a subgroup of finite index in $\mathcal{U}(\mathbb{Z}G)$ \cite{JesLea, EricAngel1}. The crux of that failure is that these components are exactly those where respectively $\SL_1(D)$ and $\SL_2(D)$  have `bad' (arithmetic) properties as algebraic group (for the meaning of `bad' see \Cref{exceptional component as bad SL} below). Therefore we will now review the structure of $\SL_2(D)$ as an algebraic group and subsequently interpret certain algebraic groups results into the  language of exceptional components. Alternatively, the reader might opt to accept the content of \Cref{property (T) and FA_n for non-exceptional 2 by 2} and go further to the next section.

Let $D$ be a finite dimensional division algebra over $\Q$ of degree $d$ and denote $\mathcal{Z}(D)$ by $K$. Further, let $E$ be a splitting field of $D$, i.e. $D \otimes_K E \cong  \Ma_d(E)$. Call the latter isomorphism $\varphi$. Then $\varphi$ restricts to an embedding of $D$, viewed as $D \otimes 1$, into $\Ma_d(E)$ and 
$$\SL_2(D) = \left\{\ \left( \begin{matrix}
 a & b \\ c& d
\end{matrix} \right) \in \Ma_2(D)\ \middle|\ \det \left( \begin{matrix}
 \varphi (a) & \varphi (b) \\ \varphi (c) & \varphi (d)
\end{matrix} \right) = 1\ \right\}.$$

So, in the above, using $\varphi$, we identify $\Ma_2(D)$ with a $K$-subspace of $\Ma_{2d}(E)$. Then we see that $\SL_2(D)$ actually is the Zariski closed subspace of the affine space $K^{4d^2} \cong \Ma_2(D)$ defined by the polynomial equation $\operatorname{RNr}_{\Ma_2(D)/K} = 1$. Due to this, $\SL_2(D)$ can be viewed as the $K$-rational points of an algebraic group $\Gamma$ defined over $K$.

More generally let $\Gamma$ be a linear algebraic $K$-group. Then, by $\operatorname{rank}_{K} \Gamma(K)$ we denote the dimension of a maximal $K$-split torus of $\Gamma(K)$, called {\it reductive $K$-rank} of $\Gamma$.  Recall that a \emph{$K$-split torus} is a commutative algebraic subgroup $T$ of $\Gamma(K)$ which is diagonalizable over $K$, i.e.\ $T$ is defined over $K$ and $K$-isomorphic to  $\prod_{1 \leq i\leq q} K^*$, where $q = \dim T = \operatorname{rank}_{K} \Gamma(K)$. All maximal split $K$-tori of $\Gamma$ are conjugate over $K$ \cite[15.2.6.]{SpringerBook} and hence $\operatorname{rank}_{K} \Gamma(K)$ is independent of the choice of $T$. 
 
In our case, $\Gamma(K) = \SL_2(D)$ with $K$ a global field of characteristic $0$ (i.e.\ a finite extension of $\Q$). Consequently, for every valuation $v$ of $K$ the completion of $K$ with respect to $v$, denoted $K_v$, is a local field. Note that $\Gamma(K_v) = \SL_2(D \otimes_K K_v) = \SL_1( \Ma_2(D\otimes_K K_v))$. It is not hard to see that $\operatorname{rank}_K \SL_n(D) = n-1$, where the diagonal matrices with entries in $K$ and determinant $1$ form a maximal $K$-split torus. With this terminology at hand we can be more precise about the `bad' behaviour of $\SL_1(D)$ and $\SL_2(D)$.

\begin{remark}\label[remark]{exceptional component as bad SL}
For $n \geq 2$, $\Ma_n(D)$ is exceptional exactly when $\SL_n(D)$ has a negative answer to the Congruence Subgroup Problem. The reason for the failure is that the type (II)  components are exactly those where $\SL_n(D)$ is an algebraic group with $\mbox{$S-\operatorname{rank}(\SL_n(D)) $}:= \sum_{v \in S} \operatorname{rank}(\SL_n(D \otimes_K K_v) )$, called $S$-rank, equal to $1$, where $S$ is the set of Archimedean places of $\mathcal{Z}(D)$.
In the case of a type (I) exceptional component, still very little is known about the answer to the Congruence Subgroup Problem for $\SL_1(D)$. One reason for this is the lack of unipotent elements in $\SL_1(D)$, which is also an obstruction for the construction of generic units contributing to a subgroup of finite index in these components.
Since those excluded in type (I) are such that $\SL_1(\O)$ is finite, for any order $\O$, such orders do not pose a problem.
\end{remark}

A precise rank computation yields the following well known result for which we unfortunately could not find a concrete reference. Therefore, for the convenience of the reader, we sketch a proof. 

\begin{lemma}[Folklore]\label[lemma]{rank of SL2}
Let $D$ be a finite dimensional division $\Q$-algebra and suppose $n \geq 2$. Then there exists an Archimedean place $v$ of $K:=\mathcal{Z}(D)$ such that $\operatorname{rank}_{K_v} \SL_n(D \otimes_K K_v) = 1$ if and only if
\begin{itemize}
\item $n=2$ and
\item $D$ is a (number) field or $D = \qa{a}{b}{K}$ with $a,b < 0$ and $K$ is not totally imaginary.
\end{itemize} 
\end{lemma}
\begin{proof}
Note that for $n \geq 2$, $\SL_n(D \otimes_K K_v)$ contains a $K_v$-split torus. Hence $\operatorname{rank}_{K_v} (\SL_n(D \otimes_K K_v)) \geq 1 $ for any valuation $v$ of $K$. Suppose first that $D=K$ is a number field. Then $D \otimes_K K_v \cong K_v$ for any place $v$ and it is clear that $\operatorname{rank}_{K_v} (\SL_n(K_v)) = n-1$. In particular the rank equals $1$ if and only if $n=2$. 

Now suppose that $D \neq K$. As $D$ is a central simple algebra over $K$, and $K_v$ is a simple $K$-algebra, one has, by \cite[Proposition~2.1.8]{EricAngel1}, that $D \otimes_K K_v$ is a central simple algebra over $K_v$, say $\Ma_{d}(D')$. Now, $\SL_n(D \otimes_K K_v) = \SL_1(\Ma_n(D\otimes_K K_v)) \cong \SL_1(\Ma_{nd}(D')) = \SL_{nd}(D')$. As before, a $K_v$-split torus of $\SL_{nd}(D')$ consists of the diagonal matrices with values in $K_v$ and determinant $1$. Therefore $\operatorname{rank}_{K_v} \SL_{nd}(D') \geq nd-1$ (actually the former torus is maximal and hence equality even holds). In particular, $\operatorname{rank}_{K_v} \SL_{nd}(D')=1$ if and only if $n=2$ and $d=1$.  The latter implies that $K_v$ does not split $D$ and thus $K_v \neq \mathbb{C}$. Consequently we may assume that $K$ is not totally imaginary. Let $v$ be a real place and $D\otimes_K \mathbb{R} = D'$ a non-commutative real division algebra. Then by Frobenius theorem $D' \cong \qa{-1}{-1}{\mathbb{R}} \cong \qa{a}{b}{\mathbb{R}}$ for any $a,b <0$. This finishes the proof because $D'$ was obtained by tensoring $D$ with $K_v$ over its center.
\end{proof}

Let $S$ be a finite set of places of $K =\ZZ(D)$ containing the Archimedean ones. In case $\SL_2(D)$ is of so-called higher rank, i.e.\ $\operatorname{rank}_{K_v} \SL_2(D \otimes_K K_v) \geq 2$ for all $v \in S$, strong fixed point properties hold such as property \T. By Delorme-Guichardet's Theorem \cite[Theorem 2.12.4]{ValetteBook} a countable discrete group $\Gamma$ has \emph{property \T} if and only if every affine isometric action of $\Gamma$ on a real Hilbert space has a fixed point.  For background on property \T we refer the reader to the nicely written book \cite{ValetteBook}. In particular \cite[Theorem~2.12.6]{ValetteBook} shows that property \T implies property \FA and, since property \T descends to finite index subgroups, also \HFA.

In \cite[Theorem~(5.8), page~131]{MargulisBook}, Margulis showed that $S$-arithmetic subgroups of connected semisimple $K$-groups of higher rank have property \T. In \cite[Theorem 1.1.]{Farb}, Farb showed that $S$-arithmetic subgroups of almost simple simply-connected connected $K$-groups of $K$-rank $n \geq2$ have property \HFA[n-1].  In the following theorem we restrict these results to our context and explain in the proof how to deduce it from the original theorems.

\begin{theorem}[Margulis-Farb]\label{property (T) and FA_n for non-exceptional 2 by 2}
Let $D$ be a finite dimensional division $\Q$-algebra with $\mathcal{Z}(D) = \mathbb{Q}(\sqrt{-d})$ where $d \geq 0$ and let $\O$ be an order in $D$. Suppose that $\Ma_n(D)$ is non-exceptional. Then $\SL_n(\O)$ has property \T. If moreover $n \geq 3$ then it also has property \HFA[n-2].
\end{theorem}
\begin{proof}
Set $K := \mathcal{Z}(D) = \mathbb{Q}(\sqrt{-d})$. Recall that $\SL_n(D)$ is a connected almost $K$-simple algebraic group (i.e.\ all proper connected algebraic $K$-subgroups are finite) due to the assumption on $K$. Furthermore $\SL_n(D)$ is also simply connected (i.e.\ any central isogeny $\varphi:  H \rightarrow \SL_n(D)$, with $H$ a connected algebraic group, is an algebraic group isomorphism). Due to the form of $K$, it has a unique (up to equivalence) Archimedean valuation, say $v$. Note that $K_v = \mathbb{R}$ if $K = \mathbb{Q}$ and $K_v = \mathbb{C}$ if $K = \mathbb{Q}(\sqrt{-d})$ with $d > 0$. Taking $S = \{ v \}$, the set of all Archimedean places, we see that an $S$-arithmetic subgroup of $\SL_n(D)$ is simply an arithmetic subgroup of $\SL_n(D)$ of which $\SL_n(\O)$ is an example. To obtain property \T we invoke the celebrated theorem of Margulis \cite[Theorem~(5.8), page~131]{MargulisBook}. In order to apply the latter we need that $\operatorname{rank}_{K_v} \SL_n(D\otimes_K K_v) \geq 2$ for the unique Archimedean place $v$, which by \Cref{rank of SL2} and the form of the center amounts to say that $\SL_n(D) \notin \left\lbrace \SL_2(\mathbb{Q}(\sqrt{-d})), \SL_2\left(\qa{a}{b}{\mathbb{Q}}\right) \right\rbrace$, where $d \geq 0$ and $a,b < 0$. By \Cref{When is unit group order finite} these are exactly the division algebras having an order with finite unit group. In other words \cite[Theorem~(5.8), page~131]{MargulisBook} can be applied if $\Ma_n(D)$ is non-exceptional.

Now if $n \geq 3$, then $\operatorname{rank}_{K} \SL_n(D) = n-1 \geq 2$.  Hence all the conditions of \cite[Th. 1.1.]{Farb} with $S = \{ v \}$ are also satisfied, implying property \HFA[n-2].
\end{proof}

If $\Ma_n(D)$ is non-exceptional, the groups $\E_n(J)^{(m)}$, for any $m \geq 1$ and $J$ a non-zero ideal in an order $\O$ of $D$, have finite index in $\SL_n(\O)$.  By \cite[Theorem 11.2.3 and Theorem 12.4.3]{EricAngel1} and \Cref{th: elementary matrices FR} one can obtain, in case $n \geq 3$, an alternative proof that $\SL_n(\O)$ and $\E_n(J)^{(m)}$ satisfy property \HFR and \HFA[n-2] without use of \Cref{property (T) and FA_n for non-exceptional 2 by 2} (or more precisely independent of the deep theorems \cite[Th. (5.8), page~131]{MargulisBook} and \cite[Th. 1.1.]{Farb}). 

\begin{corollary}\label[corollary]{HFR voor E_n en SLn met n >3}
Let $\O$ be an order in a finite dimensional division $\Q$-algebra and $J$ a non-zero ideal in $\O$. Then $\E_n(J)^{(m)}$, for any $m\geq 1$, has property \HFR and \HFA[n-2] if $n \geq 3$. In particular, $\SL_n(\O)$ has property \HFR and \HFA[n-2] if $n \geq 3$. 
\end{corollary}
\begin{proof}
By \cite[Theorem 11.2.3 and Theorem 12.4.3]{EricAngel1} $\E_n(J)^{(m)}$ has finite index in $\SL_n(\O)$. In particular the former has property \HFR and \HFA[n-2] if and only if the latter does. Let $H$ be a subgroup of $\SL_n(\O)$ of index $[\SL_n(\O) : H] = m < \infty$. Then clearly $\E_n(\O)^{(m)} \leq H$ and $\E_n(\O)^{(m)}$ is of finite index in $H$ by the above. Now since $\O$ is a finitely generated $\Z$-module, using \Cref{theo:E_n(R)_has_FR},  $\E_n(\O)^{(m)}$ has property \FR and \FA[n-2]. Consequently by \Cref{FAconserved} also $H$ has property \FR and \FA[n-2]. 
\end{proof}

\subsection{Describing exceptional components of cut groups}\label[subsection]{sectie exceptional components cut groups}\label[subsection]{descsimple}

In order to describe when $\U(\Z G)$ has property \HFA, by \Cref{prop: reductie tot sln en center}, one has to investigate the components of $\Q G$.  By \Cref{property (T) and FA_n for non-exceptional 2 by 2} we are left with the exceptional components. Moreover, by \Cref{FA implies cut}, whenever $\U(\Z G)$ has property (HFA), $G$ must be a cut group. Therefore, we now investigate the possible exceptional Wedderburn-Artin components of $\Q G$ in case $G$ is a cut group.

We first consider components of type (I), i.e. the exceptional $1\times1$ components, of $\Q G$ for $G$ a cut group. Surprisingly, it turns out that there none. This result will be crucial in the representation theoretical applications later on.

\begin{proposition} \label[proposition]{no excp 1 by 1 for cut groups}
 Let $G$ be a finite cut group. Then $\Q G$ has no exceptional components of type (I).
\end{proposition}

Suppose that $D$ is a $1 \times 1$ component of $\mathbb{Q}G$. Then the proof of \Cref{no excp 1 by 1 for cut groups} consists of the following steps.
\begin{enumerate}
\item\label{enum:1_no_except_1_1_comp} There exists a primitive central idempotent $e$ such that $D = \mathbb{Q}Ge$. The group $H = Ge$ is a finite subgroup of $\U(D)$, hence a Frobenius complement \cite[2.1.2, page~45]{ShiWeh}. If $G$ is cut, also its epimorphic image $H$ is cut. Frobenius complements that are cut were classified by B{\"a}chle \cite[Proposition~4.2]{Bachle}.
\item Some of the groups $H$ obtained in \eqref{enum:1_no_except_1_1_comp} are indeed subgroups of a division algebra. This can be decided using Amitsur's classification \cite{Amitsur}, but we will give a direct argument.
\item For all remaining $H$, the smallest division algebra generated by $H$ and hence also $D$ is determined. 
\end{enumerate}
These steps will be realized in \Cref{cut and division} and, as just explained, \Cref{no excp 1 by 1 for cut groups} follows immediately from this.  

\begin{proposition} \label[proposition]{cut and division} A finite group $G$ is both cut and isomorphic to a subgroup of units of a division $\Q$-algebra $D$ if and only if $G$ is one of the following groups
\begin{enumerate}
 \item\label{it:cut_subgroups_in D_1} $1$, $C_2$, $C_3$, $C_4$, $C_6$,
 \item\label{it:cut_subgroups_in D_2} $C_3 \rtimes C_4$, where the action is by inversion,
 \item\label{it:cut_subgroups_in D_3} $Q_8$,
 \item\label{it:cut_subgroups_in D_4} $\SL(2, 3)$. 
\end{enumerate}
Moreover, the $\Q$-span of these groups in any division algebra is, respectively,
\begin{enumerate}[label=(\Roman*),ref=\Roman*]
 \item\label{it:cut_divalg_in D_1} $\Q$, $\Q$, $\Q(\zeta_3)$, $\Q(\sqrt{-1})$, $\Q(\zeta_3)$,
 \item\label{it:cut_divalg_in D_2} $\left( \frac{-1,-3}{\Q}\right)$,
 \item\label{it:cut_divalg_in D_3} $\left( \frac{-1,-1}{\Q}\right)$,
 \item\label{it:cut_divalg_in D_4} $\left( \frac{-1,-1}{\Q}\right)$. 
\end{enumerate}
\end{proposition}
\begin{proof}
Note that all the groups listed in \eqref{it:cut_subgroups_in D_1} - \eqref{it:cut_subgroups_in D_4} are cut groups\footnote{This can be checked easily using the characterisations \cite[Proposition 2.2 (iii) \& (v)]{Bachle}. Alternatively one may use GAP. An example of a code for this is added in \ref{Appendix tabel met alle data}.}. Further, they are also subgroups of division algebras. Indeed the cyclic groups are subgroups of $\U(\Q(\zeta_{12}))$, $C_3\rtimes C_4 \cong \U(\O_3)$ is a subgroup of $\U (\left( \frac{-1,-3}{\Q}\right))$ and $Q_8 \cong \langle i,j \rangle$ and $\SL(2,3) \cong \langle i, \frac{1+i+j+k}{2}\rangle$ are subgroups of $\U(\left( \frac{-1,-1}{\Q}\right))$.

To prove the last statement (which we will use in the converse implication), we will consider $\Q[G]$ (the subring of $D$ generated by the subgroup $G$) in any division $\Q$-algebra $D$ containing $G$, for each group listed in \eqref{it:cut_subgroups_in D_1} to \eqref{it:cut_subgroups_in D_4}.
For the cyclic groups, it is clear these are the fields listed in (I).
In case of the other groups, we will follow the following strategy: $\Q[G]$ is a simple, epimorphic image of the rational group algebra $\Q G$, so it has to be a division algebra appearing in the Wedderburn-Artin decomposition of $\Q G$.
Moreover, since these groups are not abelian, the division algebra also has to be non-commutative.
Using  \textsf{GAP} \cite{GAP}, it is easy to see that the only non-abelian division algebras appearing in the decomposition of the rational group algebra for the groups \eqref{it:cut_subgroups_in D_2}, \eqref{it:cut_subgroups_in D_3} and \eqref{it:cut_subgroups_in D_4} are respectively \eqref{it:cut_divalg_in D_2}, \eqref{it:cut_divalg_in D_3} and \eqref{it:cut_divalg_in D_4}.

Lastly, from \cite[Proposition 4.2]{Bachle} it follows that the Frobenius complements that are cut groups are exactly the groups in \eqref{it:cut_subgroups_in D_1} - \eqref{it:cut_subgroups_in D_4} together with the group $C_3 \times Q_8$.
Since finite subgroups of division algebras are  Frobenius complements it suffices to prove that $C_3 \times Q_8$ is not embeddable in a skew field.
This can be done via Amitsur's famous classification theorem, but we will provide a more direct proof. Let $D$ be a division $\Q$-algebra such that $G:= C_3 \times Q_8 \leq \U(D)$. Recall now that if $B$ is a finite-dimensional central simple $F$-algebra, contained in an $F$-algebra $A$, then $A = B \otimes \operatorname{C}_A(B)$ (for example see {\cite[Theorem 4.7]{Jacobson}}), hence $\Q[G] = \Q[Q_8] \otimes_\Q \operatorname{C}_{\Q[G]}(\Q[Q_8])$.
Since the centeralizer of $Q_8$ in $G$ is $C_3 \times C_2$, it follows that $\Q[G]$ contains \begin{align*}
\Q[Q_8] \otimes_\Q \Q[\operatorname{C}_G(Q_8)]  = \Q[Q_8] \otimes_\Q \Q[C_3 \times C_2]  \cong \left( \frac{-1,-1}{\Q} \right) \otimes_{\Q} \Q(\zeta_3) \cong \left( \frac{-1,-1}{\Q(\zeta_3)} \right).
\end{align*}
It is well known that this last algebra is split (for example using \cite[Theorem 5.4.4]{Voight}),  which is in contradiction with the fact that $D$ is a division algebra.
\end{proof}

Let us now consider components of type (II).
Surprisingly if one assumes $\Ma_2(D)$ to be an exceptional component of $\mathbb{Q}G$, then the parameters $d$ and $(a,b)$ of $D= \Q(\sqrt{-d})$ (resp.\ $\left( \frac{a,b}{\Q}\right)$) are very limited. It was proven by Eisele, Kiefer and Van Gelder  \cite[Corollary~3.6]{EKVG} that only a finite number of division algebras can occur and moreover the possible parameters were described.

\begin{theorem}[Eisele, Kiefer, Van Gelder]\label{theorem EKVG}
Let $G$ be a finite group and let $e$ be a primitive central idempotent of $\mathbb{Q}G$ such that $\mathbb{Q}Ge$ is an exceptional component of type (II). Then $\mathbb{Q}Ge$ is isomorphic to one of the following algebras
\begin{enumerate}
\item[(i)] $\Ma_2(\Q)$,
\item[(ii)] $\Ma_2(\Q(\sqrt{-d}))$ with $d \in \{ 1,2,3 \}$,
\item[(iii)] $\Ma_2(\mathbb{H}_d)$ with $d \in \{ 2,3,5 \}$.
\end{enumerate}
\end{theorem}

\begin{remark}\label[remark]{exc comp have unique max order}
All the fields and division algebras appearing in \Cref{theorem EKVG} have the peculiar property to contain a norm Euclidean order $\O$ which therefore is maximal and unique up to conjugation \cite[Section 2.3]{CeChLez}. In view of \cite[(21.6), page~189]{Reiner}, this yields that also all the $2\times 2$-matrix algebras in \Cref{theorem EKVG} have, up to conjugation, a unique maximal order, namely $\Ma_2(\O)$. Recall that in case of $\Q(\sqrt{-d})$, with $ d \in \{ 0,1,2,3 \}$, the unique maximal order is their respective ring of integers $\I_d$ and in case of $\mathbb{H}_2, \mathbb{H}_3,\mathbb{H}_5$ the respective maximal orders where described in table (\ref{eq:basis_of_quat_orders}). Note that being norm Euclidean implies that these orders are also $\GE_2$-rings (see \Cref{euclidian are GE2}).
\end{remark}

Furthermore in \cite{EKVG} the authors classified the possibilities for a finite group to admit a faithful embedding in an exceptional component of type (II). More precisely, they found 55 possible groups\footnote{Note that the group with \textsc{SmallGroupID} \texttt{[24, 1]} and structure description $C_3 \rtimes C_8$ also has a faithful embedding in $\Ma_2(\mathbb{Q}(\sqrt{-1}))$, but is missing in \cite[Table 2]{EKVG}. However it is included in the table in \ref{Appendix tabel met alle data}.}, see \cite[Table 2]{EKVG}. In \ref{Appendix tabel met alle data} we add the aforementioned table, along with the information on all the exceptional type (II) components of $\Q H$, for each $H$ in the list, and certain group theoretical properties of $H$. By $G_{m, \ell}$ we denote the group with \textsc{SmallGroupID} $(m, \ell)$ in the Small Groups Library of \textsf{GAP} \cite{GAP}. For a presentation of the groups appearing see \ref{Appendix_presentation}. Using the table from \ref{Appendix tabel met alle data} it is easy to check the following.

\begin{proposition}\label[proposition]{possible 2 times 2 components}
 Let $G$ be a finite cut group. Then the following properties hold.
 \begin{enumerate}
 \item If there exists a primitive central idempotent $e_1$ of $\mathbb{Q}G$ such that $\mathbb{Q}Ge_1 \cong \Ma_2\left(\left( \frac{-1,-3}{\mathbb{Q}} \right) \right)$, then there exists another primitive central idempotent $e_2$ such that $\mathbb{Q}Ge_2 \cong \Ma_2(F)$ with $F= \mathbb{Q}$ or $\mathbb{Q}(\sqrt{-1})$.
 \item  If there exists a primitive central idempotent $e_1$ of $\mathbb{Q}G$ such that $\mathbb{Q}Ge_1 \cong \Ma_2(\mathbb{Q}(\sqrt{-2}))$, then there exists another primitive central idempotent $e_2$ such that $\mathbb{Q}Ge_2 \cong \Ma_2(\mathbb{Q})$.
 \item There exists a primitive central idempotent $e$ of $\mathbb{Q}G$ such that $\mathbb{Q}Ge \cong \Ma_2(\Q)$ if and only if $G$ maps onto $D_8$ or $S_3$.
 \item There exists a primitive central idempotent $e$ of $\mathbb{Q}G$ such that $\mathbb{Q}Ge \cong \Ma_2\left(\qa{-2}{-5}{\Q}\right)$ if and only if $G$ maps onto $G_{240, 90} \cong SL(2,5) \rtimes 2 \cong 2 \cdot S_5^+$, the Schur cover of $S_5$ of plus type.
 \item If $G$ is solvable, there exists no primitive central idempotent $e$ of $\mathbb{Q}G$ such that $\mathbb{Q}Ge \cong \Ma_2\left(\qa{-2}{-5}{\Q}\right)$.
 \item If $G$ is nilpotent, there also exists no primitive central idempotent $e$ of $\mathbb{Q}G$ such that $\mathbb{Q}Ge \cong \Ma_2\left(\qa{-1}{-3}{\Q}\right)$.
 \end{enumerate}
\end{proposition}

\section{Unit theorems for units of integral group rings}\label{sectie HFA}
In this section we will prove our main result, \Cref{mainTheoremintro} from the introduction: a  characterization of  property \HFA for $\U (\Z G)$ both in terms of $G$ and the Wedderburn-Artin components of $\Q G$.  First we   observe  that  the problem can be reduced to the groups $\GL_n(\O)$, for $\O$ some order in a finite dimensional rational division algebra. 

Due to the results obtained so far, we are now able to give a short proof of the following characterization of  when $\U(\Z G)$ has property \HFA, both in ring theoretical terms and in function of the quotients of $G$.

\begin{theorem}\label{iff HFA}
Let $G$ be a finite group. Then the following properties are equivalent:
\begin{enumerate}
\item The group $\U(\Z G)$ has property $\HFA$,
\item $G$ is cut and $\mathbb{Q}G$ has no exceptional components,
\item $G$ is cut and $G$ does not map onto one of the following $10$ groups
 \begin{center}
 \begin{tabular}{llllcl}
  $D_8$, & $G_{16,6}$, & $G_{16,13}$, & $G_{32,50}$, &  or & $Q_8 \times C_3$, \\ $S_3$, & $\SL(2,3)$, & $G_{96, 202}$, & $G_{240, 90}$, & or & $G_{384, 618}$.  
 \end{tabular}
 \end{center}
\end{enumerate}
\end{theorem}
\begin{proof}
Let $\Q G =\prod_{i = 1}^n \Ma_{n_i}(D_i)$ be the Wedderburn-Artin decomposition of $\Q G$. For each $i$, let $\O_i$ be a maximal order in $D_i$. By \Cref{prop: reductie tot sln en center}, $\U(\Z G)$ has property \HFA if and only if $G$ is a cut group and all $\SL_{n_i}(\O_i)$ have property \HFA. So in (1) we may also assume that $G$ is cut.

If $n_i \geq 3$, then $\SL_{n_i}(\O_i)$ has property \HFA by \Cref{HFR voor E_n en SLn met n >3}. Next, if $n_i = 1$, then $D_i$ is a number field or a totally definite quaternion $\Q$-algebra by \Cref{no excp 1 by 1 for cut groups}. Furthermore, $\mathcal{Z}(D_i) = \Q(\sqrt{-d})$ with $d \geq 0$, since we may assume $G$ is cut by the above. Consequently, by \Cref{When is unit group order finite}, $\SL_1(\O_i)$ is finite and hence has property \HFA.  At this stage we have that $\U(\Z G)$ has property \HFA if and only if $G$ is a cut group and for each $2\times 2$-component $\Ma_2(D_i)$ of $\Q G$, the corresponding $\SL_2(\O_i)$ has property \HFA.

If $\Ma_2(D_i)$ is non-exceptional, then $\SL_2(\O_i)$ has property \HFA by \Cref{property (T) and FA_n for non-exceptional 2 by 2}. Suppose now there exists a primitive central idempotent $e_{i_0}$ in $\Q G$ such that $\Q G e_{i_0} \cong \Ma_2(D_{i_0})$ is an exceptional component of type (II). By \Cref{possible 2 times 2 components} there is a primitive central idempotent $e_i$ such that $\Q G e_i = \Ma_2(D_i)$ is also an exceptional component of type (II), but $\Q G e_i \not\cong \Ma_2\left(\qa{-1}{-3}{\Q}\right)$. Then, by \Cref{When is E_2(O) HFA}, $\E_2(\O_i)$ does not have property \HFA. As $\O_i$ is a $\GE_2$-ring (cf.\ \Cref{exc comp have unique max order}), it follows that $\E_2(\O_i)$ has finite index in $\SL_2(\O_i)$ and hence $\SL_2(\O_i)$ also does not have property \HFA.

By the above, it remains to describe the condition ``no exceptional $2\times 2$-components'' in terms of forbidden quotients of $G$. If $\Q G e$ is an exceptional component then $H= Ge$ must appear in the table in \ref{Appendix tabel met alle data}.  In a first instance one has to filter out the non-cut groups. In this list, certain groups have another smaller (in size) group in the remaining list as epimorphic image. These groups may also be filtered out. Eventually, one is left with the groups listed in the statement.
\end{proof}

Note that, by \Cref{no excp 1 by 1 for cut groups} we may substitute statement $(2)$ in \Cref{iff HFA} by the statement
\begin{enumerate}
\item[$(2')$] $G$ is cut and $\Q G$ has no exceptional components of type (II).
\end{enumerate}
It would be interesting to have a characterisation in terms of the internal structure of $G$, instead of in terms of quotients.

\begin{remark}
One can obtain a similar result as above for $\mathbb{Q}(\zeta_n)G$, where $\zeta_n$ is a primitive complex $n$th root of unity. Due to \Cref{(FA) implies cut} and Dirichlet's unit theorem, one obtains readily that if $\mathcal{U}(\mathbb{Z}(\zeta_n)G)$ has \FA, then $n$ divides $4$ or $6$. For all of these values $n$ there will be again no exceptional $1 \times 1$ component. Furthermore if $n=4$ only $\Ma_2(\mathbb{Q}(\sqrt{-1}))$ can occur as $2\times2$ exceptional component and if $n = 3$, then only $\Ma_2(\mathbb{Q}(\sqrt{-3}))$; recall that $\mathbb{Z}[\zeta_6] = \mathbb{Z}[\zeta_3]$. Using the table in \cite{BCvG} one can again describe, in terms of quotients of $G$, when such components occur. 
\end{remark}

As a consequence of \Cref{iff HFA}, we get a result on the subgroup of $\U(\Z G)$ generated by bicyclic units. Recall that a bicyclic unit of an integral group ring $\Z G$, for a finite group $G$, is a unit of the type 
$$b(g,\tilde{h}) = 1 + (1-h)g\tilde{h} \textrm{ or } b(\tilde{h},g) = 1 + \tilde{h}g(1-h),$$
where $g,h \in G$ and $\tilde{h}= \sum_{g \in \langle h \rangle} g \in \Z G$. 
Also recall that a finite group $G$ is called {\it fixed point free} if\footnote{This class of groups coincide with the Frobenius complements \cite[Proposition 11.4.6]{EricAngel1} and hence this could serve as a group theoretical definition.} it has a complex representation $\rho$ such that $1$ is not an eigenvalue of $\rho(g)$ for all $1 \neq g \in G$.

\begin{corollary}\label[corollary]{corollarybicyclic}
If the properties of \Cref{iff HFA} are satisfied and $G$ has no non-abelian homomorphic image which is fixed point free, then the subgroup of $\U(\Z G)$ generated by all the bicyclic units is neither a non-trivial amalgamated product nor an HNN extension.  
\end{corollary}

\begin{proof}

Let $\lbrace e_i \mid 1 \leq i \leq q \rbrace$ be the primitive central idempotents of $\Q G$.
Denote by $B$ the group generated by the bicyclic units. By the properties of \Cref{iff HFA}, $\U(\Z G)$ has no exceptional components and $G$ is cut.

In particular the center of $\U(\Z G)$ is finite and, by \cite[Corollary 5.5.3]{EricAngel1}, $\prod U_i$ is a subgroup of finite index in $\U(\Z G)$, where $U_i$ is a subgroup in $\U (\Z G)$ such that $1 - e_i + U_i e_i$ is of finite index in $\SL_1(\Z Ge_i)$ for every $1 \leq i \leq q$. If $e_i$ is such that $\Q Ge_i \cong \Ma_{n_i}(D_i)$ with $n_i \geq 2$, then by \cite[Theorem 11.2.5]{EricAngel1} and the proof of \cite[Theorem 11.1.3]{EricAngel1} such a group $U_i$ exists within the group $B$. If $\Q Ge_i \cong D_i$, then (since no exceptional components exist) $D_i$ is a field or a totally definite quaternion algebra. If $D_i$ is a field, then, since $G$ is cut, it is $\Q$ or an imaginary quadratic extension of $\Q$. Hence $\SL_1(D_i)$ is always finite and thus for these $e_i$ the trivial subgroup $U_i$ can be taken. Hence we have found a subgroup of  $B$ that is of finite index in $\U(\Z G)$.

Again by the conditions of \Cref{iff HFA} $B$ has property \FA. Hence, by \Cref{FAHNN}, it is neither a non-trivial amalgamated product nor an HNN extension.
\end{proof}

Another property of interest, which is weaker than \HFA, is the so-called \FAb property.
\begin{definition}
A group $\Gamma$ is said to have \emph{property \FAb} if every subgroup of finite index has finite abelianization.
\end{definition}
Clearly property \FAb is also defined on commensurability classes. We can now deduce the following.

\begin{corollary}\label[corollary]{prop (T) and (HFA) equivalent for ZG}
Let $G$ be a finite group. Then the following properties are equivalent:
\begin{enumerate}
\item $\U(\Z G)$ has property \T,
\item $\U(\Z G)$ has property \HFR,
\item $\U(\Z G)$ has property \HFA,
\item $\U(\Z G)$ has property \FAb.
\end{enumerate}
Moreover, in these cases, if $\Q G$ has no $2\times 2$-components, $\mathcal{U}(\mathbb{Z}G)$ has property \HFA[m-2] with 
$$m =  \min \{\ n \not= 1 \mid \Ma_n(D) \text{ is an epimorphic image of }\Q G \text{, with } D \text{ a f.d. $\Q$-division algebra} \ \}.$$
\end{corollary}
\begin{proof}Suppose $\U(\Z G)$ has property \FAb, then, by \Cref{FA implies cut}, $G$ is a cut group. We will show $\U(\Z G)$ has property \T. Since property \T is defined on commensurability classes it is enough to prove that $\Gamma = \prod_{i \in I} \GL_{n_i}(\mathcal{O}_i)$ has property \T, where $\mathbb{Q}G = \prod_{i\in I} \Ma_{n_i}(D_i)$ and $\mathcal{O}_i$ is an order in $D_i$ a finite dimensional division $\Q$-algebra. The group $\Gamma$ has property \T if and only if all factors do. Since $G$ is a cut group, whenever $n_i= 1$, we have, by \Cref{no excp 1 by 1 for cut groups} and \Cref{When is unit group order finite} that $\mathcal{U}(\O_i)$ is finite, in particular it has property \T.
Furthermore, $\mathcal{Z}(D) = \Q(\sqrt{-d})$ with $d \geq 0$, for every Wedderburn-Artin component $\Ma_n(D)$ of $\Q G$. Therefore \Cref{property (T) and FA_n for non-exceptional 2 by 2}, all the non-exceptional components have property \T.

Next we show that no non-exceptional components of type (II) appear as component of $\Q G$. Recall that those exceptional components are described by \Cref{theorem EKVG}.
Since \FAb is a property of commensurability classes, we know that all $\SL_{n_i}(\O_i)$ have \FAb. However, by \Cref{remarkFAb}, no exceptional component of type (II) has property \FAb.
Hence no exceptional component of type (II) exists as a component of $\Q G$, which finishes the proof of $(4) \Rightarrow (1)$.

Since property \T implies property \HFR, cf.\ \cite[Chapter~6., Proposition~11]{dlHV}, property \HFR implies property \HFA and \HFA implies property \FAb, this finishes the proof of the four equivalences.

The last part of the result follows from \Cref{property (T) and FA_n for non-exceptional 2 by 2}, the assumption on $\Q G$ and the well behaviour of the property under direct products.
\end{proof}

\begin{remark}
Property \T, \HFR, \HFA and \FAb are all properties defined on commensurability classes. In particular, \Cref{prop (T) and (HFA) equivalent for ZG} and \Cref{iff HFA} are valid for arbitrary orders in $\Q G$.
\end{remark}

\begin{corollary}\label[corollary]{odd-order cut}
Let $G$ be a group without exceptional components of type (II) (e.g.\ $|G|$ odd). Then the following properties are equivalent:
\begin{enumerate}
\item $\mathcal{U}(\mathbb{Z}G)$ has property \HFA,
\item $\mathcal{U}(\mathbb{Z}G)$ has property \FAb,
\item\label{it:noII_abfin} $\mathcal{U}(\mathbb{Z}G)^{ab}$ is finite,
\item\label{it:noII_cut} $G$ is a cut group.
\end{enumerate}
\end{corollary}
\begin{proof}
Due to \Cref{(FA) implies cut} it only remains to prove that if $G$ is cut, then $\mathcal{U}(\mathbb{Z}G)$ has property \HFA. By assumption, $G$ has no $2\times 2$ exceptional components and due to the cut property, cf.\ \Cref{no excp 1 by 1 for cut groups}, also no exceptional $1\times1$ components. Hence \Cref{iff HFA} applies. 
\end{proof}

Let $\Ma_2(D)$ be an  exceptional component of type (II) actually appearing in $\Q G$ for a finite group $G$ (see \Cref{theorem EKVG}) and let $\O$ be an order in $D$. Then $\GL_2(\O)$ has finite abelianization by \Cref{GE2Oab_finite}. So also in the presence of exceptional components of type (II) one might anticipate that \eqref{it:noII_abfin} and \eqref{it:noII_cut} in \Cref{odd-order cut} are still equivalent. Interestingly, as proven in \Cref{Trichotomy and equivalent problems}, this is equivalent to the following trichotomy.

\begin{question}\label[question]{Trichotomy conjecture}
Let $G$ be a finite cut group. Does exactly one of the following properties hold?
\begin{enumerate}
\item $\mathcal{U}(\mathbb{Z}G)$ has property \HFA.
\item $\mathcal{U}(\mathbb{Z}G)$ has property \FA but not \HFA.
\item $\mathcal{U}(\mathbb{Z}G)$ has a non-trivial amalgamated decomposition and finite abelianization.
\end{enumerate}
\end{question}

In \cite[Theorem~8.5 and Remark~8.6]{amalgamationpaper} we prove that a dichotomy holds for $\U (\Z G)$: for $G$ a finite cut group that is solvable or has an order not divisible by $5$, $\U(\Z G)$ has property \HFA or it is commensurable with a non-trivial amalgamated product.

\begin{proposition}\label[proposition]{Trichotomy and equivalent problems}
Let $G$ be a finite group and $\O$ a maximal order of $\mathbb{Q}G$ containing $\mathbb{Z}G$. Then the following properties are equivalent:
\begin{enumerate}
\item  If $\mathcal{U}(\O)^{ab}$ is finite, then $\mathcal{U}(\mathbb{Z}G)^{ab}$ is finite,
\item  If $G$ is cut, then $\mathcal{U}(\Z G)^{ab}$ is finite,
\item \Cref{Trichotomy conjecture} has a positive answer.
\end{enumerate}
\end{proposition}

For the proof we will need the following proposition.

\begin{proposition}\label[proposition]{ab GL versus ab SL}
Let $D$ be a finite dimensional division algebra over $\mathbb{Q}$ and $\O$ an order in $D$. If $\Ma_n(D)$ is non-exceptional, then the following are equivalent:
\begin{enumerate}
\item $\mathcal{Z}(D)=\Q(\sqrt{-d})$ with $d \geq 0$,
\item $\GL_n(\O)^{ab}$ is finite,
\item $\U (\O')^{ab}$ is finite for {\it some} order $\O'$ in $\Ma_n(D)$
\item $\U (\O')^{ab}$ is finite for {\it every} order $\O'$ in $\Ma_n(D)$.
\end{enumerate}
\end{proposition}
\begin{proof}
To start note that $\SL_n(\O)^{ab}$ is finite when $\Ma_n(D)$ is non-exceptional. Indeed by \Cref{property (T) and FA_n for non-exceptional 2 by 2} (or alternatively \Cref{HFR voor E_n en SLn met n >3} if $n \geq 3$) and the fact that property (T) implies finite abelianization \cite[Corollary~1.3.6]{ValetteBook}. 

Now suppose that $\mathcal{Z}(D)=\Q(\sqrt{-d})$, with $ d \geq 0$, then, by \Cref{When is unit group order finite}, $\mathcal{U}( \mathcal{Z}( \O))$ is finite. Hence in this case, $\SL_n(\O) $ has finite index in $\GL_n(\O)$  (using \cite[Proposition~5.5.1]{EricAngel1}) and consequently $\GL_n(\O)^{ab}$ is finite due to the finiteness of $\SL_n(\O)^{ab}$. In short, $(1)$ indeed implies $(2)$.

As $\Ma_n(\O)$ is an example of an order in $\Ma_n(D)$, $(2)$ implies $(3)$. We will now prove that $(3)$ implies $(1)$. So let $\O$' be an order in $\Ma_n(D)$ such that $\U (\O')^{ab}$ is finite. Then from \Cref{(FA) implies cut} it follows that $\mathcal{U}(\mathcal{Z}(\O'))$ is also finite. Consider now the order $\Ma_n(\O)$. Then $\mathcal{Z}(\O')$ and $\mathcal{Z}(\Ma_n(\O))$ are both orders in the finite dimensional semisimple $\Q$-algebra $\mathcal{Z}(\Ma_n(D))$ and hence by \Cref{order zijn commensurable} the unit group of the two orders are commensurable. In particular also $\U (\mathcal{Z}(\Ma_n(\O))) = \U (\mathcal{Z}(\O))$ is finite and thus by \Cref{When is unit group order finite}, $\mathcal{Z}(D)=\Q(\sqrt{-d})$ with $d \geq 0$, as desired.
Hence the first three items are equivalent. 

We now prove $(2)$ implies $(4)$. Suppose $\GL_n(\O)^{ab}$ is finite and let $\O'$ be an arbitrary order in $\Ma_n(D)$. As we have already shown that the first three conditions are equivalent, we already know that $\mathcal{Z}(D)=\Q(\sqrt{-d})$ for some $d \geq 0$. Recall that $\langle \SL_1(\O'), \U (\mathcal{Z}(\O')) \rangle$ is of finite index in $\U (\O')$ (see \cite[Proposition~5.5.1]{EricAngel1}). Also $\U( \O' \cap \Ma_n(\O))$ equals $\U(\O') \cap \GL_n(\O)$ and it is of finite index in both $\U(\O')$ and $\GL_n(\O)$, since the unit groups of two orders are commensurable by \Cref{order zijn commensurable}. So altogether we obtain that $\U(\O') \cap \SL_n(\O)$ is of finite index in $\U(\O')$ and $\SL_n(\O)$, which has property (T) by \Cref{property (T) and FA_n for non-exceptional 2 by 2}. In particular, $\U(\O') \cap \SL_n(\O)$ also has property \T and thus finite abelianization. This implies that also $\U (\O')^{ab}$ is finite, as desired.
The remaining implication $(4)$ to $(3)$ is trivial. 
\end{proof}

Note that \Cref{ab GL versus ab SL} yields a positive answer to \Cref{conjecture ab independent order} for non-exceptional finite dimensional simple $\Q$-algebras. 

\begin{proof}[Proof of \Cref{Trichotomy and equivalent problems}]
First we prove that (1) implies (2).
Let $\Q G = \prod_{i \in \mathcal{I}} \Ma_{n_i}(D_i)$ be the Wedderburn-Artin decomposition of $\Q G$, $e_i$ the primitive central idempotent corresponding to $\Ma_{n_i}(D_i)$ and $\O_i$ an order in $\Ma_{n_i}
(D_i)$.
Write $\mathcal{I}$ as the (disjoint) union of three sets $I_1,I_2$ and $I_3$ where $I_1$ are those indices corresponding to $1\times 1$ components, $I_2$ those with exceptional $2\times 2$ components and $I_3$ consisting of the remaining components.
Suppose $G$ is cut, then there are no exceptional $1\times 1$ component by \Cref{no excp 1 by 1 for cut groups}. Consequently, by \Cref{When is unit group order finite}, $\mathcal{U}(\O_i)^{ab}$ is finite for all $i \in I_1$.
Also $\U(\O_i)^{ab}$ is finite for any order $\O_i$ in $\Ma_n(D_i)$ with $i \in I_3$ by \Cref{ab GL versus ab SL} (recall that $G$ cut implies that $\ZZ (D_i) = \Q (\sqrt{-d})$ with $d \geq 0$). 

Let now $i \in I_2$.
Then, by \Cref{theorem EKVG} and \Cref{exc comp have unique max order}, $D_i$ has up to conjugation a unique maximal order, say $\O_{max, i}$, which is right norm Euclidean and hence $\GE_2(\O_{max, i}) = \GL_2(\O_{max, i})$.
By \Cref{GE2Oab_finite}, $\GE_2(\O_{max,i})^{ab}$ and hence also $\GL_2(\O_{max, i})^{ab}$ are finite.
Altogether, if $G$ is cut, then for any choice of orders $\O_i$ in $\Ma_{n_i}(D_i)$, when $i \in I_1 \cup I_3$ we have that 
\begin{equation}\label[equation]{eq: order with finite ab in QG}
\left| \prod_{i \in I_1 \cup I_3} \U(\O_i)^{ab} \times \prod_{j \in I_2} \GL_2(\O_{max,j})^{ab} \right| < \infty.
\end{equation}
As, by assumption, $\O$ is a maximal order of $\Q G$ containing $\Z G$, $\O \cong \prod_{i \in \I} \O_i$ with $\O_i = \O e_i$ a maximal order in $\Ma_{n_i}(D_i)$. As mentioned above, for $i\in I_2$, the maximal order $\O_i$ is conjugate to $\Ma_2(\O_{max, i})$. Since the size of the abelianization is preserved under conjugation we may assume that $\O = \prod_{i \in I_1 \cup I_3} \O_i \times \prod_{j \in I_2} \Ma_2(\O_{max,j})$ which has a unit group with finite abelianization by \eqref{eq: order with finite ab in QG}. Consequently, by (1), also $\mathcal{U}(\Z G)^{ab}$ is finite.  This finishes the proof of (1) implies (2).

We now prove that (2) implies (1). Hence assume statement (2) is true. Let $\O$ be a maximal order of $\Q G$ and suppose that $\mathcal{U}(\O)^{ab}$ is finite. Then by \Cref{(FA) implies cut}, $\mathcal{U}(\mathcal{Z}(\O))$ is finite. Consequently, since $\mathcal{Z}(\Z G)$ and $\mathcal{Z}(\O)$ are both orders in $\mathcal{Z}(\Q G)$, also $\mathcal{U}(\mathcal{Z}(\Z G))$ is finite. Hence, $G$ is cut and thus by (2), $\mathcal{U}(\Z G)$ is finite, as desired. So, we have proved that (1) and (2) are equivalent.

We will now prove that (2) implies a positive answer to \Cref{Trichotomy conjecture}. Suppose that $G$ is cut and hence that $\mathcal{U}(\Z G)^{ab}$ is finite by (2). Then if $\mathcal{U}(\Z G)$ does not have property (FA), it must have a non-trivial amalgamated decomposition by \Cref{prop:iff_conditions_FA} as desired. Conversely, in all the cases when \Cref{Trichotomy conjecture} has a positive answer, the abelianization of $\mathcal{U}(\Z G)$ is finite, so (3) clearly implies (2). 
\end{proof}

\section{Unit groups of group rings and Property \FA}\label{sectie FA voor UZG}

In this section  we consider when  $\mathcal{U}(\Z G)$ has \FA and when it has \FA but not \HFA.

\begin{theorem}\label{U(ZG) FA voor willekeurig groepen}
Let $G$ be a finite solvable group and assume that $\mathcal{U}(\Z G)$ has \FA. Then the following properties hold:
\begin{enumerate}
\item $G$ does not map epimorphically on $D_8$ and  $S_3$.
\item $\mathcal{U}(\Z G)$ does not satisfy \HFA if and only if $G$ maps onto one of the following $7$ groups
\begin{align*}  G_{16,6},\quad G_{16,13},\quad Q_8 \times C_3, \quad \SL(2,3),\quad G_{32,50},\quad G_{96, 202},\quad \text{and} \quad G_{384, 618}.
 \end{align*}
\end{enumerate}
\end{theorem}
\begin{proof}
Let $\{ e_i \mid i \in I \}$ be a full set of primitive central idempotents of $\Q G$ with $\Q Ge_i \cong \Ma_{n_i}(D_i)$ for $i \in I$ and decompose $I = I_1 \cup I_2$ in such a way that $\Q G e_i$ is an exceptional component for all $i \in I_1$ and non-exceptional for all $i \in I_2$.
Since, by assumption, $\mathcal{U}(\Z G)$ has property \FA, $G$ must be cut by \Cref{FA implies cut} and consequently, by \Cref{no excp 1 by 1 for cut groups}, all exceptional components are $2\times2$ matrix rings. The latter have, by \Cref{theorem EKVG} and \Cref{exc comp have unique max order}, up to conjugation in $\Q G$, a unique maximal order.
So,  without loss of generality, we may assume that $\Z G$ is a subring of the order $\prod_{i \in I_1} \Ma_2(\O_i) \times \prod_{j \in I_2} \Z G e_j$, where $\O_i$ is a maximal order of $D_i$. Since orders have commensurable unit groups, $\mathcal{U}(\Z G)$ has finite index in $\prod_{i \in I_1} \GL_2(\O_i) \times \prod_{j \in I_2} \mathcal{U}(\Z G e_j)$. Therefore, the latter also enjoys property \FA, and thus $\GL_2(\O_i)$ has \FA for $i \in I_1$ (also $\mathcal{U}(\Z G e_j)$ has \FA for all $j \in I_2$, however as $G$ is cut, it follows from \Cref{property (T) and FA_n for non-exceptional 2 by 2} that they all even have \HFA and hence they do not add any restriction). 

Now recall that $\GL_2(\Z)$ is a non-trivial amalgamated product, see \cite[Proposition~25]{AmalgamGL2Z}. Thus for all $i \in I_1$, $\O_i \not\cong \Z$ and so by \Cref{possible 2 times 2 components}, $G$ cannot map onto $S_3$ or $D_8$, proving (1). Now (2) follows from \Cref{iff HFA} and the fact that $G_{240, 90}$ is not solvable.
\end{proof}

In case of nilpotent groups a more precise statement can be given.

\begin{corollary}
Let $G$ be a finite nilpotent group and assume that $\mathcal{U}(\Z G)$ has \FA. Then the following properties hold:
\begin{enumerate}
\item $G$ does not map epimorphically on $D_8$.
\item $\mathcal{U}(\Z G)$ does not satisfy \HFA if and only if $G$ has $G_{16,6}$, $G_{16,13}$, $G_{32,50}$ or $Q_8 \times C_3$ as epimorphic image.
\end{enumerate}
\end{corollary}
\begin{proof}
For the first statement simply apply \Cref{U(ZG) FA voor willekeurig groepen} and note that $S_3$  is not nilpotent and hence cannot be a quotient of the nilpotent group $G$. A similar reasoning can be given for the second statement. Indeed $\SL(2,3), G_{96, 202}$ and  $G_{384, 618}$ are not nilpotent.
\end{proof}

It is natural to ask whether in \Cref{U(ZG) FA voor willekeurig groepen} and its corollary the converse of the first statement holds.
This problem is connected to the problem whether \FA for $\mathcal{U}(\Z G)$ is fully determined by the Wedderburn-Artin components of $\Q G$ (as in the hereditary case). In order to formulate some concrete questions we fix the following notations: $\Q G \cong \prod \Ma_{n_i}(D_i)$, $\Ma_{n_i}(D_i) = \Q G e_i$ with $e_i$ a central primitive idempotent of $\Q G$ and $\O_i$ a maximal order in $\Ma_{n_i}(D_i)$. 

By the proof of \Cref{U(ZG) FA voor willekeurig groepen} we know that if $\U (\Z G)$ has property $\FA$, then $\mathcal{U}(\O_i)$ has property \FA for all $i$.

\begin{question} \label[question]{questions_FA}
With notations as above:
\begin{enumerate}
\item Does $\mathcal{U}(\mathbb{Z}G)$ have property \FA if and only if $\mathcal{U}(\O_i)$ has property \FA for all $i$?
\item Does $\mathcal{U}(\mathbb{Z}G)$ have property \FA if and only if $\mathcal{U}(\mathbb{Z}G e_i)$ has property \FA for all $i$? 
\end{enumerate}
\end{question}

Note that the previous questions are connected to all the properties in \Cref{Trichotomy and equivalent problems}. Unfortunately, in general property \FA is dependant on the chosen order (in contrast to having finite center or finite $K_1$).
Indeed $\Ma_2(\mathbb{Z}[\sqrt{-3}])$ and $\Ma_2(\mathcal{I}_3)$ are both orders in $\Ma_2(\Q(\sqrt{-3}))$ but $\SL_2(\mathbb{Z}[\sqrt{-3}])$ is an amalgamated product by \cite[Theorem~4.2]{amalgamationpaper} whereas $\SL_2(\mathcal{I}_3)$  has property \FA by the remark just before \Cref{When is E2 FA and some GL2 FA}.

\begin{remark}
We expect \Cref{questions_FA} (1) to not be true in general. For example suppose that the only exceptional components are of type $\Ma_2(D)$ with $D \in \{ \Q(\sqrt{-1}), \Q(\sqrt{-3}),\mathbb{H}_2, \mathbb{H}_3 \}$ and let $\O$ be the unique maximal order in $D$. The projection $\U(\Z Ge_i)$ of $\U(\Z G)$ into that exceptional component $\Q Ge_i = \Ma_2(D)$ will be a subgroup of finite index in $\GL_2(\O)$, however usually not of index $1$. Now the group $\GL_2(\O)$ has a subgroup of very small index which has not property \FA. So it seems very plausible that even with exceptional components as above, $\U(\Z Ge_i)$ sometimes will not have property \FA and hence also not $\U(\Z G)$.
\end{remark}

This last remark also ties into the following very natural question.

\begin{question}\label[question]{question_notHFA}
Is there a finite cut group $G$, such that $\U(\Z G)$ has property \FA, but does not have property \HFA?
\end{question}

An explicit positive answer to this last question could also be used to study \Cref{questions_FA}. A negative answer on the other hand has consequences on \Cref{Trichotomy conjecture}, making the trichotomy into a dichotomy.

In order to show some properties of unit groups $\U(R$), it is common in the literature (and in our \Cref{sectie FA voor GL2's}) to blow up the group to a significantly larger group. For example, $\GE_2(R)$ and $\B_2(R)$ often help in studying properties for $\U(R)$.
In the case of $\U(\Z G)$ however, this will not work.

\begin{proposition}\label[proposition]{GE2(ZG) not FA}
Let $G$ be a finite group. Then $\B_2(\Z G)$, $\E_2(\Z G)$ and $\GE_2(\Z G)$ do not have property \FA. 
\end{proposition}

\begin{proof}
The augmentation map $$\omega : \Z G \rightarrow \Z : \displaystyle\sum\limits_{g\in G} a_g g \mapsto  \displaystyle\sum\limits_{g\in G} a_g, $$ is an epimorphism of rings. We may extend this morphism to an epimorphism of groups $\Omega : \GE_2(\Z G) \rightarrow \GE_2(\Z)$. Since  $\GE_2(\Z) = \GL_2(\Z)$ does not have property \FA, $\GE_2(\Z G)$ also does not have property \FA. The same reasoning works for $\E_2(\Z G)$ and $\E_2(\Z)=\SL_2(\Z)$. The augmentation map $\omega$ also induces an epimorphism from $\B_2(\Z G)$ to $\B_2(\Z)$ by letting $\omega$ act entry wise. However $\B_2(\Z)$ is non-trivially an amalgamated product by \Cref{Borel FA cyclic units}. Hence also $\B_2(\Z G)$ does not have property \FA
\end{proof}

Remark that for $\E_2(\Z G)$ and $\GE_2(\Z G)$ we only used that $\Z G$ has a ring epimorphism to $\Z$, so in those cases the proof works for any ring $R$ with a ring epimorphism to $\Z$. More generally the following holds.

\begin{proposition}\label[proposition]{Fa voor borel when basis of units ring theoretically}
Let $R$ be a unital ring that has a finite basis as $\Z$-module consisting of units. Then $\B_2(R)$ has property \FA if and only if $\U (R)$ has property \FA and $R$ has no ring epimorphism onto $\Z$.
\end{proposition}
\begin{proof}
For the same reasons as in the proof of \Cref{GE2(ZG) not FA}, the conditions that $R$ has no ring epimorphism onto $\Z$ is necessary. Moreover, $\U (R)$ is an epimorphic image of $\B_2(R)$, so also this condition is necessary. We will now prove that they are also sufficient. 

Note that in the implication $(1) \Rightarrow (2)$ of the proof of \Cref{equivalence finite ab for almost universal rings}, the fact that the ring is almost-universal is not used. So, since $\U(R)$ has \FA, it has finite abelianization which in turn implies that $\B_2(R)$ has finite abelianization.

Suppose $(R,+) \cong (\Z^n,+)$, i.e. $R$ has $\Z$-module basis of cardinality $n$. Recall that $\B_2(R) \cong N \rtimes \D_2(R)$ is a semi-direct product where $N$ is isomorphic to the additive group of $R$. As in the proof of \Cref{equivalence finite ab for almost universal rings} one proves that $N^2 \leq \B_2(R)'$. Since also $D_2(R)' \leq \B_2(R)'$, this shows that $\B_2(R)^{ab}$ is an epimorphic image of the group $C_2^{n} \times D_2(R)^{ab} \cong C_2^{n} \times \U(R)^{ab} \times \U(R)^{ab}$.
 Hence, the assumption provides that $\B_2(R)$ has finite abelianization. In order to prove it has property \FA, it suffices thus to show that it is not an amalgam. Suppose, by contradiction, that $\B_2(R) \cong A \ast_U B$ for some subgroups $A,B,U \leq \B_2(R)$.
 We will show this is impossible by considering the abelian normal subgroup $N$.
 
 This subgroup $N$, being abelian, is contained in the maximal normal subgroup of $\B_2(R)$ not containing any free subgroup, denoted by $NF(\B_2(R))$ (see \cite{deCornulier}) and is well-defined. This implies that $N \leq U$, by \cite[Proposition $7$]{deCornulier} or that the amalgam decomposition is so-called degenerate meaning that $U$ has index $2$ in both $A$ and $B$ and thus is normal in the whole group.
 
 Assume the first. Now we can on the one hand consider $\B_2(R) /N \cong A/N \ast_{U/N} B/N$ as a non-trivial amalgamated product, but on the other hand $\B_2(R) /N  \cong \D_2(R) \cong \U(R) \times \U(R)$, which is a group having \FA by assumption.
 This is a contradiction.
 
 If $N \nleq U$, then $U$ is of index $2$ in $A$ and $B$ and thus also normal in the whole group. Since $N$ is not a non-trivial amalgamated product (indeed, it is abelian), we know by the work of Karass and Solitar \cite[Corollary of Theorem $6$]{KarSol} that $N$ is one of the following three types of groups. \begin{enumerate}
 \item $N$ is contained in a conjugate of $A$ or $B$,
 \item $N = \bigcup\limits_{i=1}^{\infty} \left(U^{\alpha_i} \cap N \right)$ is an infinite ascending union for some $\alpha_i \in \B_2(R)$,
 \item $N = \langle z \rangle \times M$, with $z$ an element of infinite order and $M = N\cap U \cong C_{\infty}^{n-1}$.
 \end{enumerate}
 
 If $N$ is contained in a conjugate of $A$ or $B$, it should be in $A$ or $B$ since it is normal. Using the fact that it is normal would even imply that $N \leq U$, which is a contradiction.
 
 Suppose the second case is true, then $N = \bigcup\limits_{i=1}^{\infty} \left(U^{\alpha_i} \cap N \right) = \bigcup\limits_{i=1}^{\infty} \left(U \cap N \right) = U \cap N$ since $U$ is normal, but this again contradicts the fact that $N \nleq U$.
 
In the last case, the subgroup $M$ is moreover normal in $\B_2(R)$. Denote $M = \left( \begin{matrix}
1 & \widetilde{M} \\ 0 & 1
\end{matrix} \right) $. Then $\widetilde{M} \cong \Z^{n-1}$ and $(R,+) \cong \Z \oplus \widetilde{M}$. Now since $R$ has a $\Z$-module basis consisting of units and $M$ is normal in $\B_2(R)$, we get that $\widetilde{M}$ is a two-sided ideal of $R$. Indeed, it suffices to remark that, for any units $g$ and $h$ of $R$:$$\begin{pmatrix} g & 0 \\ 0 & h \end{pmatrix}^{-1} M \begin{pmatrix} g & 0 \\ 0 & h \end{pmatrix} = \begin{pmatrix} 1 & g^{-1}\widetilde{M}h \\ 0 & 1 \end{pmatrix}.$$
Therefore we may form $R/\widetilde{M}$ which is easily seen to be also isomorphic to $\Z$ as rings. However this contradicts the fact that there is no ring epimorphism from $R$ to $\Z$. So $\B_2(R)$ is also not an amalgamated product and thus altogether has property \FA, as needed.
\end{proof}

\begin{remark}
Let $R$ be a ring as in \Cref{Fa voor borel when basis of units ring theoretically}. Then $R$ is an epimorphic image of the group ring $\Z \U(R)$. We can extend this morphism to a group morphism from $\GE_2(\Z \U(R))$ (or $\E_2(\Z \U(R))$) to $\GE_2(R)$ (or $\E_2(R)$). It might thus be tempting to deduce property \FA for $\GE_2(R)$ and $\E_2(R)$ from the same properties of the same groups over the universal object $\Z \U(R)$. However, \Cref{GE2(ZG) not FA} shows that this argument is too simple and should $\GE_2(R)$ and $\E_2(R)$ have property \FA, then it is for more subtle reasons. This also shows why, in \Cref{sectie FA voor GL2's}, we did not use this universal object.
\end{remark}

\vspace{1cm}

\noindent \textbf{Acknowledgment.} We would very much like to thank John Voight for the interesting clarifications on orders in quaternion algebras. A special thanks is also required for Jan De Beule and Aurel Page and their assistance with computer algebra programs. We are also thankful to Norbert Kraemer and Alexander D. Rahm for their insights they shared with us. Moreover, we thank Shengkui Ye for bringing \cite{Ye} to our attention. Finally, we are very grateful to the referee for carefully reading the paper and for numerous valuable suggestions.

\appendix

\def\sectionmark#1{\markboth{APPENDIX}{\thesection\quad\uppercase{#1}}} 

\addtocontents{toc}{\protect\setcounter{tocdepth}{1}}

\clearpage
\section{Groups with faithful exceptional $2 \times 2$ components}\label{Appendix tabel met alle data}

In this appendix we reproduce \cite[Table~2]{EKVG} listing those finite groups $G$ that have a faithful exceptional component of type (II) in the Wedderburn-Artin decomposition of the rational group algebra $\mathbb{Q}G$ (see \Cref{definitie exceptional component})\footnote{including the group with \textsc{SmallGroupID} \texttt{[24, 1]} that was accidentally omitted in the original table}. We also add certain attributes relevant for us. For each group $G$ such that $\mathbb{Q}G$ has at least one exceptional $2 \times 2$ component in which $G$ embedds (``faithful component'') the columns of the table contain the following information:

\vspace{.5cm}

\noindent\begin{tabular}{p{0.25\linewidth}p{0.7\linewidth}}
 \textsc{SmallGroup ID}: & the identifier of the group $G$ in the small group library \\[.4cm] 
 Structure: & the structure description of the group. Colons indicate split extensions, a period an extension that is (possibly) non-split \\[.4cm]
 cut: & indicates whether the group is a cut group (see \Cref{def:cut_group}) \\[.4cm]
 d$\ell$: & derived length of the group; $\infty$ for non-solvable groups \\[.2cm]
 c$\ell$: & the nilpotency class of the group; $\infty$ indicates that the group is not nilpotent (omitted for non-solvable groups) \\[.2cm]
 exceptional components of type (II): & exceptional components of type (II) (not necessarily faithful) of the group algebra $\mathbb{Q}G$ (with multiplicity) \\[.2cm]
 quotients: & small group IDs of non-trivial quotients of $G$ that also appear in this table.
 
\end{tabular}\vspace{0,1cm}

\noindent Recall that we use the following shorthands for some algebras appearing in the table:
\[\Q(i) = \Q(\sqrt{-1}), \quad \mathbb{H}_2 = \qa{-1}{-1}{\mathbb{Q}},\quad \mathbb{H}_3 = \qa{-1}{-3}{\mathbb{Q}}\quad \mbox{ and }\quad \mathbb{H}_5 = \qa{-2}{-5}{\mathbb{Q}}.\] \vspace{0,1cm}

\noindent One way to check whether a group has the cut-property is via the following GAP code:

\begin{verbatim}
IsCutGroup := function(G)
return
  ForAll( List( ConjugacyClasses(G) , Representative ),
    x ->
      ForAll( Filtered( [2..Order(x)-1], j -> Gcd(j, Order(x)) = 1 ),
        j -> IsConjugate (G, x^j, x) or IsConjugate(G, x^j, x^-1)
      )
  );
end;
\end{verbatim}

\newgeometry{left=0cm,right=0cm,top=0cm,bottom=0cm}
 \begin{landscape}
 {\scriptsize
\begin{longtable}{@{}llcccll@{}} \\ \toprule[1.5pt]
\textsc{SmallGroupID} & Structure & cut & d$\ell$ & c$\ell$ & exceptional components of type (II) & quotients  \\ \midrule 
\endfirsthead \toprule[1.5pt] \textsc{SmallGroupID} & Structure & cut & d$\ell$ & c$\ell$ & exceptional components of type (II) & quotients  \\ \midrule 
\endhead \hline \multicolumn{6}{c}{continued}\\ \midrule[1.5pt]\endfoot\bottomrule[1.5pt]\endlastfoot
\textsf{{[}6, 1{]}} & $ S_3 $ & $\checkmark$ & $2$ & $\infty$ & $1 {\times} M_2(\mathbb{Q})$  &  \\
\textsf{{[}8, 3{]}} & $ D_8 $ & $\checkmark$ & $2$ & $2$ & $1 {\times} M_2(\mathbb{Q})$  &  \\
\textsf{{[}12, 4{]}} & $ D_{12} $ & $\checkmark$ & $2$ & $\infty$ & $2 {\times} M_2(\mathbb{Q})$  & \textsf{{[}6, 1{]}}  \\
\textsf{{[}16, 6{]}} & $ C_8 :  C_2$ & $\checkmark$ & $2$ & $2$ & $1 {\times} M_2(\mathbb{Q}(i))$  &  \\
\textsf{{[}16, 8{]}} & $QD16$ & $\checkmark$ & $2$ & $3$ & $1 {\times} M_2(\mathbb{Q})$, $1 {\times} M_2(\mathbb{Q}(\sqrt{-2}))$  & \textsf{{[}8, 3{]}}  \\
\textsf{{[}16, 13{]}} & $( C_4  \times   C_2) :  C_2$ & $\checkmark$ & $2$ & $2$ & $1 {\times} M_2(\mathbb{Q}(i))$  &  \\
\textsf{{[}18, 3{]}} & $ C_3  \times   S_3 $ & $\checkmark$ & $2$ & $\infty$ & $1 {\times} M_2(\mathbb{Q})$, $1 {\times} M_2(\mathbb{Q}(\sqrt{-3}))$  & \textsf{{[}6, 1{]}}  \\
\textsf{{[}24, 1{]}} & $ C_3 :  C_8$ & $\times$ & $2$ & $\infty$ & $1 {\times} M_2(\mathbb{Q})$, $1 {\times} M_2(\mathbb{Q}(i))$  & \textsf{{[}6, 1{]}}  \\
\textsf{{[}24, 3{]}} & $ {\rm SL} (2,3)$ & $\checkmark$ & $3$ & $\infty$ & $1 {\times} M_2(\mathbb{Q}(\sqrt{-3}))$  &  \\
\textsf{{[}24, 5{]}} & $ C_4  \times   S_3 $ & $\checkmark$ & $2$ & $\infty$ & $2 {\times} M_2(\mathbb{Q})$, $1 {\times} M_2(\mathbb{Q}(i))$  & \textsf{{[}6, 1{]}}, \textsf{{[}12, 4{]}}  \\
\textsf{{[}24, 8{]}} & $( C_6  \times   C_2) :  C_2$ & $\checkmark$ & $2$ & $\infty$ & $3 {\times} M_2(\mathbb{Q})$, $1 {\times} M_2(\mathbb{Q}(\sqrt{-3}))$  & \textsf{{[}6, 1{]}}, \textsf{{[}8, 3{]}}, \textsf{{[}12, 4{]}}  \\
\textsf{{[}24, 10{]}} & $ C_3  \times   D_8 $ & $\checkmark$ & $2$ & $2$ & $1 {\times} M_2(\mathbb{Q})$, $1 {\times} M_2(\mathbb{Q}(\sqrt{-3}))$  & \textsf{{[}8, 3{]}}  \\
\textsf{{[}24, 11{]}} & $ C_3  \times   Q_8 $ & $\checkmark$ & $2$ & $2$ & $1 {\times} M_2(\mathbb{Q}(\sqrt{-3}))$  &  \\
\textsf{{[}32, 8{]}} & $ ( C_2  \times   C_2) . ( C_4  \times   C_2)$ & $\checkmark$ & $2$ & $3$ & $2 {\times} M_2(\mathbb{Q})$, $1 {\times} M_2(\mathbb{H}_2)$  & \textsf{{[}8, 3{]}}  \\
\textsf{{[}32, 11{]}} & $( C_4  \times   C_4) :  C_2$ & $\checkmark$ & $2$ & $3$ & $2 {\times} M_2(\mathbb{Q})$, $2 {\times} M_2(\mathbb{Q}(i))$  & \textsf{{[}8, 3{]}}  \\
\textsf{{[}32, 44{]}} & $( C_2  \times   Q_8 ) :  C_2$ & $\checkmark$ & $2$ & $3$ & $2 {\times} M_2(\mathbb{Q})$  & \textsf{{[}8, 3{]}}  \\
\textsf{{[}32, 50{]}} & $( C_2  \times   Q_8 ) :  C_2$ & $\checkmark$ & $2$ & $2$ & $1 {\times} M_2(\mathbb{H}_2)$  &  \\
\textsf{{[}36, 6{]}} & $ C_3  \times  ( C_3 :  C_4)$ & $\times$ & $2$ & $\infty$ & $1 {\times} M_2(\mathbb{Q})$, $2 {\times} M_2(\mathbb{Q}(\sqrt{-3}))$  & \textsf{{[}6, 1{]}}, \textsf{{[}18, 3{]}}  \\
\textsf{{[}36, 12{]}} & $ C_6  \times   S_3 $ & $\checkmark$ & $2$ & $\infty$ & $2 {\times} M_2(\mathbb{Q})$, $2 {\times} M_2(\mathbb{Q}(\sqrt{-3}))$  & \textsf{{[}6, 1{]}}, \textsf{{[}12, 4{]}}, \textsf{{[}18, 3{]}}  \\
\textsf{{[}40, 3{]}} & $ C_5 :  C_8$ & $\times$ & $2$ & $\infty$ & $1 {\times} M_2(\mathbb{H}_5)$  &  \\
\textsf{{[}48, 16{]}} & $( C_3 :  Q_8 ) :  C_2$ & $\checkmark$ & $2$ & $\infty$ & $3 {\times} M_2(\mathbb{Q})$, $1 {\times} M_2(\mathbb{Q}(\sqrt{-2}))$, $1 {\times} M_2(\mathbb{Q}(\sqrt{-3}))$, $1 {\times} M_2(\mathbb{H}_2)$  & \textsf{{[}6, 1{]}}, \textsf{{[}8, 3{]}}, \textsf{{[}12, 4{]}}, \textsf{{[}16, 8{]}}, \textsf{{[}24, 8{]}}  \\
\textsf{{[}48, 18{]}} & $ C_3 :  Q_{16} $ & $\times$ & $2$ & $\infty$ & $3 {\times} M_2(\mathbb{Q})$, $1 {\times} M_2(\mathbb{Q}(\sqrt{-3}))$, $1 {\times} M_2(\mathbb{H}_3)$  & \textsf{{[}6, 1{]}}, \textsf{{[}8, 3{]}}, \textsf{{[}12, 4{]}}, \textsf{{[}24, 8{]}}  \\
\textsf{{[}48, 28{]}} & $  {\rm SL} (2,3) .  C_2$ & $\times$ & $4$ & $\infty$ & $1 {\times} M_2(\mathbb{Q})$, $1 {\times} M_2(\mathbb{H}_3)$  & \textsf{{[}6, 1{]}}  \\
\textsf{{[}48, 29{]}} & $ {\rm GL} (2,3)$ & $\checkmark$ & $4$ & $\infty$ & $1 {\times} M_2(\mathbb{Q})$, $1 {\times} M_2(\mathbb{Q}(\sqrt{-2}))$  & \textsf{{[}6, 1{]}}  \\
\textsf{{[}48, 33{]}} & $(( C_4  \times   C_2) :  C_2) :  C_3$ & $\times$ & $3$ & $\infty$ & $1 {\times} M_2(\mathbb{Q}(i))$  &  \\
\textsf{{[}48, 39{]}} & $( C_4  \times   S_3 ) :  C_2$ & $\checkmark$ & $2$ & $\infty$ & $4 {\times} M_2(\mathbb{Q})$, $1 {\times} M_2(\mathbb{Q}(i))$, $1 {\times} M_2(\mathbb{H}_3)$  & \textsf{{[}6, 1{]}}, \textsf{{[}12, 4{]}}, \textsf{{[}16, 13{]}}  \\
\textsf{{[}48, 40{]}} & $ Q_8   \times   S_3 $ & $\checkmark$ & $2$ & $\infty$ & $4 {\times} M_2(\mathbb{Q})$, $1 {\times} M_2(\mathbb{H}_2)$  & \textsf{{[}6, 1{]}}, \textsf{{[}12, 4{]}}  \\
\textsf{{[}64, 37{]}} & $ ( C_4  \times   C_2) . ( C_4  \times   C_2)$ & $\checkmark$ & $2$ & $4$ & $2 {\times} M_2(\mathbb{Q})$, $2 {\times} M_2(\mathbb{H}_2)$  & \textsf{{[}8, 3{]}}  \\
\textsf{{[}64, 137{]}} & $( C_4 :  Q_8 ) :  C_2$ & $\checkmark$ & $2$ & $3$ & $6 {\times} M_2(\mathbb{Q})$, $2 {\times} M_2(\mathbb{H}_2)$  & \textsf{{[}8, 3{]}}  \\
\textsf{{[}72, 19{]}} & $( C_3  \times   C_3) :  C_8$ & $\times$ & $2$ & $\infty$ & $2 {\times} M_2(\mathbb{H}_3)$  &  \\
\textsf{{[}72, 20{]}} & $( C_3 :  C_4)  \times   S_3 $ & $\checkmark$ & $2$ & $\infty$ & $4 {\times} M_2(\mathbb{Q})$, $1 {\times} M_2(\mathbb{Q}(i))$, $1 {\times} M_2(\mathbb{H}_3)$  & \textsf{{[}6, 1{]}}, \textsf{{[}12, 4{]}}, \textsf{{[}24, 5{]}}  \\
\textsf{{[}72, 22{]}} & $( C_6  \times   S_3 ) :  C_2$ & $\checkmark$ & $2$ & $\infty$ & $5 {\times} M_2(\mathbb{Q})$, $2 {\times} M_2(\mathbb{Q}(\sqrt{-3}))$, $1 {\times} M_2(\mathbb{H}_3)$  & \textsf{{[}6, 1{]}}, \textsf{{[}8, 3{]}}, \textsf{{[}12, 4{]}}, \textsf{{[}24, 8{]}}  \\
\textsf{{[}72, 24{]}} & $( C_3  \times   C_3) :  Q_8 $ & $\times$ & $2$ & $\infty$ & $4 {\times} M_2(\mathbb{Q})$, $1 {\times} M_2(\mathbb{H}_3)$  & \textsf{{[}6, 1{]}}, \textsf{{[}12, 4{]}}  \\
\textsf{{[}72, 25{]}} & $ C_3  \times   {\rm SL} (2,3)$ & $\checkmark$ & $3$ & $\infty$ & $4 {\times} M_2(\mathbb{Q}(\sqrt{-3}))$  & \textsf{{[}24, 3{]}}  \\
\textsf{{[}72, 30{]}} & $ C_3  \times  (( C_6  \times   C_2) :  C_2)$ & $\checkmark$ & $2$ & $\infty$ & $3 {\times} M_2(\mathbb{Q})$, $6 {\times} M_2(\mathbb{Q}(\sqrt{-3}))$  & \textsf{{[}6, 1{]}}, \textsf{{[}8, 3{]}}, \textsf{{[}12, 4{]}}, \textsf{{[}18, 3{]}}, \textsf{{[}24, 8{]}}, \textsf{{[}24, 10{]}}, \textsf{{[}36, 12{]}}  \\
\textsf{{[}96, 67{]}} & $ {\rm SL} (2,3) :  C_4$ & $\checkmark$ & $4$ & $\infty$ & $1 {\times} M_2(\mathbb{Q})$, $2 {\times} M_2(\mathbb{Q}(i))$  & \textsf{{[}6, 1{]}}  \\
\textsf{{[}96, 190{]}} & $( C_2  \times   {\rm SL} (2,3)) :  C_2$ & $\checkmark$ & $4$ & $\infty$ & $2 {\times} M_2(\mathbb{Q})$  & \textsf{{[}6, 1{]}}, \textsf{{[}12, 4{]}}  \\
\textsf{{[}96, 191{]}} & $  {\rm SL} (2,3) .  C_2) :  C_2$ & $\times$ & $4$ & $\infty$ & $2 {\times} M_2(\mathbb{Q})$  & \textsf{{[}6, 1{]}}, \textsf{{[}12, 4{]}}  \\
\textsf{{[}96, 202{]}} & $(( C_2  \times   Q_8 ) :  C_2) :  C_3$ & $\checkmark$ & $3$ & $\infty$ & $1 {\times} M_2(\mathbb{H}_2)$  &  \\
\textsf{{[}120, 5{]}} & $ {\rm SL} (2,5)$ & $\times$ & $\infty$ &  & $1 {\times} M_2(\mathbb{H}_3)$  &  \\
\textsf{{[}128, 937{]}} & $( Q_8   \times   Q_8 ) :  C_2$ & $\checkmark$ & $3$ & $4$ & $6 {\times} M_2(\mathbb{Q})$, $4 {\times} M_2(\mathbb{H}_2)$  & \textsf{{[}8, 3{]}}  \\
\textsf{{[}144, 124{]}} & $  {\rm SL} (2,3) .  C_2)$ & $\times$ & $4$ & $\infty$ & $4 {\times} M_2(\mathbb{Q})$, $4 {\times} M_2(\mathbb{H}_3)$  & \textsf{{[}6, 1{]}}, \textsf{{[}48, 28{]}}  \\
\textsf{{[}144, 128{]}} & $ S_3   \times   {\rm SL} (2,3)$ & $\checkmark$ & $3$ & $\infty$ & $1 {\times} M_2(\mathbb{Q})$, $3 {\times} M_2(\mathbb{Q}(\sqrt{-3}))$, $1 {\times} M_2(\mathbb{H}_2)$  & \textsf{{[}6, 1{]}}, \textsf{{[}18, 3{]}}, \textsf{{[}24, 3{]}}  \\
\textsf{{[}144, 135{]}} & $( C_3  \times   C_3) : ( C_8 :  C_2)$ & $\checkmark$ & $2$ & $\infty$ & $1 {\times} M_2(\mathbb{Q}(i))$, $4 {\times} M_2(\mathbb{H}_3)$  & \textsf{{[}16, 6{]}}  \\
\textsf{{[}144, 148{]}} & $( C_3  \times   C_3) : (( C_4  \times   C_2) :  C_2)$ & $\checkmark$ & $2$ & $\infty$ & $8 {\times} M_2(\mathbb{Q})$, $1 {\times} M_2(\mathbb{Q}(i))$, $4 {\times} M_2(\mathbb{H}_3)$  & \textsf{{[}6, 1{]}}, \textsf{{[}12, 4{]}}, \textsf{{[}16, 13{]}}, \textsf{{[}48, 39{]}}  \\
\textsf{{[}160, 199{]}} & $(( C_2  \times   Q_8 ) :  C_2) :  C_5$ & $\times$ & $3$ & $\infty$ & $1 {\times} M_2(\mathbb{H}_2)$  &  \\
\textsf{{[}192, 989{]}} & $( {\rm SL} (2,3) :  C_4) :  C_2$ & $\checkmark$ & $4$ & $\infty$ & $3 {\times} M_2(\mathbb{Q})$, $1 {\times} M_2(\mathbb{Q}(\sqrt{-3}))$, $2 {\times} M_2(\mathbb{H}_2)$  & \textsf{{[}6, 1{]}}, \textsf{{[}8, 3{]}}, \textsf{{[}12, 4{]}}, \textsf{{[}24, 8{]}}  \\
\textsf{{[}240, 89{]}} & $  {\rm SL} (2,5) .  C_2$ & $\times$ & $\infty$ &  & $1 {\times} M_2(\mathbb{H}_5)$  &  \\
\textsf{{[}240, 90{]}} & $ {\rm SL} (2,5) :  C_2$ & $\checkmark$ & $\infty$ &  & $1 {\times} M_2(\mathbb{H}_5)$  &  \\
\textsf{{[}288, 389{]}} & $( C_3  \times   C_3) : (( C_4  \times   C_4) :  C_2)$ & $\checkmark$ & $3$ & $\infty$ & $2 {\times} M_2(\mathbb{Q})$, $2 {\times} M_2(\mathbb{Q}(i))$, $2 {\times} M_2(\mathbb{H}_3)$  & \textsf{{[}8, 3{]}}, \textsf{{[}32, 11{]}}  \\
\textsf{{[}320, 1581{]}} & $ ((( C_2  \times   Q_8 ) :  C_2) :  C_5) .  C_2$ & $\times$ & $4$ & $\infty$ & $2 {\times} M_2(\mathbb{H}_2)$  &  \\
\textsf{{[}384, 618{]}} & $(( Q_8   \times   Q_8 ) :  C_2) :  C_3$ & $\checkmark$ & $3$ & $\infty$ & $1 {\times} M_2(\mathbb{H}_2)$  &  \\
\textsf{{[}384, 18130{]}} & $(( Q_8   \times   Q_8 ) :  C_3) :  C_2$ & $\checkmark$ & $4$ & $\infty$ & $1 {\times} M_2(\mathbb{Q})$, $1 {\times} M_2(\mathbb{H}_2)$  & \textsf{{[}6, 1{]}}  \\
\textsf{{[}720, 409{]}} & $ {\rm SL} (2,9)$ & $\times$ & $\infty$ &  & $2 {\times} M_2(\mathbb{H}_3)$  &  \\
\textsf{{[}1152, 155468{]}} & $((( Q_8   \times   Q_8 ) :  C_3) :  C_2) :  C_3$ & $\checkmark$ & $4$ & $\infty$ & $1 {\times} M_2(\mathbb{Q})$, $1 {\times} M_2(\mathbb{Q}(\sqrt{-3}))$, $1 {\times} M_2(\mathbb{H}_2)$  & \textsf{{[}6, 1{]}}, \textsf{{[}18, 3{]}}  \\
\textsf{{[}1920, 241003{]}} & $ C_2 . (( C_2  \times   C_2  \times   C_2  \times   C_2) :  A_5 )$ & $\times$ & $\infty$ &  & $1 {\times} M_2(\mathbb{H}_2)$  &  \\
\end{longtable}
 }
  \end{landscape}
\restoregeometry
\newpage

\clearpage
\section{Some Group Presentations}\label{Appendix_presentation}

We give the presentations of certain groups apperaing in \Cref{iff HFA} (the indices indicate their \textsc{SmallGroup ID}s). We start with the following nilpotent groups:
\begin{align*}
 G_{16,6} =  \langle\ a, b\ |\  & a^8 = b^2 = 1,\ a^b = a^5\ \rangle \cong C_8 \rtimes C_2, \\ 
 G_{16,13} =  \langle\ a, b, c\ |\  & a^4 = b^2 = c^2 = 1 = (a, b) = (a, c),\ b^c = a^2b \ \rangle \cong (C_4 \times C_2) \rtimes C_2, \\
 G_{32,50} =  \langle\ i, j, a, b\ |\  & i^4 = 1,\ i^2 = j^2,\ i^j = i^{-1},\ a^2 = 1, (i, a) = (j, a) = 1,\\  & b^2 = 1,\ i^b = i^{-1},\ j^b = j^{-1},\ a^b = i^2a \ \rangle \cong (Q_8 \times C_2) \rtimes C_2.
\end{align*}

The group $G_{16,13} \cong D_8 \Ydown C_4$ is the central product of $D_8$ and $C_4$ (central subgroups of order $2$ identified) and $G_{32,50} \cong Q_8 \Ydown D_8$ is the central product of $Q_8$ with $D_8$. The group $G_{16,6}$ is sometimes called the modular group of order $16$.

We also need the following non-nilpotent groups:
\begin{align*}
 G_{96, 202} =  \langle\ i, j, b, t, a\ |\  & i^4 = 1,\ i^2 = j^2, i^j = i^{-1},\ b^3 = 1,\ i^b = j,\ j^b = ij, \\
 & t^2 = 1,\ (i, t) = (j, t) = (b, t) = 1, \\ 
 & a^2 = 1,\ (i, a) = (j, a) = (b, a) = 1,\ t^a = i^2t \ \rangle, \\ 
 G_{240, 90} =  \langle\ x, y, z, a\ |\  & x^3 = y^5 = z^2 = 1,\ (x, z) = (y, z) = 1,\ (xy)^2 = z, \\ & a^2 = 1,\ (z, a) = 1,\ x^a = x^2,\ y^a = (xy^3)^2 \ \rangle, \\
 G_{384, 618} =  \langle\ i_1, j_1, i_2, j_2, a \ |\  & i_1^4 = 1,\ i_1^2 = j_1^2,\ i_1^{j_1} = i_1^{-1},\ i_2^4 = 1,\ i_2^2 = j_2^2,\ i_2^{j_2} = i_2^{-1},\ \\
    & (i_1, i_2) = (i_1, j_2) = (j_1, i_2) = (j_1, j_2) = 1,\ a^6 = 1, \\ 
    & i_1^a = j_2^{-1},\ j_1^a = (i_2j_2)^{-1},\ i_2^a = j_1^{-1},\ j_2^a = (i_1j_1)^{-1} \ \rangle.
\end{align*}
They have the following structures: $G_{96, 202} \cong (\SL(2, 3) \times C_2) \rtimes C_2$, $G_{240, 90} \cong SL(2,5) \rtimes 2 \cong 2 \cdot S_5^+$, the Schur cover of $S_5$ of plus type, and $G_{384, 618} \cong (Q_8 \times Q_8) \rtimes C_6$.

\clearpage

\bibliographystyle{plain}
\bibliography{FA}

\begin{thebibliography}{10}

\bibitem{Amitsur}
S.~A. Amitsur.
\newblock Finite subgroups of division rings.
\newblock {\em Trans. Amer. Math. Soc.}, 80:361--386, 1955.

\bibitem{Bachle}
Andreas B\"{a}chle.
\newblock Integral group rings of solvable groups with trivial central units.
\newblock {\em Forum Math.}, 30(4):845--855, 2018.

\bibitem{BMJM}
Andreas B\"achle, Mauricio Caicedo, Eric Jespers, and Suganda Maheshwary.
\newblock Global and local properties of finite groups with only finitely many
  central units in their integral group ring.
\newblock page 12 pages, 2018.
\newblock \href{https://arxiv.org/abs/1808.03546}{\nolinkurl{arXiv:1808.03546
  [math.GR]}}.

\bibitem{BCvG}
Andreas B\"{a}chle, Mauricio Caicedo, and Inneke Van~Gelder.
\newblock A classification of exceptional components in group algebras over
  abelian number fields.
\newblock {\em J. Algebra Appl.}, 15(5):1650092, 32, 2016.

\bibitem{amalgamationpaper}
Andreas B\"achle, Geoffrey Janssens, Eric Jespers, Ann Kiefer, and Doryan
  Temmerman.
\newblock Higher modular groups as amalgamated products and a dichotomy for
  integral group rings.
\newblock {\em submitted}, 2018.
\newblock \href{https://arxiv.org/abs/1811.12226}{\nolinkurl{arXiv:1811.12226
  [math.GR]}}.

\bibitem{BakReh}
Anthony Bak and Ulf Rehmann.
\newblock The congruence subgroup and metaplectic problems for {${\rm
  SL}_{n\geq 2}$} of division algebras.
\newblock {\em J. Algebra}, 78(2):475--547, 1982.

\bibitem{BMP17}
G.~K. Bakshi, S.~Maheshwary, and I.~B.~S. Passi.
\newblock {Integral group rings with all central units trivial}.
\newblock {\em J. Pure Appl. Algebra}, 221(8):1955--1965, 2017.

\bibitem{Bass}
H.~Bass.
\newblock {$K$}-theory and stable algebra.
\newblock {\em Inst. Hautes \'{E}tudes Sci. Publ. Math.}, (22):5--60, 1964.

\bibitem{ValetteBook}
Bachir Bekka, Pierre de~la Harpe, and Alain Valette.
\newblock {\em Kazhdan's property ({T})}, volume~11 of {\em New Mathematical
  Monographs}.
\newblock Cambridge University Press, Cambridge, 2008.

\bibitem{BrauerSurvey}
Richard Brauer.
\newblock Representations of finite groups.
\newblock In {\em Lectures on {M}odern {M}athematics, {V}ol. {I}}, pages
  133--175. Wiley, New York, 1963.

\bibitem{BriHae}
Martin~R. Bridson and Andr\'{e} Haefliger.
\newblock {\em Metric spaces of non-positive curvature}, volume 319 of {\em
  Grundlehren der Mathematischen Wissenschaften [Fundamental Principles of
  Mathematical Sciences]}.
\newblock Springer-Verlag, Berlin, 1999.

\bibitem{CeChLez}
Jean-Paul Cerri, J\'{e}r\^{o}me Chaubert, and Pierre Lezowski.
\newblock Euclidean totally definite quaternion fields over the rational field
  and over quadratic number fields.
\newblock {\em Int. J. Number Theory}, 9(3):653--673, 2013.

\bibitem{Chiswell}
Ian Chiswell.
\newblock {\em Introduction to {$\Lambda$}-trees}.
\newblock World Scientific Publishing Co., Inc., River Edge, NJ, 2001.

\bibitem{Cohn1}
P.~M. Cohn.
\newblock On the structure of the {${\rm GL}_{2}$} of a ring.
\newblock {\em Inst. Hautes \'Etudes Sci. Publ. Math.}, (30):5--53, 1966.

\bibitem{Cohn2}
P.~M. Cohn.
\newblock A presentation of {${\rm SL}_{2}$} for {E}uclidean imaginary
  quadratic number fields.
\newblock {\em Mathematika}, 15:156--163, 1968.

\bibitem{CorKar}
Yves Cornulier and Aditi Kar.
\newblock On property ({FA}) for wreath products.
\newblock {\em J. Group Theory}, 14(1), 2011.

\bibitem{CulMor}
Marc Culler and John~W. Morgan.
\newblock Group actions on {${\bf R}$}-trees.
\newblock {\em Proc. London Math. Soc. (3)}, 55(3):571--604, 1987.

\bibitem{CulVog}
Marc Culler and Karen Vogtmann.
\newblock A group-theoretic criterion for property {${\rm FA}$}.
\newblock {\em Proc. Amer. Math. Soc.}, 124(3):677--683, 1996.

\bibitem{deCornulier}
Yves de~Cornulier.
\newblock Infinite conjugacy classes in groups acting on trees.
\newblock {\em Groups Geom. Dyn.}, 3(2):267--277, 2009.

\bibitem{dlHV}
Pierre de~la Harpe and Alain Valette.
\newblock La propri\'et\'e {$(T)$} de {K}azhdan pour les groupes localement
  compacts (avec un appendice de {M}arc {B}urger).
\newblock {\em Ast\'erisque}, (175):158, 1989.

\bibitem{Dennis}
R.~Keith Dennis.
\newblock The {$GE_{2}$} property for discrete subrings of {${\bf C}$}.
\newblock {\em Proc. Amer. Math. Soc.}, 50:77--82, 1975.

\bibitem{AmalgamGL2Z}
Dragomir~\v{Z}. Djokovi\'{c} and Gary~L. Miller.
\newblock Regular groups of automorphisms of cubic graphs.
\newblock {\em J. Combin. Theory Ser. B}, 29(2):195--230, 1980.

\bibitem{EKVG}
Florian Eisele, Ann Kiefer, and Inneke Van~Gelder.
\newblock Describing units of integral group rings up to commensurability.
\newblock {\em J. Pure Appl. Algebra}, 219(7):2901--2916, 2015.

\bibitem{EisMar}
Florian Eisele and Leo Margolis.
\newblock A counterexample to the first {Z}assenhaus conjecture.
\newblock {\em Adv. Math.}, 339:599--641, 2018.

\bibitem{ErsJai}
Mikhail Ershov and Andrei Jaikin-Zapirain.
\newblock Property ({T}) for noncommutative universal lattices.
\newblock {\em Invent. Math.}, 179(2):303--347, 2010.

\bibitem{Farb}
Benson Farb.
\newblock Group actions and {H}elly's theorem.
\newblock {\em Adv. Math.}, 222(5):1574--1588, 2009.

\bibitem{FineBook}
Benjamin Fine.
\newblock {\em Algebraic theory of the {B}ianchi groups}, volume 129 of {\em
  Monographs and Textbooks in Pure and Applied Mathematics}.
\newblock Marcel Dekker, Inc., New York, 1989.

\bibitem{Fitz}
Robert~W. Fitzgerald.
\newblock Norm {E}uclidean quaternionic orders.
\newblock {\em Integers}, 12(2):197--208, 2012.

\bibitem{FroFin}
Charles Frohman and Benjamin Fine.
\newblock Some amalgam structures for {B}ianchi groups.
\newblock {\em Proc. Amer. Math. Soc.}, 102(2):221--229, 1988.

\bibitem{GAP}
The GAP~Group.
\newblock {\em {GAP -- Groups, Algorithms, and Programming, Version 4.10.0}},
  2018.

\bibitem{GonPas}
J.~Z. Gon\c{c}alves and D.~S. Passman.
\newblock Embedding free products in the unit group of an integral group ring.
\newblock {\em Arch. Math. (Basel)}, 82(2):97--102, 2004.

\bibitem{delRioGonc}
Jairo~Z. Gon\c{c}alves and \'{A}ngel Del~R\'{i}o.
\newblock A survey on free subgroups in the group of units of group rings.
\newblock {\em J. Algebra Appl.}, 12(6):1350004, 28, 2013.

\bibitem{GueHigWei}
Erik Guentner, Nigel Higson, and Shmuel Weinberger.
\newblock The {N}ovikov conjecture for linear groups.
\newblock {\em Publ. Math. Inst. Hautes \'{E}tudes Sci.}, (101):243--268, 2005.

\bibitem{HardyWright}
G.~H. Hardy and E.~M. Wright.
\newblock {\em An introduction to the theory of numbers}.
\newblock Oxford University Press, Oxford, sixth edition, 2008.
\newblock Revised by D. R. Heath-Brown and J. H. Silverman, With a foreword by
  Andrew Wiles.

\bibitem{Hertweck}
Martin Hertweck.
\newblock A counterexample to the isomorphism problem for integral group rings.
\newblock {\em Ann. of Math. (2)}, 154(1):115--138, 2001.

\bibitem{Higman}
Graham Higman.
\newblock The units of group-rings.
\newblock {\em Proc. London Math. Soc. (2)}, 46:231--248, 1940.

\bibitem{IN12}
I.M. Isaacs and G.~Navarro.
\newblock Sylow 2-subgroups of rational solvable groups.
\newblock {\em Math. Z.}, 272(3-4):937--945, 2012.

\bibitem{Jacobson}
Nathan Jacobson.
\newblock {\em Basic algebra. {II}}.
\newblock W. H. Freeman and Company, New York, second edition, 1989.

\bibitem{JanJesTem}
Geoffrey Janssens, Eric Jespers, and Doryan Temmerman.
\newblock Free products in the unit group of the integral group ring of a
  finite group.
\newblock {\em Proc. Amer. Math. Soc.}, 145(7):2771--2783, 2017.

\bibitem{JespersFree}
Eric Jespers.
\newblock Free normal complements and the unit group of integral group rings.
\newblock {\em Proc. Amer. Math. Soc.}, 122(1):59--66, 1994.

\bibitem{JesRioCrelle}
Eric Jespers and Angel del R\'{\i}o.
\newblock A structure theorem for the unit group of the integral group ring of
  some finite groups.
\newblock {\em J. Reine Angew. Math.}, 521:99--117, 2000.

\bibitem{EricAngel1}
Eric Jespers and \'{A}ngel del R\'{\i}o.
\newblock {\em Group ring groups. {V}ol. 1. {O}rders and generic constructions
  of units}.
\newblock De Gruyter Graduate. De Gruyter, Berlin, 2016.

\bibitem{EricAngel2}
Eric Jespers and \'Angel del R\'{\i}o.
\newblock {\em Group ring groups. {V}ol. 2. {O}rders and generic constructions
  of units}.
\newblock De Gruyter Graduate. De Gruyter, Berlin, 2016.

\bibitem{JesLea}
Eric Jespers and Guilherme Leal.
\newblock Describing units of integral group rings of some {$2$}-groups.
\newblock {\em Comm. Algebra}, 19(6):1809--1827, 1991.

\bibitem{JeOlvGdR}
Eric Jespers, Gabriela Olteanu, \'{A}ngel del R\'{i}o, and Inneke Van~Gelder.
\newblock Central units of integral group rings.
\newblock {\em Proc. Amer. Math. Soc.}, 142(7):2193--2209, 2014.

\bibitem{JPdRRZ}
Eric Jespers, Antonio Pita, \'{A}ngel del R\'{\i}o, Manuel Ruiz, and Pavel
  Zalesskii.
\newblock Groups of units of integral group rings commensurable with direct
  products of free-by-free groups.
\newblock {\em Adv. Math.}, 212(2):692--722, 2007.

\bibitem{KarSol}
A.~Karrass and D.~Solitar.
\newblock The subgroups of a free product of two groups with an amalgamated
  subgroup.
\newblock {\em Trans. Amer. Math. Soc.}, 150:227--255, 1970.

\bibitem{KimmerleSurvey}
Wolfgang Kimmerle.
\newblock Unit groups of integral group rings: old and new.
\newblock {\em Jahresber. Dtsch. Math.-Ver.}, 115(2):101--112, 2013.

\bibitem{KleinertSurvey}
Ernst Kleinert.
\newblock Units of classical orders: a survey.
\newblock {\em Enseign. Math. (2)}, 40(3-4):205--248, 1994.

\bibitem{Kle84}
D.~Kletzing.
\newblock {\em Structure and representations of {${\bf Q}$}-groups}, volume
  1084 of {\em Lecture Notes in Mathematics}.
\newblock Springer-Verlag, Berlin, 1984.

\bibitem{Lubot96}
Alexander Lubotzky.
\newblock Free quotients and the first {B}etti number of some hyperbolic
  manifolds.
\newblock {\em Transform. Groups}, 1(1-2):71--82, 1996.

\bibitem{MarSeh}
Zbigniew~S. Marciniak and Sudarshan~K. Sehgal.
\newblock Units in group rings and geometry.
\newblock In {\em Methods in ring theory ({L}evico {T}erme, 1997)}, volume 198
  of {\em Lecture Notes in Pure and Appl. Math.}, pages 185--198. Dekker, New
  York, 1998.

\bibitem{MarRioSurvey}
Leo Margolis and \'{A}ngel del R\'{\i}o.
\newblock Finite subgroups of group rings: a survey.
\newblock {\em Adv. Group Theory Appl.}, 8:1--37, 2019.

\bibitem{MargulisBook}
G.~A. Margulis.
\newblock {\em Discrete subgroups of semisimple {L}ie groups}, volume~17 of
  {\em Ergebnisse der Mathematik und ihrer Grenzgebiete (3) [Results in
  Mathematics and Related Areas (3)]}.
\newblock Springer-Verlag, Berlin, 1991.

\bibitem{Minasyan}
Ashot Minasyan.
\newblock New examples of groups acting on real trees.
\newblock {\em J. Topol.}, 9(1):192--214, 2016.

\bibitem{Nica}
Bogdan Nica.
\newblock The unreasonable slightness of {${\rm E}_2$} over imaginary quadratic
  rings.
\newblock {\em Amer. Math. Monthly}, 118(5):455--462, 2011.

\bibitem{Reiner}
I.~Reiner.
\newblock {\em Maximal orders}.
\newblock Academic Press [A subsidiary of Harcourt Brace Jovanovich,
  Publishers], London-New York, 1975.
\newblock London Mathematical Society Monographs, No. 5.

\bibitem{RoggScott}
Klaus Roggenkamp and Leonard Scott.
\newblock Isomorphisms of {$p$}-adic group rings.
\newblock {\em Ann. of Math. (2)}, 126(3):593--647, 1987.

\bibitem{Rosenberg}
Jonathan Rosenberg.
\newblock {\em Algebraic {$K$}-theory and its applications}, volume 147 of {\em
  Graduate Texts in Mathematics}.
\newblock Springer-Verlag, New York, 1994.

\bibitem{Scott}
L.~L. Scott.
\newblock On a conjecture of {Z}assenhaus, and beyond.
\newblock In {\em Proceedings of the {I}nternational {C}onference on {A}lgebra,
  {P}art 1 ({N}ovosibirsk, 1989)}, volume 131 of {\em Contemp. Math.}, pages
  325--343. Amer. Math. Soc., Providence, RI, 1992.

\bibitem{SehgalBook93}
S.~K. Sehgal.
\newblock {\em Units in integral group rings}, volume~69 of {\em Pitman
  Monographs and Surveys in Pure and Applied Mathematics}.
\newblock Longman Scientific \& Technical, Harlow; copublished in the United
  States with John Wiley \& Sons, Inc., New York, 1993.
\newblock With an appendix by Al Weiss.

\bibitem{SehgalSurvey03}
S.~K. Sehgal.
\newblock Group rings.
\newblock In {\em Handbook of algebra, {V}ol. 3}, volume~3 of {\em Handb.
  Algebr.}, pages 455--541. Elsevier/North-Holland, Amsterdam, 2003.

\bibitem{Serre}
Jean-Pierre Serre.
\newblock {\em Trees}.
\newblock Springer Monographs in Mathematics. Springer-Verlag, Berlin, 2003.
\newblock Translated from the French original by John Stillwell, Corrected 2nd
  printing of the 1980 English translation.

\bibitem{ShiWeh}
M.~Shirvani and B.~A.~F. Wehrfritz.
\newblock {\em Skew linear groups}, volume 118 of {\em London Mathematical
  Society Lecture Note Series}.
\newblock Cambridge University Press, Cambridge, 1986.

\bibitem{SpringerBook}
T.~A. Springer.
\newblock {\em Linear algebraic groups}, volume~9 of {\em Progress in
  Mathematics}.
\newblock Birkh\"{a}user Boston, Inc., Boston, MA, second edition, 1998.

\bibitem{Doryanthesis}
Doryan Temmerman.
\newblock {\em Fixed point properties for low rank linear groups over orders
  and applications to integral group rings}.
\newblock PhD thesis, Vrije Universiteit Brussel, 2019.
\newblock
  \url{http://homepages.vub.ac.be/~dtemmerm/publications/PhD%20Thesis.pdf}.

\bibitem{Tits}
Jacques Tits.
\newblock Syst\`emes g\'{e}n\'{e}rateurs de groupes de congruence.
\newblock {\em C. R. Acad. Sci. Paris S\'{e}r. A-B}, 283(9):Ai, A693--A695,
  1976.

\bibitem{Trefethen}
Stephen Trefethen.
\newblock Non-{A}belian composition factors of finite groups with the
  {CUT}-property.
\newblock {\em J. Algebra}, 522:236--242, 2019.

\bibitem{Voight}
John Voight.
\newblock {\em Quaternion algebras, v.0.9.13}.
\newblock June 10, 2018.
\newblock \url{https://math.dartmouth.edu/~jvoight/
  quat/quat-book-v0.9.13.pdf}.

\bibitem{WeissPGroup}
Alfred Weiss.
\newblock Rigidity of {$p$}-adic {$p$}-torsion.
\newblock {\em Ann. of Math. (2)}, 127(2):317--332, 1988.

\bibitem{WeissNilpotent}
Alfred Weiss.
\newblock Torsion units in integral group rings.
\newblock {\em J. Reine Angew. Math.}, 415:175--187, 1991.

\bibitem{Ye}
Shengkui Ye.
\newblock Low-dimensional representations of matrix groups and group actions on
  {$\rm CAT(0)$} spaces and manifolds.
\newblock {\em J. Algebra}, 409:219--243, 2014.

\end{thebibliography}

\end{document}